\documentclass{amsart}
\usepackage{amsmath}
\usepackage{amssymb}
\usepackage[all]{xy}

\theoremstyle{plain}
\newtheorem{thm}{Theorem}[section]
\newtheorem{thm*}{Theorem}[section]
\newtheorem{cor}[thm]{Corollary}
\newtheorem{prop}[thm]{Proposition}
\newtheorem{lemma}[thm]{Lemma}

\theoremstyle{definition}
\newtheorem{defn}[thm]{Definition}
\newtheorem{remark}[thm]{Remark}
\newtheorem{ex}[thm]{Example}
\newtheorem{notation}[thm]{Notation}

\newenvironment{changemargin}[2]{%
  \begin{list}{}{%
    \setlength{\topsep}{0pt}%
    \setlength{\leftmargin}{#1}%
    \setlength{\rightmargin}{#2}%
    \setlength{\listparindent}{\parindent}%
    \setlength{\itemindent}{\parindent}%
    \setlength{\parsep}{\parskip}%
  }%
  \item[]}{\end{list}}

\numberwithin{equation}{thm}

\newcommand{\cM}{\mathcal M}
\newcommand{\cN}{\mathcal N}
\newcommand{\cU}{\mathcal U}
\newcommand{\cH}{\mathcal H}

\newcommand{\cu}{\mathfrak u}

\def\JType{\operatorname{JType}\nolimits}
\def\Spec{\operatorname{Spec}\nolimits}
\def\Ker{\operatorname{Ker}\nolimits}

\def\Coker{\operatorname{Coker}\nolimits}
\def\rk{\operatorname{Rk}\nolimits}
\def\Im{\operatorname{Im}\nolimits}
\def\Proj{\operatorname{Proj}\nolimits}
\def\Span{\operatorname{Span}\nolimits}
\def\dim{\operatorname{dim}\nolimits}

\def\Ad{\operatorname{Ad}\nolimits}

\def\res{\operatorname{res}\nolimits}

\newcommand{\bG}{\mathbb G}
\newcommand{\cO}{\mathcal O}
\newcommand{\bA}{\mathbb A}

\newcommand{\bP}{\mathbb P}
\newcommand{\bZ}{\mathbb Z}
\newcommand{\bF}{\mathbb F}
\newcommand{\cp}{\mathcal P}

\newcommand{\cE}{\mathcal E}
\newcommand{\cX}{\mathcal X}
\newcommand{\cY}{\mathcal Y}

\newcommand{\m}{\mathfrak m}
\newcommand{\fG}{\mathfrak G}

\newcommand{\fg}{\mathfrak g}

\newcommand{\gl} {\mathfrak {gl}}

\newcommand{\p}{\mathfrak p}

\newcommand{\ol}{\overline}
\newcommand{\ul}{\underline}

\def\rk{\operatorname{rk}\nolimits}

\def\Supp{\operatorname{Supp}\nolimits}

\def\id{\operatorname{id}\nolimits}
\def\Spec{\operatorname{Spec}\nolimits}
\def\sl2{\operatorname{SL_{2(2)}}\nolimits}
\def\Ga2{\operatorname{\mathbb G_{\rm a(2)}}\nolimits}
\def\SL{\operatorname{SL}\nolimits}
\def\GL{\operatorname{GL}\nolimits}

\def\bH{\operatorname{H^\bu}\nolimits}
\def\HHH{\operatorname{H}\nolimits}
\def\Ext{\operatorname{Ext}\nolimits}
\def\End{\operatorname{End}\nolimits}
\def\Hom{\operatorname{Hom}\nolimits}

\def\ann{\operatorname{ann}\nolimits}

\def\PG{\operatorname{\mathbb P(G)}}
\def\proj{\operatorname{proj}\nolimits}

\def\St{\operatorname{St}\nolimits}

\newcommand{\wt}{\widetilde}

\newcommand{\Gar}{\mathbb G_{a(r)}}
\newcommand{\GarK}{\mathbb G_{a(r),K}}
\newcommand{\Gas}{\mathbb G_{a(s)}}

\newcommand{\Z}{\mathbb Z}

\newcommand{\E}{\mathcal E}

\newcommand{\bu}{\bullet}

\date\today

\begin{document}

 \title{Constructions for  infinitesimal group schemes}

\author[ Eric M. Friedlander and Julia Pevtsova]
{Eric M. Friedlander$^*$ and 
Julia Pevtsova$^{**}$}

\address {Department of Mathematics, University of Southern California,
Los Angeles, CA  90089-2532}
\email{eric@math.northwestern.edu, \ efriedla@usc.edu}

\address {Department of Mathematics, University of Washington, 
Seattle, WA  98195-4350}
\email{julia@math.washington.edu}

\thanks{$^*$  partially supported by the NSF grant  \#0300525}
\thanks{$^{**}$  partially supported by the NSF grant  \#0629156}


\begin{abstract}

Let $G$ be an infinitesimal group scheme over a field $k$ of characteristic $p >0$.  We introduce 
the global $p$-nilpotent operator $\Theta_G: k[G] \to k[V(G)]$, where $V(G )$ is the 
scheme which represents 1-parameter subgroups of $G$.  This operator $\Theta_G$ applied to $M$ 
encodes the local Jordan type of $M$, and leads to computational insights into 
the representation theory of $G$.    For certain $kG$-modules $M$ 
(including those of constant Jordan type), 
we employ $\Theta_G$ to associate various algebraic vector bundles on $\bP(G)$, 
the projectivization of $V(G)$.  
These vector bundles not only distinguish certain representations with the same 
local Jordan type, but also provide a 
method of constructing algebraic vector bundles on $\bP(G)$.

\sloppy{

}

\end{abstract}

\maketitle

\section{Introduction}

In \cite{SFB1}, \cite{SFB2}, the foundations of a theory of support varieties were established 
for an infinitesimal group scheme $G$ over a field $k$ of characteristic $p > 0$, extending
earlier work for elementary abelian $p$-groups
and $p$-restricted finite dimensional Lie algebras (\cite{C}, \cite{FPa}).  
These foundations
relied upon cohomological calculations to identify cohomological support varieties and 
introduced 1-parameter subgroups to provide an alternate, representation-theoretic 
perspective.  
In contrast to the situation for finite groups, the cohomological variety for $G$ infinitesimal
is of considerable geometric complexity; partly for this reason, computations of explicit
examples are challenging.  
In this present paper, we build upon this earlier work as well as more recent work
of the authors (\cite{FP1}, \cite{FP2}) to initiate a more
detailed investigation of representations of $G$.  Although representations of 
infinitesimal group schemes  are less familiar than 
representations of finite groups, their importance is evident: for example, 
the representation theories of the family 
of all infinitesimal kernels $G = \fG_{(r)}$ of a smooth connected algebraic group $\fG$
is essentially equivalent to the rational representation theory of $\fG$.

An important structure associated to an infinitesimal group scheme $G$ of height $\leq r$
is the 
scheme $V(G)$ of one-parameter subgroups $\Gar \to G$ (see \cite{SFB1}). 
In this paper  we observe that the representability of $V(G)$ 
leads to a $p$-nilpotent element $\Theta_G$ in $kG \otimes k[V(G)]$, 
where $kG$ is the group algebra of $G$.    For any $kG$-module $M$, $\Theta_G$
determines a  {\it  global $p$-nilpotent operator} on $M$.  The operator 
$\Theta_G$ encodes the local Jordan type of a $kG$-module $M$ which in
turn determines the support variety of $M$.  Even though the scheme $V(G)$ was 
generalized to all finite group schemes in \cite{FP1}, \cite{FP2} via the notion of $\pi$-points, 
the construction of $\Theta_G$ does not appear to extend to arbitrary finite groups.

The homogeneity of our global operator  $\Theta_G$  enables us to associate to a $kG$-module 
$M$ of constant Jordan type a collection of
vector bundles on $\bP(G) = \Proj k[V(G)]$, which we view as a family of global invariants
of $M$.  Certain modules with the same local Jordan type can be distinguished by these global invariants.

The vector bundles that we investigate are constructed directly and explicitly from $kG$-modules. Hence, 
we offer an important new method to create interesting examples of algebraic vector bundles on varieties 
of the form $\bP(G)$.  Varieties of this form include projective spaces, weighted projective spaces, and
various singular varieties associated to algebraic groups.  For example, for $G = \GL_{n(r)}$,
$\bP(G)$ is the projectivization of the variety of $r$-tuples of pairwise commuting 
$p$-nilpotent matrices.  We expect that our technique of constructing algebraic 
vector bundles on such singular, geometrically interesting varieties will lead to insights
into their algebraic K-theory. 

The reader might find it instructive to contrast our 
use of representations of $G$ to construct  vector bundles on $\bP(G)$
with the Borel-Weil construction which employs bundles on flag varieties
for an algebraic group $\fG$ to construct representations of $\fG$.
Our construction of vector bundles plays a role in the
forthcoming papers by the first author, Jon Carlson, and Andrei Suslin \cite{CFS} and 
by the second author and David Benson \cite{BP}.

In this paper, we also attempt to address the lack of specific examples in the representation
theory of infinitesimal groups schemes (other than those of height 1).  Throughout this
paper, we work with the following four 
fundamental,  yet concrete, classes of examples.  

\vspace{0.1in}
i.) $p$-restricted Lie algebras, $\fg$; 

ii.) infinitesimal additive group schemes, $\mathbb G_{a(r)}$;

iii.) infinitesimal general 
linear groups, $\GL_{n(r)}$; 

iv.) the height 2, infinitesimal special linear group, 
$\SL_{2(2)}$.  

\vspace{0.1in}
\noindent
We consistently endeavor  to make our general results
more  concrete by applying them to our examples.

\vspace{0.1in}
In Section 1, we recall some of the highlights from \cite{SFB1}, \cite{SFB2} 
concerning the cohomology and theory
of supports of finite dimensional $kG$-modules for an infinitesimal group 
scheme $G$.  A key result summarized in Theorem \ref{iso} is the close relationship
between the spectrum $\Spec \HHH^\bu(G,k)$ of the cohomology of $G$ and the 
scheme $V(G)$ representing (infinitesimal) 1-parameter subgroups 
of an infinitesimal group scheme $G$.  

In the second section, we define
the global $p$-nilpotent operator 
$\Theta_G: k[G] \longrightarrow k[V(G)]$
for an infinitesimal group scheme $G$.   
For any finite dimensional $kG$-module $M$, 
$\Theta_G$ determines a $p$-nilpotent endomorphism
of the free $k[V(G)]$-module  $M\otimes k[V(G)]$. We establish in 
Proposition \ref{homog} that $\Theta_G$ is homogeneous,
where  $k[V(G)]$ is equipped with its natural grading.  We also verify
that $\Theta_G$ is natural with respect to change of group.

In the third section, we verify that specializations $\theta_v$ of $\Theta_G$
at points $v \in V(G)$  determine the local Jordan type of a finite dimensional 
$kG$-module $M$.   Theorem \ref{matrix} can be
viewed as providing an algorithm for obtaining the local Jordan type
 in terms of the representation $G \to \GL_N$ defining the $kG$-module $M$.
We utilize $\Theta_G$ and its specializations  to establish constraints 
for a $kG$-module $M$ to be of constant rank (and thus of constant Jordan type). 
We also  establish the
relationship  between the local Jordan type of a   module and 
its  Frobenius twists.   

We envision that some of our constructions for infinitesimal group
schemes may lead to analogues for a general finite group
scheme.  With this in mind, we begin the fourth section with a dictionary
between 1-parameter subgroups for infinitesimal group schemes and $\pi$-points
for general finite group schemes.  
Given a finite dimensional $kG$-module $M$, we consider the projectivization
of the operator $\Theta_G$, 
$$
\wt\Theta_G: M \otimes \cO_{\PG} \to 
 M \otimes \cO_{\PG}(p^{r-1}),
$$
a $p$-nilpotent operator on the free, coherent sheaf $ M \otimes \cO_{\PG} $
on $\PG$.
We verify in Proposition \ref{sec} that $\wt\Theta_G$ determines via base
change the local Jordan type of a $kG$-module $M$ at any 
1-parameter subgroup $\mu_v: \bG_{a(r),k(v)} \to G_{k(v)}$.
Theorem \ref{equiv} shows that the condition
that $M$ be of constant $j$-rank is equivalent to the condition that the
coherent sheaf $\Im \wt\Theta_G^j$ be locally free.

In the fifth section, we initiate an investigation of the algebraic vector bundles
$\Ker\{ \wt \Theta_G^j,\cM\}$,  $\Im\{ \wt \Theta_G^j,\cM\}$  
on $\PG$ associated to $kG$-modules 
of constant Jordan type and more generally of constant $j$-rank.  We give examples
of such $kG$-modules in each of our four representative examples, and investigate the 
associated vector bundles.   As we see, 
taking kernels of powers of the global $p$-nilpotent power operator
sends modules of constant Jordan type to vector bundles.   We also obtain vector bundles
by  taking kernels modulo images (as inspired by a 
construction of M. Duflo and V. Serganova for Lie superalgebras in \cite{DS}).
As an application, we prove in Proposition \ref{endo} a geometric characterization 
of endotrivial modules.

Finally, in the last section, we  provide numerous explicit examples.  These include   
the infinitesimal group scheme
$G = \bG_{a(1)} \times \bG_{a(1)}$, which has the same representation theory as
 the elementary abelian $p$-group  $\Z/p \times \Z/p$, as well as 
 the  first Frobenius kernel  of the reductive group $\SL_2$.  One intriguing comparison which we
 investigate in particularly simple examples is the relationship between the Grothendieck
 group of projective $kG$-modules and the Grothendieck group of
 algebraic vector bundles on $\bP(G)$.  Combined with our explicit calculations, Proposition
\ref{r-times} can be viewed both as a means to 
distinguish certain non-isomorphic projective $kG$-modules and as a means
of constructing non-isomorphic algebraic vector bundles on $\PG$.

Throughout, $k$ will denote an arbitrary field of characteristic $p > 0$.  Unless
explicit mention is made to the contrary,
$G$ will denote an infinitesimal group scheme over $k$.  If $M$ is a $kG$-module
and $K/k$ is a field extension, then we denote by $M_K$ the $KG$-module 
obtained by base extension.

The authors  are grateful to the University of Bielefeld and to MSRI for  their hospitality.   
We thank Dan Grayson for pointing out the occurrence of weighted projective
spaces and  Paul Smith for numerous useful conversations.


\section{Infinitesimal group schemes}

The purpose of this first section is to summarize the important role played 
by (infinitesimal) 1-parameter subgroups  of an infinitesimal group scheme 
as presented in \cite{SFB1}.  The four representative examples of Example \ref{four},
($\ul \fg, ~ \Gar, ~ \GL_{n(r)}, ~ \SL_{2(2)}$), and their associated schemes
of 1-parameter subgroups discussed in Example \ref{var} will serve as 
explicit models  to which we will frequently return.

\begin{defn}
A finite group scheme $G$ over $k$ is a group scheme over $k$ whose
coordinate algebra $k[G]$ is finite dimensional over $k$.

Equivalently, $G$ is a functor from commutative $k$-algebras to groups,
$R \mapsto G(R)$, represented by a finite dimensional commutative $k$-algebra,
the coordinate algebra $k[G]$ of $G$.

Associated to $G$, we have its group algebra $kG = \Hom_k(k[G],k)$; 
more generally, for any
commutative $k$-algebra $R$, we have the $R$-group algebra 
$RG = \Hom_k(k[G],R)$.
\end{defn}

\begin{notation}  If $f: G\to H$ is a map of finite group schemes, we denote by 
$$
f^*:k[H] \to k[G]\quad  \text{ and } \quad f_*:kG \to kH
$$ 
the induced maps on coordinate and group algebras respectively.
\end{notation}

Observe that the $R$-group algebra of $G$ consists of all $k$-linear
homomorphisms, whereas $G(R) = \Hom_{k-alg}(k[G],R)$ is
the subgroup of $RG^\times$ consisting of $k$-algebra homomorphisms.

\begin{defn}
Let $G$ be a finite group scheme over $k$ and $M$ a $k$-vector space.
Then a (left) $kG$-module structure on $M$ is given by one of the following
equivalent sets of data (see, for example, \cite{Jan}):  
\begin{itemize}
\item The structure $M \to M \otimes k[G]$ of a (right) $k[G]$-comodule on $M$.
\item The structure $kG \otimes M \to M$ of a $kG$-module on $M$.
\item A functorial (with respect to $R$) group action 
$G(R) \times (R\otimes M) \to (R\otimes M)$.
\end{itemize}
\end{defn}

For most of this paper we shall restrict our consideration to infinitesimal 
group schemes, a special class of finite group schemes which we now define.

\begin{defn}
An infinitesimal group scheme $G$ (over $k$) of height $\leq r$ is a finite group
scheme whose coordinate algebra $k[G]$  is a local algebra
with maximal ideal $\m$ such that $x^{p^r} = 0$ for all $x \in \m$.
\end{defn}

\begin{ex} 
\label{four}
We shall frequently consider the following four examples.

\vspace{0.1in}
(1) A finite dimensional $p$-restricted Lie algebra $\fg$ corresponds naturally with
a height 1 infinitesimal group scheme which we denote $\ul \fg$ (\cite[I.8.5]{Jan}).  The group
algebra of $\ul \fg$ is the restricted enveloping algebra $\cu(\fg)$ of $\fg$.  If $\fg$ is
the Lie algebra of a group scheme $\fG$, then the coordinate algebra of $\ul \fg$
is given by $k[\fG]/(x^p,x \in \m)$, where $\m$ is the maximal ideal of
 $k[\fG]$ at the identity of $\fG$.

\vspace{0.1in}
(2)
 Let $\bG_a$ denote the additive group, so that $k[\bG_a] = k[T]$ with 
 coproduct defined by $\nabla(T) = T\otimes 1 + 1 \otimes T$.  As a functor,
 $\bG_a: (comm \ k-alg) \to (grps)$ sends an algebra $R$ to its underlying
 abelian group.  For any $ r \geq 1$, we consider the $\rm r^{th}$ Frobenius
 kernel of $\bG_a$, 
 $$\bG_{a(r)} \ \equiv \Ker\{ F^r: \bG_a \to \bG_a \}.$$
 Here $F: \bG_a \to \bG_a$ is the (geometric) Frobenius specified by its map
 on coordinate algebras $k[T] \to k[T]$ given as the $k$-linear map sending 
 $T$ to $T^p$.
 The coordinate algebra of $\bG_{a(r)}$ is given by $k[\bG_{a(r)}] = k[T]/T^{p^r}$,
 whereas the group algebra of $\bG_{a(r)}$ is given by
 \begin{equation} 
 \label{Gar}
 k\Gar  \simeq k[\Gar]^\# \simeq k[u_0,\ldots,u_{r-1}]/(u_0^p, \ldots, u_{r-1}^p),
 \end{equation}
  where
 $u_i$ is a linear dual to $T^{p^i}$, $ 0 \leq i \leq r-1. $

\vspace{0.1in}
(3) Let $\GL_n$ denote the general linear group, the representable functor 
sending a commutative algebra $R$ to the group $\GL_n(R)$.  For any $ r \geq 1$, 
we consider the 
$\rm r^{th}$ Frobenius kernel of  $\GL_n$,  
$$\GL_{n(r)} \ \equiv \Ker\{ F^r:\GL_n \to \GL_n \},$$
where the geometric Frobenius $$F: \GL_n(R) \to \GL_n(R)$$ is defined by raising  
each matrix entry to the $\rm p^{th}$ power. 
The coordinate algebra of $\GL_{n(r)}$ is given by
$$k[\GL_{n(r)}] \ = \ \frac{k[X_{i,j}]}{(X_{i,j}^{p^r} - \delta_{i,j})}_{1 \leq i, j \leq n},$$
whereas the group algebra of $\GL_{n(r)}$ is given as
$$k\GL_{n(r)} \ = \ \Hom_k(k[\GL_{n(r)}],k),$$
the $k$-space of linear functionals
$k[\GL_{n(r)}]$ to $k$.   The coproduct
$$\nabla: k[\GL_{n(r)}] \  \to\  k[\GL_{n(r)}] \otimes k[\GL_{n(r)}] $$
is given by sending $X_{i,j}$ to $\sum_k X_{ik}\otimes X_{kj}$.
 
\vspace{0.1in}
(4)
 The height 2 infinitesimal group scheme 
 $\SL_{2(2)}$ is essentially a special case of $\GL_{n(r)}$.  This is once
 again defined as the kernel of the second iterate of Frobenius
 $$\SL_{2(2)} \ \equiv \Ker\{ F^2: \SL_2 \to \SL_2 \}.$$
 The coordinate algebra of $\SL_{2(2)}$ is given by
$$k[\SL_{2(2)}] = \frac{k[X_{1,1}, X_{1,2}, X_{2,1}, X_{2,2}]}
{(X_{1,1}X_{2,2} - X_{1,2}X_{2,1} - 1,X_{i,j}^{p^2} - \delta_{i,j})}$$  
  whereas the group algebra of $\SL_{2(2)}$ is given as
 $$k\SL_{2(2)} = k\langle e,f,h,e^{(p)},f^{(p)},h^{(p)} \rangle / \langle 
 relations \rangle $$ with
 $e,\ f,\ h,\ e^{(p)},\ f^{(p)},\ h^{(p)}$ the dual basis vectors  to  
 $X_{1,2}, \ X_{2,1},\ X_{1,1}-1, 
 \ X_{1,2}^p, \ X_{2,1}^p, \ (X_{1,1}-1)^p$ respectively.
 {\sloppy
 
 }

\end{ex}

We denote by $\mathbb G_{a(r),R}$ the base extension of $\Gar$ to a commutative $k$-algebra $R$.
\begin{defn}
A (infinitesimal) 1-parameter subgroup of height $r$ of an affine group scheme $G_R$
over a commutative $k$-algebra $R$ is 
a homomorphism of $R$-group schemes $\bG_{a(r),R} \to G_R$. 
\end{defn}

We recall the description of height $r$ 1-parameter subgroups of $\GL_n$ given
in \cite{SFB1}.

\begin{prop} 
\label{coact}
\cite[1.2]{SFB1}
If $G = \GL_n$ and if $R$ is a commutative $k$-algebra, 
then  a 1-parameter subgroup of  $\GL_{n,R}$ of height $r$, 
$f: \bG_{a(r),R} \to \GL_{n,R}$,
is naturally (with respect to $R$) equivalent to a comodule map 
$$\Delta_f: R^n \to R[T]/T^{p^r} \otimes_R R^n, \quad 
\Delta_f(v) = \sum_{j=0}^{p^{r-1}} T^j \otimes \beta_j(v), \quad \beta_j \in M_n(R)$$
satisfying the constraints of being counital and coassociative.  This in turn is
equivalent to specifying an $r$-tuple of matrices
$\alpha_0 = \beta_0, \alpha_1 = \beta_p,\ldots,\alpha_{r-1} = \beta_{p^{r-1}}$ in
$M_n(R)$ such that each $\alpha_i$ has $p^{th}$ power 0 and such that the
$\alpha_i$'s pairwise commute.  The other coefficient matrices $\beta_j$
are given by the formula 
\begin{equation}
\label{beta}
\beta_j = \frac{\alpha_0^{j_0}\cdots\alpha_{r-1}^{j_{r-1}}}{(j_0)!\cdots(j_{r-1})!} \in M_n(R),
\quad j = \sum_{i=0}^{r-1} j_ip^i {\text ~ with ~ } 0\leq j_i < p.
\end{equation}
\end{prop}

As shown in \cite{SFB1}, Proposition \ref{coact} implies the following representability
of the functor of 1-parameter subgroups of height $r$.

\begin{thm}
\label{represent}
\cite[1.5]{SFB1}
For any affine group scheme $G$, the functor from commutative $k$-algebras
to sets
$$R \ \mapsto \ \Hom_{\rm grp \, sch}(\bG_{a(r),R},G_R)$$
is representable by an affine scheme $V_r(G) = \Spec k[V_r(G)]$.  Namely, this
functor is naturally isomorphic to the functor
$$R \ \mapsto \ \Hom_{\rm k-alg}(k[V_r(G)],R).$$
\end{thm}

By varying $r$, we can associate  a family  of affine schemes  to an affine group 
scheme $G$.  In the following remark we make explicit the relationship 
between  various $V_r(G)$ for the same $G$ and varying $r$'s.

\begin{remark}
\label{varying}  For $r > s \geq 1$, 
let $p_{r,s}: \Gar \to\Gas$ be the canonical projection given by the natural 
embedding of the coordinate algebras

\centerline
{
$\UseComputerModernTips
 \xymatrix{p_{r,s}^*: k[\Gas] = k[T]/T^{p^s} \ar[rr]^-{T \rightarrow 
T^{p^{r-s}}} && k[T]/T^{p^r} = k[\Gar].}$
}
\noindent
The corresponding map on group algebras 
$$
\UseComputerModernTips
 \xymatrix{k\Gar \simeq k[u_0, \ldots, u_{r-1}]/(u^p_0, \ldots, u^p_{r-1}) \ar[r]^{p_{r,s, *}} & 
k\Gas \simeq k[v_0, \ldots, v_{s-1}]/(v^p_0, \ldots, v^p_{s-1})}
$$ 
sends $\{u_0, \ldots, u_{r-s-1}\}$ to $\{ 0, \ldots, 0 \}$, and 
$\{ u_{r-s},  \ldots, u_{r-1} \}$ to $\{v_0, \ldots, v_{s-1} \}$.    

Precomposition with $p_{r,s}$ determines a canonical embedding 
of affine schemes  
$$\UseComputerModernTips
 \xymatrix{i_{s,r}: V_s(G) \ar@{^(->}[r] &V_r(G)},$$
where a one-parameter subgroup  $\mu:  \bG_{a(s),R} \to G_R$ of height $s$
is sent to the one-parameter subgroup $ \mu \circ p_{r,s} : \bG_{a(r),R} 
\to \bG_{a(s),R} \to  G_R$ of height $r$. 
The construction is transitive, that is, we have $i_{s,r} = i_{s^\prime, r} 
\circ \,i_{s, s^\prime}$ for $s \leq s^\prime \leq r$.  Hence, we have an inductive system 
$$V_1(G) \subset V_2(G) \subset \ldots \subset V_r(G) \subset \ldots$$
\end{remark}

Conversely, any one-parameter subgroup
$\bG_{a(s^\prime),R} \rightarrow G_R$ can be decomposed as
$$\UseComputerModernTips
 \xymatrix{\bG_{a(s^\prime),R} \ar[r]^-{p_{s^\prime,s}} & \bG_{a(s),R} \ar@{^(->}[r] & G_R}$$ 
for some $s \leq s^\prime$.  If $G$ is an infinitesimal group scheme of height 
$\leq r$ then,  we may choose $s\leq r$.  This justifies the following definition 

\begin{defn}
\label{limm}
Let $G$ be an infinitesimal group scheme.  Then  the 
closed immersion  $i_{r, r^\prime}: V_r(G) \hookrightarrow V_{r^\prime}(G)$ for $r^\prime >r$ 
is an isomorphism provided the height of $G$ is $\leq r$.  We denote by $V(G)$ the stable
value of $V_r(G)$,
$$V(G) \ \equiv  \varinjlim_r \ V_r(G).$$ 
\end{defn}

We next make explicit the construction of 1-parameter subgroups for $\GL_n$ as in
Proposition \ref{coact}.  
This construction can be applied to any affine group scheme 
of exponential type (see \cite[\S 1]{SFB1} and  
also \cite{Mc} for an extended list of groups  of exponential type).
We define the homomorphism $$\exp_{\ul \alpha}: \bG_{a(r),R} \to \GL_{n,R}$$ 
of $R$-group schemes corresponding to an $r$-tuple 
$\ul \alpha = (\alpha_0,\ldots,\alpha_{r-1}) \in M_n(R)^{\times r}$ 
of pairwise commuting $p$-nilpotent matrices  to be
the natural transformation of group-valued functors on commutative
$R$-algebras $S$ sending any $s\in S$ with $s^{p^r} = 0$ to  
\begin{equation}
\label{exp}
\exp(s\alpha_0)\cdot \exp(s^p \alpha_1)\cdot \cdots \cdot 
\exp(s^{p^{r-1}} \alpha_{r-1}) \in \GL_n(S),
\end{equation}
where for any $p$-nilpotent matrix $A \in \GL_n(S)$ we set
$$\exp(A) =  1 + A + \frac{A^2}{2} + \ldots + \frac{A^{p-1}}{(p-1)!}.$$

The following proposition proved in \cite{SFB1} identifies the functor of 1-parameter
subgroups in the case of infinitesimal general linear groups.

\begin{prop}
\label{tuples}
\cite[1.2]{SFB1}
The scheme of one-parameter subgroups $V_r(\GL_n)$ is isomorphic 
to the scheme of $r$-tuples of 
pairwise commuting $p$-nilpotent $n\times n$ matrices $N^{[r]}_p(gl_n)$; 
the identification is given by sending 
$\ul \alpha = (\alpha_0,\ldots,\alpha_{r-1}) \in N^{[r]}_p(gl_n)(R)$ to the  
one-parameter subgroup $\exp_{\ul \alpha}: \bG_{a(r),R} \to \GL_{n,R}$.
\end{prop}

\begin{ex} 
\label{var}
 We describe $V(G)$ in each of the four examples of Example \ref{four}.

\vspace{0.1in}
\noindent 
(1) $V(\ul \fg) \simeq N_p(\fg)$, the closed subvariety of the affine space
underlying $\fg$ consisting of $p$-nilpotent elements $x\in \fg$ (that is, $x^{[p]} = 0$).
Let  $\fg_a$ be the  Lie algebra  of the additive group $\bG_a$.  Note that $\fg_a$ is a 
one-dimensional restricted Lie algebra with trivial $p$-restriction.
Each $p$-nilpotent element
$x \in \fg_R = \fg \otimes_k R$ determines a map of $p$-restricted Lie algebras 
over $R$ where $R$ is a commutative $k$-algebra: 
$\fg_{a,R} \to \fg_R$. The corresponding map
of height 1 infinitesimal group schemes $\bG_{a(1),R} \to \ul \fg_R$
is the associated 1-parameter subgroup of $\ul \fg$.

\vspace{0.1in}
\noindent
(2)  $V(\bG_{a(r)}) \simeq \bA^r$.   The $r$-tuple
$\ul a = (a_0,\ldots,a_{r-1}) \in R^{\times r} = \bA^r(R)$ corresponds to the 1-parameter 
subgroup $\mu_{\ul a}: \bG_{a(r),R} \to \bG_{a(r),R}$ whose map on 
coordinate algebras $R[T]/T^{p^r} \to  R[T]/T^{p^r}$ sends $T$ to
$\sum_i a_iT^{p^i}$ (\cite[1.10]{SFB1}).

\vspace{0.1in}
\noindent
(3) By Proposition \ref{tuples},  $V(\GL_{n(r)}) \ = \ N_p^{[r]}(gl_n)$, 
the variety of $r$-tuples of pairwise
commuting, $p$-nilpotent $n\times n$ matrices.     The embedding 
$i_{r, r+1}: V_r(\GL_n) \simeq N_p^{[r]}(gl_n) \subset 
V_{r+1}(\GL_n)\simeq N_p^{[r+1]}(gl_n)$  
described in Remark~\ref{varying}   is given by sending an $r$-tuple 
$(\alpha_0,\ldots,\alpha_{r-1})$ 
to the  $(r+1)$-tuple $(0, \alpha_0,\ldots,\alpha_{r-1})$.

Let $X_{i,j}$ be the coordinate functions of $R[\GL_{n(r)}] \simeq 
R[X_{i,j}]/(X_{i,j}^{p^{r}} - \delta_{i,j})$.
Then $\exp_{\ul \alpha}^*: R[\GL_{n(r)}] \to R[\Gar] 
$ is given by sending $X_{i,j}$ for some $1 \leq i,j \leq n$ to the
$(i,j)$-entry of the polynomial $p_{\ul \alpha}(t)$ with matrix coefficients
whose coefficient of $t^d$ is computed as the
multiple of $s^d$ in the $(i,j)$-entry of the matrix (\ref{exp}).

Upon performing the indicated multiplication in (\ref{exp}),
the coefficient of  $p_{\ul \alpha}(t)$ multiplying  $s^{p^\ell}$ is
$\alpha_\ell$ for $0 \leq \ell < r$, whereas coefficients of 
$p_{\ul \alpha}(t)$ multiplying $s^n$ for 
$n$ not a power of $p$ are determined as in formula (\ref{beta}).  
Consequently, we conclude that $\exp_{\ul \alpha}^*(X_{i,j}) $ is a 
polynomial in $t$ whose coefficient multiplying $T^{p^\ell}$ is
$(\alpha_\ell)_{i,j}$ for $0 \leq \ell < r$.

\vspace{0.1in}
\noindent
(4) Since $\SL_{2(2)}$ is a group scheme with an embedding of 
exponential type (see \cite[1.8]{SFB1}), 
its variety admits a description  similar to the one of $\GL_{n(r)}$. 
Namely, $V(\SL_{2(2)})$ is the variety of pairs of $p$-nilpotent 
proportional  $2\times 2$ matrices $\ul \alpha = (\alpha_0, \alpha_1)$. 
This variety is given explicitly as
the affine scheme with coordinate algebra
 $$k[V(\SL_{2(2)})] =k[x_0,y_0,z_0,x_1,y_1,z_1]/
(x_iy_i - z_i^2,x_0y_1 - x_1y_1, z_0y_1 - z_1y_0,x_0z_1 - x_1z_0).$$

\noindent
We give an explicit description of the map on coordinate algebras $$\exp^*_{\ul \alpha}: 
R[\SL_{2(2)}] \to R[\Ga2] \simeq R[T]/T^{p^2}$$ induced by the  one-parameter subgroup 
$\exp_{\ul \alpha}:   \bG_{a(2),R} \to \SL_{2(2),R}$. This description follows 
immediately from the general discussion in the previous example.  
Let $\ul \alpha =(\left [\begin{array}{cc} 
c_0 & a_0 \\
b_0 & -c_0
\end{array} \right ], \left [\begin{array}{cc} 
c_1 & a_1 \\
b_1 & -c_1
\end{array} \right ] ) \in  N^{[2]}(sl_2)$.  Then $\exp^*_{\ul \alpha}$ is 
determined by the formulae
$$X_{1,1} \mapsto 1 + c_0T + c_1T^p, \quad  X_{1,2} \mapsto  a_0T  + a_1T^p$$
$$X_{2,1} \mapsto b_0T + b_1T^p, \quad  X_{2,2} \mapsto 1 - c_0T  - c_1T^p, $$
where $X_{i,j}$ are the standard polynomial generators of 
$k[\SL_{2(2)}] 
\simeq \frac{k[X_{1,1}, X_{1,2}, X_{2,1}, X_{2,2}]}{(\det-1,
X_{i,j}^{p^2} - \delta_{i,j})}.$  

\end{ex}

\vskip .2in
\begin{remark}
\label{canon}
If $k(v)$ denotes the field of definition of the point $v \in V(G)$ for an 
infinitesimal group scheme $G$ (see \cite[p.743]{SFB2} for a 
discussion of the field of definition), then we have a 
naturally associated map $\Spec k(v) \to V(G)$  and, hence, 
an associated group scheme homomorphism over $k(v)$ (for $r$ sufficiently large):
$$
\UseComputerModernTips
 \xymatrix{\mu_v: \bG_{a(r), k(v)} \ar[r] & G_{k(v)}}.
$$
Note that if $K/k$ is a  field extension and $\mu: \bG_{a(r), K} \to G_K$ is a group scheme 
homomorphism, then this data defines a point $v \in V(G)$
and a field embedding $k(v) \hookrightarrow K$ such that $\mu$ is 
obtained from $\mu_v$ via scalar extension from $k(v)$ to $K$.
\end{remark}

We next recall the {\it rank variety} and {\it cohomological support variety}
of a $kG$-module of an infinitesimal group scheme.  We denote by 
$$
\bH(G,k) = 
\begin{cases}
\HHH^*(G,k), & \text{if $p = 2$,}\\
\HHH^{\rm ev}(G,k) & \text{if $p > 2$}. \\
\end{cases}
$$

\begin{defn}
\label{suppvar}
Let $G$ be a finite group scheme and $M$ a finite dimensional $kG$-module.
We define the cohomological support variety for $M$ to be
$$|G|_M \ \equiv \ V(\ann_{\bH(G,k)}\Ext_{kG}^*(M,M)),$$
the reduced closed subscheme
of $|G| = \Spec \bH(G,k)_{\rm red}$ given as the variety of the annihilator 
ideal of $\Ext_{kG}^*(M,M)$.  
\end{defn}

The map of $R$-algebras (but not of Hopf algebras
for $r>1$), 
\begin{equation}
\label{epsilon}
\UseComputerModernTips
 \xymatrix{ \epsilon: R[u]/u^p \ar[rr]^-{u \mapsto u_{r-1}} &
& R[u_0,\ldots,u_{r-1}]/(u_i^p) \ \simeq \ R\Gar},   
\end{equation}
makes its first appearance in the following definition and will recur
throughout this paper.

\begin{defn}
\label{rankvar}
Let  $G$ be an infinitesimal group scheme  and $M$ a finite 
dimensional $kG$-module.
We define the rank variety for $M$ to be the reduced closed subscheme $V(G)_M$
whose points are given as follows:
$$V(G)_M = \{ v \in V(G) : (\mu_{v,*} \circ \epsilon)^*(M_{k(v)}) 
\text{ is not free as a }  k[u]/u^p \text{ - module } \}. $$
\end{defn}
\vspace{0.1in}

Proposition \cite[6.2]{SFB2} asserts that $V(G)_M$ is a closed subvariety of $V(G)$.   
A key result of \cite{SFB2} is the following theorem relating the scheme of
1-parameter subgroups $V(G)$ 
to the  cohomology of $G$.

\begin{thm} (\cite[5.2, 6.8, 7.5]{SFB2}) 
\label{iso}  
Let $G$ be an infinitesimal group scheme of height $\leq r$.  There  is a natural 
homomorphism of $k$-algebras
$$\psi: \bH(G,k) \to k[V(G)]$$
with nilpotent kernel whose image contains the $p^r$-th  power  of each element 
of $k[V(G)]$.  Hence, the associated morphism of schemes 
$$\Psi: V(G) \to \Spec \bH(G,k)$$
is a $p$-isogeny.  

If $M$ is a finite dimensional $kG$-module, then $\Psi$ restricts to a homeomorphism 
$$\UseComputerModernTips
 \xymatrix{\Psi_M: V(G)_M \ar[r]^-{\sim} & |G|_M.}$$
Furthermore, every closed conical subspace of $V(G)$ is of the form $V(G)_M$ for some
finite dimensional $kG$-module $M$.
\end{thm}

In the special  case of   of $G = \GL_{n(r)}$ the isogeny $\Psi$ has an explicitly 
constructed inverse.

\begin{thm} (\cite [5.2]{SFB1})
\label{gln}
There exists a  homomorphism of $k$-algebras 
$$\phi: k[V(\GL_{n(r)})] \to \bH(\GL_{n(r)},k)$$
such that $\psi \circ  \phi$ is the $r^{th}$ iterate of the $k$-linear Frobenius map.
Hence,   the associated  morphisms of schemes
$$\Psi: V(\GL_{n(r)}) \to \Spec \bH(\GL_{n(r)},k), \quad \Phi: \Spec \bH(\GL_{n(r)},k) \to V(\GL_{n(r)})$$
are mutually inverse homeomorphisms.

\end{thm}

\begin{ex} 
\label{varM} 
We investigate $V(G)_M$ for the four examples of Example \ref{four}. 

\vspace{0.1in}
(1) 
Let $M$ be a $p$-restricted $\fg$-module of dimension $m$, given by the map
of $p$-restricted Lie algebras $\rho: \fg \to \End_k(M)\simeq \gl_m$.  Then
$V(\ul \fg)_M \subset V(\ul {\gl_m})$ consists of those $p$-nilpotent elements of $\fg$ whose
Jordan type (as an $m\times m$-matrix in $\gl_m$) has at 
least one block of size $< p$ (see \cite{FPa}).

\vspace{0.1in}
(2) 
 For $G = \Gar$, $kG \simeq kE$ where
$E$ is an elementary abelian $p$-group of rank $r$.  The rank variety of a 
$kE$-module was first investigated in \cite{C}.

We consider directly the rank variety $V(\Gar)_M$ of a finite dimensional
$k\Gar$-module $M$.  The data of such a module is the choice of $r$
$p$-nilpotent, pair-wise commuting
 endomorphisms $\wt u_0,\ldots,\wt u_{r-1} \in \End_k(M)$, given as
the image of the distinguished generators of $k\Gar$ as in (\ref{Gar}).
A 1-parameter subgroup of $\Gar$ has the form
$\mu_{\ul a}: \bG_{a(r),K} \to \bG_{a(r),K} $ for some 
$r$-tuple $\ul a = (a_0, \ldots, a_{r-1})$  of $K$-rational points
as in Example \ref{var}(2).  The condition that $\mu_{\ul a}$ be a point of
$V(\Gar)_M$ is the condition that 
$(\mu_{\ul a} \circ \epsilon)^*(M_K)$ is not free as a $K[u]/u^p$-module, which is equivalent to the condition that 
$M_K$ is not free as a $K[\wt u]/{\wt u}^p$-module where  
$\wt u = a_{r-1}\wt u_0 + a_{r-2}^p\wt u_1 \cdots + a_0^{p^{r-1}}\wt u_{r-1}
\in \End_K(M_K)$ (see \cite[6.5]{SFB2}).

\vspace{0.1in}
(3) 
Let $M$ be a finite dimensional $kG$-module with $G = \GL_{n(r)}$.
By Theorem \ref{iso},  $V(\GL_{n(r)})_M \subset V(\GL_{n(r)})$ is the closed subvariety 
whose set of points
in a field $K/k$ are 1-parameter subgroups $: \exp_{\ul \alpha}: \GarK \to \GL_{n(r),K}$ 
indexed by  $r$-tuples $\ul \alpha = (\alpha_0, \ldots,\alpha_{r-1}) 
\in M_n(K)$ of $p$-nilpotent, pairwise
commuting matrices such that $(\exp_{\ul \alpha, *} \circ \, \epsilon)^*(M_K)$ is not
a free as a $K[u]/u^p$-module.   The action of $u$ on $M_K$
is determined utilizing Example \ref{var}(3).
Namely, the action of $u$ is given by composing the
coproduct $M_K \to K[\GL_{n(r)}] \otimes M_K$ defining the $\GL_{n(r)}$-module 
structure on $M_K$ with the linear functional
$$\xymatrix{
u_{r-1} \circ \exp_{\ul \alpha}^*: \,\, K[\GL_{n(r)}]  \ar[rr]^-{\exp^*_{\ul \alpha}}&& K[\Gar] \ar[r]^-{u_{r-1}} & K}.
$$
In \S 3, we shall investigate this case  in more detail by considering some
concrete examples.

 \vspace{0.1in}
 (4)  A complete description of support varieties for simple  modules  
 for $\SL_{2(r)}$ can be found  in \cite[\S 7]{SFB2}.
We describe the situation for $G = \SL_{2(2)}$.  Let $S_\lambda$  be 
irreducible modules of
highest weight $\lambda$, where  $  
0 \leq \lambda \leq p^2-1$.  For $\lambda < p-1$, the  module $S_\lambda$ has 
dimension less than $p$
and thus $V(G)_{S_\lambda} = V(G)_{S^{(1)}_\lambda}= V(G)$. 
Here, $S_\lambda^{(i)}$ is the $i^{\rm th}$ Frobenius twist of $S_\lambda$.
For $\lambda = p-1$, the 
restriction of $S_{p-1}$ to $\SL_{2(1)} \subset \SL_{2(2)}$ is projective 
(the Steinberg module for $\SL_{2(1)}$) 
but $S_{p-1}$ is not itself projective. Hence, $V(G)_{S_{p-1}}$ is a proper non-trivial 
subvariety of $V(G)$. Using the notation introduced in  
Example~\ref{var}(4), we have 
$$V(G)_{S_{p-1}} = \{ (\alpha_0, 0 )  \, | \, \alpha_0 \in  N(sl_2) \} \subset V(G),$$ 
and
$$V(G)_{S_{p-1}^{(1)}} = \{ (0, \alpha_1)  \, | \, \alpha_1 \in  N(sl_2) \} \subset V(G)$$
(see \cite[6.10]{SFB2}).
$V(G)_{S_{p-1}}$ can be described 
as the subscheme  of $V(G)$ defined 
by the equations $x_1=y_1=z_1=0$.  For $\lambda = \lambda_0 + \lambda_1 p$ 
where $\lambda_0, \lambda_1 \leq p-1$ we have 
$S_\lambda \simeq S_{\lambda_0} \otimes S_{\lambda_1}^{(1)}$ by the Steinderg 
tensor product theorem.  Hence,   we can compute 
the support variety of $S_\lambda$ using the tensor product property of support varieties.  
For $\lambda = p^2-1$, $S_\lambda$ is the Steinberg module for $\SL_{2(2)}$, 
it is projective and, hence, 
$V(G)_{S_{p^2-1}} = \{ 0 \}$.
Overall, we get  
$$
V(G)_{S_\lambda} = 
\begin{cases}
N^{[2]}(sl_2), & \text{if $\lambda_0, \lambda_1 \not = p-1$,}\\
\{ (\alpha_0, 0 )  \, | \, \alpha_0 \in  N(sl_2) \} & \text{if 
$\lambda_0 = p-1, \lambda_1 \not = p-1$}, \\
\{ (0, \alpha_1 )  \, | \, \alpha_1 \in  N(sl_2) \} & \text{if 
$\lambda_0 \not = p-1, \lambda_1 = p-1$}, \\
0 & \text{if $\lambda = p^2-1 $.}\\
\end{cases}
$$
\end{ex}
\vskip .2in


\section{Global $p$-nilpotent operators}
\label{univ-sec}

In this section, we introduce in Definition \ref{global}
and study the global $p$-nilpotent operator
$$\Theta_G: k[G]\  \to \ k[V(G)],$$ 
a $k$-linear but not multiplicative map 
defined for any infinitesimal group scheme $G$.  This operator, when viewed as an
element of $kG \otimes k[V(G)]$, encodes all 1-parameter subgroups of $G$:  any
1-parameter subgroup $\mu: \mathbb G_{a(r),K} \to G_K$ corresponds to a $K$-valued 
point of $V(G)$, and $\mu_{*}(u_{r-1})$ equals the specialization in $KG$ of 
$\Theta_G \in kG \otimes k[V(G)]$ at this point.

If $M$ is a rational $G$-module, then $\Theta_G$ determines the 
$k[V(G)]$-linear operator
$$\Theta_M: M \otimes k[V(G)] \ \to \ M\otimes k[V(G)]$$
as formalized in Definition \ref{action}.

Before giving definitions, we mention as motivation the example 
$G = \bG_{a(1)}^{\times 2}$.  In this case, the group algebra $kG$
equals $k[x,y]/(x^p,y^p)$, the scheme of 1-parameter subgroups equals
$V(G) = \bA^2$, and  $k[V(G)] = k[\bA^2] = k[s,t]$.  In this 
special case, $\Theta_G$ takes the form
$$\Theta_G \ = \ x\otimes s + y \otimes t \ \in \  
k[x,y]/(x^p,y^p) \otimes k[s,t].$$
If $M$ is a $kG$-module, then the $k[s,t]$-linear operator $\Theta_M$ is 
given by 
$$\Theta_M: M\otimes k[s,t] \to M\otimes k[s,t],$$
$$ m\otimes 1 \mapsto 
xm \otimes s + ym\otimes t.$$
 ``Specializing" $\Theta_M$ at some $(a,b) \in K^2$ for some field extension 
$K/k$ yields the action of $ax + by$ on $M_K$.

\vspace{0.2in}
To construct our global operator, we proceed as follows.
Let $G$ be an algebraic affine group scheme over $k$ (that is, $G$ is an 
affine group scheme such that the coordinate algebra $k[G]$ is finitely 
generated over $k$ (\cite[3.3]{W})) and  let $A = k[V_r(G)]$. The natural 
isomorphism of covariant functors on commutative $k$-algebras $R$
$$
\Hom_{\text{grp sch}}(\bG_{a(r),R}, G_R) \ \simeq  \Hom_{k-\text{alg}}(A,R),
$$
given in Theorem \ref{represent} implies the existence of a universal
1-parameter subgroup of height $r$
$$\UseComputerModernTips
\xymatrix{\cU_{G,r}: \bG_{a(r), A} \  \ar[r]& \ G_A,}$$
the subgroup corresponding to the identity map on $A$. 
The  subgroup $\cU_{G,r}$ induces a map on coordinate algebras 
$$\UseComputerModernTips
\xymatrix{\cU^*_{G,r}: A \otimes k[G]   \ar[r] &  A \otimes k[\bG_{a(r)}].}
$$
Recall that  $k\Gar \simeq k[u_0,\ldots,u_{r-1}]/(u_0^p, \ldots, u_{r-1}^p)$, where $u_{r-1}$, 
the dual element to $T^{p^{r-1}}$,   is a linear functional
$u_{r-1}: k[\Gar] \to k$.  

\begin{defn}
\label{global}
Let $G$ be an algebraic affine group scheme over $k$.  We define
$$
\UseComputerModernTips
\xymatrix{\Theta_{G,r}: k[G] \ \ar[r] & \  k[V_r(G)]}
$$
 to be the $k$-linear, $p$-nilpotent functional defined by   the composition 
\begin{equation}
\label{u-r}
\UseComputerModernTips
\xymatrix{k[G] \ar[r]^-{1 \otimes  \id} & k[V_r(G)] \otimes k[G] \ar[r]^-{ \cU^*_{G,r}} &  
k[V_r(G)] \otimes k[\bG_{a(r)}] \ar[rr]^-{\id \otimes u_{r-1}} && k[V_r(G)] .}
\end{equation}
\end{defn}

As an element of  $\Hom(k[G],k[V_r(G)]) \equiv kG \otimes k[V_r(G)]$, $\Theta_{G,r}$ can 
be equivalently defined as 
\begin{equation}
\label{Theta}
\Theta_{G,r}\ = \ \cU_{G,r,*}(u_{r-1}) \ \in \ kG \otimes k[V_r(G)],
\end{equation}
where 
$$\cU_{G,r,*}: k\Gar \otimes k[V_r(G)] \to kG \otimes k[V_r(G)].
$$
Thus, $\Theta_{G,r}$ as given in (\ref{u-r}) satisfies the property that its 
composition with $k[V_r(G)] \to K$ corresponding to the
$K$-rational point $\mu: \bG_{a(r),K} \to G_K$ equals $\mu_*(u_{r-1})$.
\vskip .in

The following proposition justifies using  the simplified notation
\begin{equation}
\label{simp}
\Theta_G: k[G] \to k[V(G)],
\end{equation}
where $V(G) = \varinjlim_r V_r(G)$ as in Definition \ref{limm}.
Namely, $\Theta_G$ is defined to be isomorphic to $\Theta_{G,r}$ provided that 
$G$ is infinitesimal of height $\leq r$.

\vspace{0.1in}
Recall the canonical projection $p_{r^\prime, r}: \bG_{a(r^\prime)} \twoheadrightarrow \Gar$, and the induced closed embedding $i_{r, r^\prime}: V_r(G) \hookrightarrow V_{r^\prime}(G)$, introduced in Remark~\ref{varying}, for $r^\prime \geq r$.

\begin{prop}
\label{indep}
Let $G$ be an an infinitesimal group scheme and  let $r^\prime \geq r$. Let 
$A_r = k[V_r(G)], \ A_{r^\prime} = k[V_{r^\prime}(G)]$, and  let 
$$\phi: A_{r^\prime} \twoheadrightarrow A_r $$ be the 
surjective homomorphism corresponding to the canonical embedding 
$i_{r, r^\prime}: V_r(G) \hookrightarrow V_{r^\prime}(G)$.     
Consider $A_{r}$ as an $A_{r^\prime}$-module via $\phi$.
Then
$$\Theta_{G,r} = \Theta_{G,r^\prime} \otimes_{A_{r^\prime}} 
1 \ \in \ kG \otimes A_{r^\prime} \otimes_{A_{r^\prime}} A_{r} \simeq kG \otimes A_{r}.$$

Moreover, if $G$ is an infinitesimal group scheme of height $\leq r$, then 
$\Theta_{G, r}$ is thereby naturally identified with $\Theta_{G, {r^\prime}}$. 
\end{prop}

\begin{proof}

Consider  the composition 
\begin{equation}
\UseComputerModernTips
 \xymatrix{ \bG_{a(r^\prime),A_{r}} \ar[dr]_{p_{r^\prime, r}} 
\ar[rr]^{\cU_{G,r} \circ p_{r^\prime, r}} &&  G_{A_{r}}\\
& \bG_{a(r),A_{r}}   \ar[ur]_{\cU_{G, r}} &  
.}
\end{equation}
Since $\cU_{G,r} \in V_r(G)(A_r)  \simeq \Hom(A_r, A_r)$  corresponds to the 
identity map on $A_r$, and  $p_{r^\prime,r}$ is the map that 
induces $\phi :A_{r^\prime} \to A_r$, we conclude that the composition 
$\cU_{G,r} \circ p_{r^\prime, r} \in V_{r^\prime}(G) \simeq 
\Hom(A_{r^\prime}, A_r)$ corresponds to $\phi$.  Hence,  the universality 
of $\cU_{G,r^\prime}$ implies that $\cU_{G,r} \circ p_{r^\prime, r}$ is 
obtained by pushing down  the universal  one-parameter 
subgroup $\cU_{G,r^\prime}$ via $\phi:A_{r^\prime} \to A_r$.   
Therefore, we conclude
\begin{equation}
\cU_{G,r} \circ p_{r^\prime, r} \ = \ \cU_{G,r^\prime}\otimes_{A_{r^\prime}} A_r
\end{equation}
which implies the equality of maps of group algebras
\begin{equation}
\label{rprime}
\cU_{G,r^\prime,*} \otimes_{A_{r^\prime}} A_r = \cU_{G,r,*} \circ p_{{r^\prime},r, *}: 
\ k\bG_{a(r^\prime)} \otimes A_r\to kG \otimes A_r.
\end{equation}
  Since $p_{r^\prime,r,*}(u_{r^\prime-1}) = u_{r-1} \in k\Gar$,   we conclude
  $(\cU_{G,r,*} \circ p_{{r^\prime},r, *})(u_{r^\prime -1}) = \cU_{G,r,*} (u_{r-1}) = 
\Theta_{G,r},$
  whereas 
  $ (\cU_{G,r^\prime, *} \otimes_{A_{r^\prime}} A_r)(u_{r^\prime-1}) =  
  \cU_{G,r^\prime, *} (u_{r^\prime-1})  \otimes_{A_{r^\prime-1}} 1 = 
  \Theta_{G,r^\prime} \otimes_{A_{r^\prime}} 1.$

The second statement follows immediately from the fact that for $G$ 
of height $\leq r$, the map $\phi:A_{r^\prime} \to A_r$ 
is an isomorphism as shown in Remark~\ref{varying}. 
\end{proof}

Let $G$ be an affine group scheme over $k$, $M$ be a $kG$-module, 
and $\nabla_M: M \to  M \otimes k[G]$ be 
the corresponding co-action. A $k$-linear functional
with values in a commutative $k$-algebra $A$, $\Theta: k[G] \ \to \ A$, determines
an action of $\Theta$ on $M \otimes A $ which  is the  $A$-linear extension 
\begin{equation}
\label{exten}
\Theta_M: M \otimes A  \to M \otimes A 
\end{equation}
of the map 
$$\UseComputerModernTips
\xymatrix{M  \ \ar[r]^-{\quad  \nabla_M \quad} & M \otimes k[G] 
\ \ar[r]^-{\, \id  \otimes \Theta_M \,} & M \otimes A  }.$$
If $G$ is finite, we may view $\Theta \in \Hom_k(k[G], A)$ as an element of 
$kG \otimes A$ (which we
also denote by $\Theta$).  From this point of view, the action (\ref{exten})
is simply multiplication by $\Theta$.

We now define the {\it global $p$-nilpotent operator} on a $G$-module $M$.

\begin{defn}
\label{action}
 Let $G$ be an infinitesimal group scheme.
For any $k[G]$-comodule $M$ with co-action $\nabla_M: M \to  M  \otimes k[G]$, 
we define the $p$-nilpotent operator 
$$\Theta_M: M  \otimes k[V(G)] \to M  \otimes k[V(G)]$$
to be the $k[V(G)]$-linear extension of the map 
$(\id_M \otimes \Theta_G) \circ \nabla_M: M \to M  \otimes k[V(G)].$
\end{defn}
\begin{remark}
The fact that $\Theta_M$ is $p$-nilpotent follows immediately from (\ref{Theta}) 
since $u_{r-1}^p=0$.
\end{remark}

Slightly abusing notation, we shall often refer to $\Theta_G$ itself as the global $p$-nilpotent operator.

We reformulate the pairing (\ref{exten}) in a more geometric fashion as follows.

\begin{prop}
\label{alg-act}
Let $G$ be a group scheme over $k$, $V$ be an affine $k$-scheme,
and let $M$ be a finite dimensional $G$-module.  Then a $k$-linear functional
$\Theta: k[G] \to k[V]$ determines the pairing of  $k$-schemes
\begin{equation}
\label{pair}
V \times M \to M.
\end{equation}
As a pairing of representable functors of commutative $k$-algebras $A$, this
pairing sends 
$(v: k[V] \to A, m = \sum_i a_i \otimes m_i)$ to 
$\sum_i a_i(\sum_j v(\Theta(f_{i,j})) m_j)$ where 
$\nabla(m_i) = \sum_j f_{i,j}\otimes m_j$.
\end{prop}

\begin{ex}
\label{p-univ}
We describe the global $p$-nilpotent operator $\Theta_G$ in each of
the four examples of Example \ref{four}.

\vspace{0.1in}
(1) Let $G = \GL_{m(1)} \equiv \ul \gl_m$, with group algebra $k \ul\gl_m = u(\gl_m). $
Then the composition of
$$\Theta_G: k[G] = k[X_{i,j}]/(X_{i,j}^p - \delta_{i,j}) \to k[N_p(\gl_m)]$$
with some $K$-rational point $x \in N_p(\gl_m)$ is the evaluation of the 
matrix coordinate functions $X_{i,j}$ on $x$.   In other words,
$$\Theta_{\ul \gl_m}    \ = \ \sum_{1\leq i,j \leq m} x_{i,j} \otimes \ol X_{i,j}  \ 
\in \ \cu(\gl_m)\otimes  k[N_p(\gl_m)],$$
where $\ol X_{i,j}$ is the image of the $i,j$-matrix coefficient on $\gl_m$.
 
For a general $p$-restricted Lie algebra $\fg$, we have
\begin{equation}
\label{lie}
\Theta_{\ul \fg}  \ = \ \sum x \otimes \ol X \ \in \ \cu(\fg)\otimes  k[N_p(\fg)] 
\end{equation}
where the sum is over  basis elements $X \in \fg^\#$ with
image $\ol X \in  k[N_p(\fg)] $ and with dual $x \in \fg$.

We record an explicit formula for the universal $p$-nilpotent operator  
in the case of $\fg=sl_2$ for future reference. 
We have $k[N_p(sl_2)] \simeq k[x,y,z]/(xy+z^2)$.  
Let $e,f,h$ be the standard basis of the  $p$-restricted Lie algebra $sl_2$.  Then 
$$\Theta_{\ul sl_2} = xe + yf + zh.$$
Observe that this formula agrees with  the presentation of a ``generic" 
$\pi$-point for $u(sl_2)$ as given in \cite[2.5]{FP2}.

\vspace{0.1in}
(2)
Take $G = \Gar$.  Then $k[\Gar] \simeq k[T]/T^{p^r}$, and 
$k[V(\Gar)] \simeq k[x_0,\ldots,x_{r-1}]$ is graded in such a 
way that $x_i$ has degree $p^i$ (see Proposition \ref{homog} below).  
We compute $\xymatrix{\Theta_{\Gar}: k[\Gar] \ar[r] & k[V(\Gar)] }$ explicitly 
in this case  (see also \cite[6.5.1]{SFB2}).    

One-parameter subgroups  of $\mathbb G_{a(r),K}$  are  in one-to-one 
correspondence  with the additive polynomials in 
$K[T]/T^{p^r}$, that is, polynomials of the  form $p(T) = a_0T + a_1T^p + \ldots + a_{r-1}T^{p^{r-1}}$  
(see \cite[1.10]{SFB1}).  The map on coordinate algebras  induced by the 
universal one-parameter subgroup 
$\cU: \mathbb G_{a(r),k[V(G)]} \to \mathbb G_{a(r),k[V(G)]}$ is given by the  
``generic" additive  polynomial:
$$\UseComputerModernTips
\xymatrix{\cU^*: k[x_0,\ldots,x_{r-1}][T]/T^{p^r} \ar[r] & k[x_0,\ldots,x_{r-1}][T]/T^{p^r}}.
$$
$$ T \mapsto x_0T + x_1T^p + \ldots + x_{r-1}T^{p^{r-1}}.$$
To determine the   linear functional 
$$\UseComputerModernTips
\xymatrix{\Theta_{\Gar} = u_{r-1} \circ \cU^* : k[T]/T^{p^r} \ar[r]& k[x_0. \ldots, x_{r-1}]},
$$
it suffices to  determine  the values of $\Theta_{\Gar}$ on the linear generators 
$\{ T^i \}, 0 \leq i \leq p^r-1$.
Since $u_{r-1}$ is the dual to $T^{p^{r-1}}$, this further reduces to determining 
the coefficient  by $T^{p^{r-1}}$  
in $\cU^*(T^i) = (x_0T + x_1T^p + \ldots + x_{r-1}T^{p^{r-1}})^i$.  Computing this 
coefficient, we conclude that $\Theta_{\Gar}$  is 
 given explicitly on the basis elements
of $k[\Gar] \simeq k[T]/T^{p^r}$ by
\begin{equation}
\label{expl_coor}
\quad T^i \mapsto 
\sum\limits_{\stackrel{i_0 + i_1 + \cdots + i_{r-1} = i} {i_0 + i_1p + 
\cdots + i_{r-1}p^{r-1} = p^{r-1}}} {i \choose {i_0, i_1, \ldots, i_{r-1}}} 
x_0^{i_0}\ldots x_{r-1}^{i_{r-1}},
\end{equation}
where ${i \choose {i_0, i_1, \ldots, i_{r-1}}} = \frac{i!}{i_0!i_1!\ldots i_{r-1}!}$ 
is the multinomial coefficient.  
Let $\{ v_0, \ldots, v_{p^r-1} \}$ be the  linear basis of $k\Gar$ dual to  
$\{ 1, T, \ldots, T^{p^{r}-1} \}$.  Dualizing  (\ref{expl_coor}), we  
conclude that $\Theta_{\Gar}$ as an element of $k\Gar \otimes k[V(\Gar)] $)  
has the following form:
\begin{equation}
\label{expl_group}
\Theta_{\Gar} = \sum\limits_{i=0}^{p^r-1} v_i\left[\sum\limits_{\stackrel{i_0 + i_1 + 
\cdots + i_{r-1} = i} {i_0 + i_1p + \cdots + i_{r-1}p^{r-1} = p^{r-1}}} 
{i \choose {i_0, i_1, \ldots, i_{r-1}}} 
x_0^{i_0}\ldots x_{r-1}^{i_{r-1}} \right] .
\end{equation} 
By \cite[1.4]{SFB1}, $v_i$ can expressed in 
terms  of the algebraic generators $u_j$ of $k\Gar$ via  
the following formulae 
$$
v_i=\frac{u_0^{j_0}\ldots u_{r-1}^{j_{r-1}}}
{j_0!\ldots j_{r-1}!}
,$$
where $i=j_0 +j_1p+\ldots +j_{r-1}p^{r-1}\quad
(0\le j_\ell \le p-1)$ is the $p$-adic expansion of $i$.    

Using these formulae, it is straightforward to calculate the term of $\Theta_{\Gar}$ which is linear with respect to $u_i$ 
(and homogeneous of degree $p^{r-1}$  with respect to the grading  of $k[V(\Gar)]$: 
$$u_0x_{r-1} + u_1x^p_{r-2} + \ldots + u_{r-1}x_0^{p^{r-1}}$$
The ``linear" term gives the entire operator $\Theta_{\Gar}$ for $r=1,2$, but for $r \geq 3$, the non-linear terms start to appear. 

\vspace{0.1in}
(3) Let $G = \GL_{n(r)}$. Recall that $V(\GL_{n(r)})$ is the $k$-scheme of
$r$-tuples of $p$-nilpotent, pair-wise commuting matrices.  For 
notational convenience, let $A$ denote $k[V(\GL_{n(r)})] = k[M_n^{\times r}]/I$,
a quotient of the coordinate algebra of the variety $M_n^{\times r}$ of 
$r$-tuples of $n\times n$ matrices.
Then 
$\cU_{\GL_{n(r)}}: \bG_{a(r),A} \to \GL_{n(r),A}$ is specified by the 
$A$-linear map on coordinate algebras 
\begin{equation}\label{lin-coord}
\cU_{\GL_{n(r)}}^*: A[\GL_{n(r)}] \ \to \ A[T]/T^{p^r}, \quad X_{i,j} 
\mapsto \sum_{\ell=0}^{p^r-1} (\beta_\ell)_{i,j}T^j
\end{equation}
where $\{ X_{i,j}; 1 \leq i,j \leq n \}$ are the matrix coordinate functions of $\GL_n$,
where $\beta_\ell$ is given as in formula (\ref{beta}) in terms 
of the matrices $\alpha_0,\ldots,\alpha_{r-1} \in M_n(A)$, and  
$\alpha_i = \beta_{p^i}$ have matrix coordinate functions which generate $A$.  
(Indeed, the $n^2r$
entries of $\alpha_0,\ldots,\alpha_{r-1}$ viewed as variables generate 
$A$, with relations given by the conditions that these matrices must be
$p$-nilpotent and pairwise commuting.)

The $p$-nilpotent operator 
$$\Theta_{\GL_{n(r)}}:  k[\GL_{n(r)}] \to k[V(\GL_{n(r)})]$$
is given by the $k$-linear functional sending a polynomial in the matrix
coefficients $P(X_{i,j}) \in k[\GL_{n(r)}]$ to the coefficient of $T^{p^{r-1}}$
of the sum of products corresponding to the polynomial $P$ 
given by replacing each $X_{i,j}$ by
$\sum_{\ell=0}^{p^r-1}(\beta_\ell)_{i,j}T^\ell$ (when taking products of matrix 
coefficients, one uses the usual rule for matrix multiplication);
$$P(X_{i,j}) \ \mapsto \ P(\sum_{\ell=0}^{p^r-1}(\beta_\ell)_{i,j}T^\ell) \ \mapsto \
\text{coeff \ of \ } T^{p^{r-1}}.$$

For example, the coaction $k^n \to   k[\GL_n] \otimes k^n$ corresponding to the 
natural representation of $\GL_n$ on $k^n$
determines an action of $\Hom_k(k[\GL_{n(r)}],A) \subset \Hom_k(k[\GL_n],A)$
on $A^n$, so that $\Theta_{GL_{n(r)}}: A^n \to A^n$ is  given in matrix form by 
$(\Theta_{GL_{n(r)}}(X_{i,j}))$.

\vspace{0.1in}
(4) We consider $G = \SL_{2(2)}$, and assume notations and conventions of Example~\ref{var}(4). Let $A = k[\SL_{2(2)}]$. Using the general discussion in (\ref{p-univ}(3)) above (also compare to  (\ref{var}(4))), one readily computes that the map   on coordinate algebras 
$\cU_{\SL_{2(2)}}^*: A[\SL_{2(2)}] \to A[\bG_{a(2)}]\simeq A[T]/T^{p^2}$
is given by
$$X_{1,1} \mapsto 1 + z_0T + z_1T^p, \quad X_{1,2} \mapsto x_0T + x_1T^p$$
$$X_{2,1} \mapsto y_0T + y_1T^p, \quad X_{2,2} \mapsto 1 - z_0 - z_1T^p.$$
By (\ref{Theta}), $\Theta_{\SL_{2(2)}} = \cU_{\SL_{2(2)}*}(u_1)$, the functional given by 
``reading off the coefficient of $T^p$".  

Let $e,f,h,e^{(p)},f^{(p)},h^{(p)} \in k\SL_{2(2)}$ denote respectively the linear duals of
the functions  $X_{1,2}, X_{2,1}, X_{1,1}-1, X_{1,2}^p,X_{2,1}^p, (X_{1,1}-1)^p$ on $\SL_{2(2)}$, and set 
$$e^{(i)} = \frac{e^i}{i!},\quad  f^{(i)} = \frac{f^i}{i!}, \quad 
{h\choose i} = \frac{h(h-1)(h-2)\ldots(h-i+1)}{i!}$$
for $i<p$.  Fix the linear basis of $k[\SL_{2(2)}]$ given by 
 powers of $X_{1,2}, X_{2,1}, X_{1,1}-1$ 
 (in this fixed order). Then the element of $k\sl2$  dual to 
 $X^i_{1,2} X^j_{2,1} (X_{1,1}-1)^{\ell}$ for $i + j + \ell \leq p$   
 is given by 
 $$(X^i_{1,2} X^j_{2,1} (X_{1,1}-1)^{\ell})^\# \ = \ e^{(i)} f^{(j)} {h\choose \ell} $$ 
 (where ${h\choose p}$ is identified with $ h^{(p)}$ by definition). \\
 With these conventions $\Theta_{\SL_{2(2)}} \in k\SL_{2(2)} \otimes k[V(\SL_{2(2)})]$
 equals
 \begin{equation}
 \label{exp-sl}
 x_1e + y_1f + z_1 h + x_0^pe^{(p)}  +  
y_0^pf^{(p)} + z_0^ph^{(p)} + \sum\limits_{\stackrel{i+j+\ell=p} {i, j, \ell <p}} 
x_0^iy_0^jz_0^\ell e^{(i)} f^{(j)} {h\choose \ell}.
\end{equation}
\end{ex}

\vskip .2in

Our motivational example for $G=\bG_{a(1)} \times \bG_{a(1)}$ from the beginning of this section is a special case of (\ref{p-univ}(1)).
\begin{ex}
 \label{exp-elem} Let $G = \bG_{a(1)}^{\times r}$. Then $G$ corresponds to the abelian $p$-nilpotent  Lie algebra $g_a^{\oplus r}$, and $kG  = k[u_0, \ldots, u_{r-1}]/(u_0^p, \ldots, u_{r-1}^p)$. We have $V(G) \simeq \mathbb A^r$, and $k[V(G)] \simeq k[X_0, \ldots, X_{r-1}]$ where all generators are in degree one.  Then $\Theta_G \in kG \otimes k[V(G)]$ is given by the simple formula 
$$\Theta_G = u_0X_0 + \cdots + u_{r-1}X_{r-1}.$$
It is useful to contrast this formula with a much more complicated result for $G = \Gar$ in (\ref{p-univ}(2)).
\end{ex}

\vskip .1in

To complement Example \ref{p-univ}, we make explicit the action of $\Theta_G$
on some representation of $G$  for each of the four types
of finite group schemes we have been considering in examples.

\begin{ex}
\label{ex-rep}
(1) Let $G = \ul \fg$ and let $M = \fg^{ad}$ denote the adjoint representation 
of the $p$-restricted Lie algebra $\fg$; let $\{ x_i \}$ be a basis for $\fg$.
We identify $\Theta_{\ul \fg}$ as the $k[N_p(\fg)]$-linear endomorphism 
$$\Theta_{\ul \fg}:  \fg^{ad} \otimes  k[N_p(\fg)]  \ \to \ \fg^{ad} \otimes  k[N_p(\fg)], \quad
x\otimes 1 \mapsto \ \sum_i  [x_i,x] \otimes X_i ,$$
where $X_i $ is the image  under the projection 
$S^*(\fg^\#) \to k[N_p(\fg)]$  of the dual basis element to $x_i$.

\vspace{0.1in}
(2)  Let $M$ denote the cyclic $k\bG_{a(r)}$-module 
$$M \ = \ k[u_0,\ldots,u_{r-1}]/(u_0, u_1^p, \ldots, u_{r-1}^p) \ \simeq \
k[u_1,\ldots,u_{r-1}]/(u_1^p, \ldots, u_{r-1}^p).$$  
 As recalled in Example \ref{p-univ}(2), $k[V(\bG_{a(r)})] = k[\bA^r] = 
 k[a_0,\ldots,a_{r-1}]$, $k\Gar = k[u_0,\ldots,u_{r-1}]/(u_i^p)$, 
 and 
$$\Theta_{\bG_{a(r)}} \in \ A[u_0,\ldots,u_{r-1}]/(u_i^p)$$
is given by the complicated,
but explicit formula (\ref{expl_group}).  We conclude that 
$$\Theta_{\bG_{a(r)}}: M \otimes A \ \to \ M \otimes A$$
is the $A$-linear endomorphism sending $u_i$ to $\ol \Theta_{\bG_{a(r)}} \cdot u_i$,
where $\ol \Theta_{\bG_{a(r)}}$ is the image of $\Theta_{\bG_{a(r)}}$ under the projection 
$A[u_0,\ldots,u_{r-1}]/(u_i^p) \to A[u_1,\ldots,u_{r-1}]/(u_i^p)$.

\vspace{0.1in}
(3)  Let $M$ be the restriction to $\GL_{n(r)}$ of the canonical $n$-dimensional 
rational $\GL_n$-module $V_n$.  By Example \ref{var}(3), $A = k[V(\GL_{n(r)}]$ is the
quotient of $k[\gl_n]^{\otimes r}$ by the ideal generated by the equations satisfied
by an $r$-tuple of $n\times n$-matrices with the property that each matrix is
$p$-nilpotent and that the matrices pair-wise commute.  
The complexity of the map
$$
\Theta_{\GL_{n(r)}}: V_n \otimes A \ \to \ V_n \otimes A
$$
is revealed even in the case $n=2$ which is worked out 
explicitly below.

\vspace{0.1in}
(4) Let $M$ be the restriction to $\SL_{2(2)}$ of the rational $\GL_2$--representation $V_2$.  Then Example \ref{var}(4) gives an explicit
description of $A = k[V(\SL_{2(2)})]$ as a quotient of $k[x_0,y_0,z_0,x_1,y_1,z_1]$
and (\ref{exp-sl}) gives $\Theta_{\SL_{2(2)}}$ explicitly.
Since $V_2$ is a homogeneous polynomial representation of $\GL_2$ of degree $1$, 
the divided powers $e^{(p)}$, 
$f^{(p)}$ and $h^{(p)}$ as well as  all products  of the form 
$e^{(i)}f^{(j)}{h \choose \ell}$ act trivially on $M$. Hence,  
the  map 
$$\Theta_{\SL_{2(2)}}: M \otimes A \ \to \ 
M\otimes A$$ is given by the matrix 
$$\UseComputerModernTips
 \xymatrix{ A^2 \ar[rr]^{\left [\begin{array}{cc} 
z_1 & x_1\\
y_1 & -z_1
\end{array} \right ]} && A^2. }
$$
\end{ex}
\vskip .2in

When viewing group schemes as functors, it is often convenient to think of
$G_{k[V(G)]}$ as $G \times V(G)$ (i.e., $G \times V(G) \ = \
\Spec  (k[V(G)] \otimes k[G]) $).  From this point of view, $\cU_G$ has the form
$$\UseComputerModernTips
 \xymatrix{ \cU_G: \Gar \times V(G) \ \ar[r]& \ G \times V(G).}$$
 
 The following naturality property of $\Theta_G$ will prove useful when 
we consider $M\otimes k[V(G)]$ as a free, coherent sheaf on $V(G)$
and restrict this sheaf to $V(H) \subset V(G)$ equipped with its action
of $H$.

\begin{prop}
\label{pull}
Choose $r > 0$ and consider a closed embedding $i: H \hookrightarrow G $ of algebraic affine group schemes
over $k$.  Let  $\phi: V_r(H) \to V_r(G)$ be the closed embedding of affine schemes induced by $i$, 
with associated surjective map $\phi^*: k[V_r(G)] \to k[V_r(H)]$ on coordinate algebras. 
Then the following square commutes
\begin{equation}
\label{pullback1}
\UseComputerModernTips
 \xymatrix{
\bG_{a(r)} \times V_r(H) \ar[d]_{id \times \phi} \ar[r]^-{\cU_{H,r}} & H \times V_r(H)
\ar[d]^{i \times \phi} \\
\bG_{a(r)} \times V_r(G) \ar[r]^-{\cU_{G,r}} & G \times V_r(G)
.}
\end{equation}
Consequently,
the following square of 
$k$-linear maps commutes:
\begin{equation}
\label{pullback2}
\UseComputerModernTips
\xymatrix{
k[G] \ar[d]^{i^*} \ar[rr]^-{\Theta_{G,r}} && k[V_r(G)] \ar[d]^{\phi^*} \\
k[H] \ar[rr]^-{\Theta_{H,r}} && k[V_r(H)]
.}
\end{equation}
Thus, for any rational $G$-module $M$ we have a compatibility 
of coactions on $M$:
\begin{equation}
\label{pullback3}
\UseComputerModernTips
\xymatrix{
M \ar[rr]^-{\nabla_M} \ar[rrd]_{\nabla_{M\downarrow_H}} &&  M  \otimes k[G] \ar[d]^{\id\otimes i^*} 
\ar[rr]^-{ \id\otimes\Theta_{G,r}} &&   M \otimes k[V_r(G)] \ar[d]^{\id\otimes\phi^* } \\
&& M  \otimes k[H]\ar[rr]^-{\id\otimes\Theta_{H,r}} &&   M \otimes k[V_r(H)]
.}
\end{equation}

\end{prop}

\begin{proof}
The fact that $\phi: V_r(H) \to V_r(G)$ induced by the closed embedding 
$i: H \hookrightarrow G $ is itself a closed embedding is given by \cite[1.5]{SFB1}.
By universality of $\cU_{G,r}$, the composition $(i \times id) \circ \cU_{H,r}:
\bG_{a(r)} \times V_r(H) \to G \times V_r(H)$ is 
obtained by pull-back of $\cU_{G,r}$ via some morphism 
$V_r(H) \to V_r(G)$.  By comparing maps on $R$-valued points, we verify that
this morphism must be $\phi$.  This implies the commutativity of (\ref{pullback1}).

The commutative square (\ref{pullback1})  gives a commutative square on coordinate algebras: 
\begin{equation}
\label{pullback5}
\xymatrix{
k[V_r(G)] \otimes k[G] \ar[d]^{\phi^* \otimes \, i^*} \ar[r]^-{\cU_{G,r}^*} & 
k[V_r(G)] \otimes k[\Gar] \ar[d]^{\phi^* \otimes \, \id} \\
k[V_r(H)] \otimes  k[H]  \ar[r]^-{\cU_{H,r}^*} & k[V_r(H)]  \otimes k[\Gar]
.}
\end{equation}
 Concatenating (\ref{pullback5}) on the right with the
commutative square of linear maps
$$
\xymatrix{
k[V_r(G)] \otimes k[\Gar] \ar[d]^-{\phi^* \otimes  \id} \ar[r]^-{u_{r-1}} & k[V_r(G)] \ar[d]^{\phi^*  } \\
k[V_r(H)] \otimes  k[\Gar]  \ar[r]^-{u_{r-1}} & k[V_r(H)] 
}
$$
and  with the inclusions $k[G] \to k[V_r(G)] \otimes k[G]$ and $k[H] \to k[V_r(H)] \otimes k[H]$ 
on the left, we obtain a commutative diagram: 
$$\xymatrix{
k[G] \ar[d]^{i^*} \ar@{^(->}[r] & k[V_r(G)] \otimes k[G] \ar[d]^{\phi^* \otimes \, i^*} \ar[r]^-{\cU_{G,r}^*} & 
k[V_r(G)] \otimes k[\Gar] \ar[d]^{\phi^* \otimes \, \id}  \ar[r]^-{u_{r-1}} & k[V_r(G)] \ar[d]^{\phi^*  }\\
k[H] \ar@{^(->}[r] & k[V_r(H)] \otimes  k[H]  \ar[r]^-{\cU_{H,r}^*} & k[V_r(H)]  
\otimes k[\Gar] \ar[r]^-{u_{r-1}} & k[V_r(H)]
.}
$$
Eliminating the middle square,  we obtain the square (\ref{pullback2}).  Hence, it is commutative.

Finally, the commutativity of (\ref{pullback3}) follows immediately from the
commutativity of (\ref{pullback2}).
\end{proof}

Pre-composition determines a natural action
$V_r(G) \times V_r(\Gar) \ \to \ V_r(G)$ for any algebraic affine group scheme $G$.  
Recall from \cite[1.10]{SFB1} that $ V_r(\Gar) \simeq \bA^r$:
morphisms of group schemes $\mathbb G_{a(r),A} \to \mathbb G_{a(r),A}$ over $A$ have
associated maps on coordinate algebras $A[T]/T^{p^r} \to A[T]/T^{p^r}$ given
by additive polynomials, that is polynomials of the form $a_0T + a_1T^p + \cdots a_{r-1}T^{p^{r-1}}$.
Restricting the action $V_r(G) \times V_r(\Gar) \ \to \ V_r(G)$ 
to {\it linear} polynomials $\bA^1 \subset  V_r(\Gar) \simeq  \bA^r$
determines a natural action 
\begin{equation}
\label{grading}
V_r(G) \times \bA^1 \ \to \ V_r(G),
\end{equation}
which is equivalent by \cite[1.11]{SFB1} to a functorial grading on $k[V_r(G)]$.

\begin{prop}
\label{grade}
Let $G$ be an algebraic affine group scheme over $k$.  Then the coordinate algebra $k[V_r(G)]$
of $V_r(G)$ is a graded algebra generated by homogeneous generators
of degrees $p^i, \ 0 \leq i < r$.
\end{prop}

\begin{proof} The coordinate algebra $k[V_r(G)]$ is graded by \cite[1.12]{SFB1}.
If $G = \GL_N$, then an $R$-valued point of $V_r(\GL_N)$ is given by 
an $r$-tuple of $N\times N$ pair-wise commuting, $p$-nilpotent matrices
with entries in $R$, $(\alpha_0,\ldots,\alpha_{r-1})$.   
 The action of $c \in V(\bG_{a(1)})(R)$
on $(\alpha_0,\alpha_1,\ldots,\alpha_{r-1}) \in V(\GL_{N(r)})(R)$  is given by the formula
$$ (\alpha_0,\alpha_1,\ldots,\alpha_{r-1}) \times c = 
(c\alpha_0,c^p\alpha_1,\ldots,c^{p^{r-1}}\alpha_{r-1}).
$$ 
Hence, the coordinate functions 
of the matrix $\alpha_i$ have grading $p^i$ and, therefore, $k[V_r(GL_N)]$ is generated 
by homogeneous elements of degree $p^i, \ 0 \leq i < r$. 

Let $i: G \to GL_N$ be a closed embedding of a finite group
 scheme $G$ into some $\GL_N$.  The naturality of the grading (see \cite[1.12]{SFB1}) implies
 that the surjective map $\phi: k[V_r(\GL_N)] \to k[V_r(G)]$ is a map of
 graded algebras.
 \end{proof}

\begin{prop}
\label{homog}
For any algebraic affine group scheme $G$ and integer $r > 0$, the $k$-linear map
$$\Theta_{G,r}: k[G] \ \to \ k[V_r(G)]$$
has image contained in the homogeneous summand of $k[V_r(G)]$ of
degree $p^{r-1}$.  

If $G$ is infinitesimal, then this is equivalent to the following: 
$$\cU_{G,r*}:  k\Gar  \otimes k[V_r(G)]  \ \to \  kG \otimes k[V_r(G)] $$
sends $u_{r-1}  \otimes 1 \in  k\Gar \otimes k[V_r(G)]$ to some $\sum x_i \otimes a_i \in    kG \otimes k[V_r(G)]$ 
with each $a_i \in k[V_r(G)]$ homogeneous of  degree $p^{r-1}$. 
\end{prop}

\begin{proof}
Let $A = k[V_r(G)]$. Since $\Theta_{G,r}$ factors through the $r^{\rm th}$ Frobenius kernel of $G$, we may assume $G$ is infinitesimal. 
Let $\langle \lambda_i \rangle$ be a set of   linear generators of $k[G]$, 
and  $\langle \check \lambda_i \rangle$  be the dual set of  linear 
generators  of $kG$. Then  $\cU_{G,r,*}(u_{r-1}) = 
\sum \check \lambda_i \otimes f_i$ if and  only if $u_{r-1}(\cU_G^*(\lambda_i)) = f_i$ if 
and only if $\cU_{G,r}^*(\lambda_i) = \ldots + f_i T^{p^{r-1}} + \ldots$.  Hence, 
the assertion that $\Theta_{G,r}$ is homogeneous of degree $p^{r-1}$ is equivalent 
to showing  that the  map $k[G] \to A$ defined by reading 
off the coefficient  of  
$$\cU_{G,r}^*: k[G] \ \to \ A \otimes k[G] \ \to \ A \otimes k[\Gar] \ \to \ A[T]/T^{p^r}$$ 
of the 
monomial $T^{p^{r-1}}$ is  homogeneous of degree $p^{r-1}$.  

The coordinate algebra $k[\Gar] \simeq k[T]/T^{p^r}$ has a natural grading 
with $T$ assigned degree 1.  This grading 
corresponds to the monoidal action of $\mathbb A^1$ on $\Gar$ by multiplication: 
$$\UseComputerModernTips
 \xymatrix{ \Gar \times \mathbb A^1  \ar[rr]^-{ s \times a \mapsto sa} && \Gar}.$$
We proceed to prove that this action is compatible with the action of $\mathbb A^1$ on 
$V_r(G)$ which defines the grading on $A$ in the 
sense that the following diagram commutes:
\begin{equation}
\label{compatible}
\UseComputerModernTips
 \xymatrix{ \Gar \times \mathbb A^1 \times V_r(G) 
\ar[d]^-{ {\rm action} \times \id}  \ar[rr]^-{\id \times {\rm action}}  && \Gar \times V_r(G)
\ar[r]^-{\cU_{G,r}} & G \times V_r(G) \ar[d]^{pr_{_G}} \\
\Gar \times V_r(G) \ar[rr]^-{\cU_{G,r}} && 
G\times V_r(G) \ar[r]^-{pr_{_G}} & G.}
\end{equation}

Commutativity of (\ref{compatible}) is equivalent to the commutativity of 
the corresponding
diagram of $S$-valued points for any choice of finitely generated commutative 
$k$-algebras $S$ and element $a \in S$:

\begin{equation}
\label{S-valued}
\UseComputerModernTips
 \xymatrix{ \Gar(S)  \times V_r(G)(S) \ar[d]^-{ a \times 1}  \ar[rr]^-{1 \times a}  
&& \Gar(S) \times V_r(G)(S)
\ar[r]^-{\cU_{G,r}(S)} & G(S) \times V_r(G)(S) \ar[d]^{pr_{_G}} \\
\Gar(S) \times V_r(G)(S) \ar[rr]^-{\cU_{G,r}(S)} && 
G(S)\times V_r(G)(S) \ar[r]^-{pr_{_G}} & G(S).}
\end{equation}

Choose an embedding of $G$ into some $\GL_{N(r)}$.
Using Proposition \ref{pull} and the naturality with respect
to change of $G$ of the action of $\bA^1$ on $V_r(G)$, we can compare 
the diagram (\ref{S-valued}) for $G$ and for $\GL_{N(r)}$.  The injectivity
of $G(S) \to \GL_{N(r)}(S)$ implies that it 
suffices to assume that $G = \GL_{N(r)}$.
Let $s \in \Gar(S)$, 
$\underline \alpha = (\alpha_0, \ldots, \alpha_{r-1}) \in V_r(\GL_{N})(S)$. 
Then $a \circ \underline \alpha =  
(a\alpha_0, a^p\alpha_1, \ldots, a^{p^{r-1}}\alpha_{r-1})$, so
$\exp_{\underline \alpha}(s) = 
\exp(s\alpha_0)\exp(s^p\alpha_1) \ldots \exp(s^{p^{r-1}}\alpha_{p-1}) \in \GL_{N(r)}(S)$
by (\ref{exp}). 
Thus, restricted to the   point $(s, \underline  \alpha) \in (\Gar \times V_r(\GL_{N}))(S)$,
(\ref{S-valued}) becomes
\begin{equation}
\label{s-alpha}
\UseComputerModernTips
 \xymatrix{ (s, \underline \alpha) \ar[d]^-{ a \times 1}  \ar[rr]^-{1 \times a}  && 
(s, a \circ \underline \alpha)
\ar[r]^-{\cU_{G,r}}&   (\exp_{a\circ \underline \alpha}(s),  a \circ \underline \alpha  ) \ar[d]\\
(as,  \underline \alpha )\ar[rr]^-{\cU_{G,r}} && 
(\exp_{\underline \alpha}(as),  \underline \alpha  ) \ar[r]& \exp_{\underline \alpha}(as).}
\end{equation}
Commutativity of (\ref{s-alpha}) is implied by the evident
equality $\exp_{a\circ \underline \alpha}(s) = \exp_{\underline \alpha}(as)$. 

Consequently, we have a commutative diagram  on coordinate algebras 
corresponding  to (\ref{compatible}): 
\begin{equation}
\label{actt}
\UseComputerModernTips
 \xymatrix{ A \otimes k[t] \otimes k[\Gar]     && 
A \otimes k[\Gar] \ar[ll]_-{  {\rm act}^* \otimes \id}
 & A \otimes k[G] \ar[l]_-{\cU_{G,r}^*}\\
A \otimes k[\Gar] \ar[u]_-{ \id \otimes  {\rm act}^*}  && 
A \otimes k[G] \ar[ll]_-{\cU_{G,r}^*} & k[G].\ar@{_(->}[u] \ar@{_(->}[l]}
\end{equation}
The map $\UseComputerModernTips
 \xymatrix{{\rm act}^*: A \ar[r] & A \otimes k[t] =  A  \otimes k[\mathbb A^1]}$  of the upper horizontal arrow 
is the map on coordinate algebras  which corresponds to the 
grading  on $A$.  The left vertical map corresponds to the grading 
on $k[\Gar] \simeq k[T]/T^{p^r}$ and  is given explicitly by $ T \mapsto  t \otimes T$.      

For $\lambda \in k[G]$, write
$\cU_{G,r}^*(1 \otimes \lambda) = \sum f_i \otimes c_iT^i  \in A \otimes k[\Gar]$. 
The composition of the lower horizontal and left vertical
maps of (\ref{actt}) sends $\lambda$ to  $\sum f_i \otimes t^i \otimes c_iT^i$.   
On the other hand, the composition of the right vertical and upper
horizontal maps of (\ref{actt}) sends $\lambda$ to 
$\sum  {\rm act}^*(f_i)  \otimes c_iT^i $. 
We conclude that
$$f_i \otimes t^i =  {\rm act}^*(f_i),$$
so that $f_i$ is homogeneous of degree $i$.
\end{proof}

As a corollary (of the proof of) Proposition \ref{homog}, we see why
for $G$ infinitesimal of height $\leq r$  
the homogeneous degree of $\Theta_{G,r} \in  kG \otimes k[V_r(G)] $ is 
$p^{r-1}$ whereas the homogeneous degree of 
$\Theta_{G,r+1} \in kG \otimes k[V_{r+1}(G)]$  is $p^r$.

\begin{cor}
Let $G$ be an infinitesimal group of height $\leq r$.  Then the map
$i^*: k[V_{r+1}(G)] \to k[V_r(G)]$ of Proposition \ref{indep} is a graded isomorphism
which divides degrees by $p$.
\end{cor}

\begin{proof}
Let $\pi^*: k[V_r(G)] \to k[V_{r+1}(G)]$ be the inverse of $i^*$.
The commutativity of (\ref{compatible}) implies that we may compute the 
effect on degree of $\pi^*$ by identifying the effect on degree of the map
$p^*: k[\bG_{a(r)}] = k[t]/t^{p^r} \ \to \ k[t]/t^{p^{r+1}} = k[\bG_{a(r+1)}]$.
Yet this map clearly multiplies degree by $p$.
\end{proof}


\section{$\theta_v$ and local Jordan type}
\label{Jordan}

The purpose of this section is to exploit our universal $p$-nilpotent operator
$\Theta_G$ to investigate the local Jordan type of a finite dimensional
$kG$-module $M$.   The local Jordan type of $M$ gives much more detailed
information about a $kG$-module $M$ than the information which can be 
obtained from  the support variety (or, rank variety) of $M$.
In this section,  we work through
various examples, give an algorithm for computing local Jordan types,
and understand the effect of Frobenius twists.  Moreover, we establish 
restrictions on the rank and dimension of $kG$-modules of constant
Jordan type.

\begin{defn}
\label{local}
Let $G$ be an infinitesimal group scheme and $v \in V(G)$.  Let $k(v)$ denote
the residue field of $V(G)$ at $v$, and let 
$$\mu_v =   \cU_G \otimes_{k[V(G)]} k(v): \bG_{a(r),k(v)} \ \to \ G_{k(v)}$$
be the associated 1-parameter subgroup (for $r \geq \rm ht(G)$).    We define the
{\it local $p$-nilpotent operator at $v$}, $\theta_v$ , to be 
$$\theta_v =  \Theta_G \otimes_{k[V(G)]} k(v) = \ \mu_{v*} (u_{r-1}) \ \in \ k(v)G.$$
Equivalently,  for a $k(v)$-rational  point $v: \Spec k(v)\to V(G)$,  
we define $\theta_v = \Theta_G(v)$,
the evaluation of $\Theta_G$ at $v$: 
$$\xymatrix{\theta_v: k[G] \ar[r]^-{\Theta_G}& k[V(G)] \ar[r]& k(v)}$$ where 
the second map corresponds to the point $v$.
\end{defn}

In the special case that $G = \GL_{n(r)}$ 
for some $n > 0$, we use the alternate notation $\theta_{\ul \alpha}$ for
the local $p$-nilpotent operator at  $ \ul \alpha = (\alpha_0,
\ldots,\alpha_{r-1}) \in V(\GL_{n(r)}) \simeq N_p^{[r]}(gl_n)$:
\begin{equation}
\label{local-f}
\theta_{\ul \alpha} \ = \ \exp_{\ul \alpha, *} (u_{r-1})
\ \in \ k(\ul \alpha)\GL_{n(r)},
\end{equation}
where $ k(\ul \alpha)$ is the residue field of $\ul \alpha \in V(\GL_{n(r)})$.

Let $K$ be a field.  Then a finite dimensional $K[u]/u^p$-module $M$ is a direct
sum of cyclic modules of dimension ranging from 1 to $p$.  We may thus write
$M \simeq a_p[p] + \cdots + a_1[1]$, where $[i]$ is the cyclic $K[u]/u^p$-module
$K[u]/u^i$ of dimension $i$.  We refer to the $p$-tuple
\begin{equation}
\label{jtype}
\JType(M,u) \ = \ (a_p,\ldots,a_1)
\end{equation}
as the {\it Jordan type} of the $K[u]/u^p$-module $M$.  We also refer to $\JType(M,u)$
as the Jordan type of the $p$-nilpotent operator $u$ on $M$.

For simplicity, we introduce the following notation.

\begin{defn}
\label{local-jordan}
With notation as in Definition \ref{local}, we set
$$\JType(M,\theta_v) \ \equiv \ \JType((\mu_{v,*})^*(M_{k(v)}),u_{r-1}).$$
We refer to this Jordan type as the {\it local Jordan type} of $M$ at $v \in V(G)$.
\end{defn}

\begin{remark}
\label{rk}
Essentially by definition, the rank variety $V(G)_M$ of a finite dimensional $kG$-module
for an infinitesimal group scheme $G$ is the closed, reduced subscheme of $V(G)$
consisting of those points $v \in V(G)$ at which the local Jordan type of $M$ has some 
Jordan block of size $< p$; in other words, those $v \in V(G)$ for which
$\JType(M,\theta_v) \ \not= \ \frac{\dim M}{p} [p].$
\end{remark}

The following proposition will enable us to make more concrete and explicit the 
local Jordan type of a $kG$-module $M$ at a given 1-parameter subgroup of $G$.

\begin{prop}
\label{repr} 
Let $ \ul \alpha = (\alpha_0,\ldots,\alpha_{r-1})\in V(\GL_{n(r)})$  be an $r$-tuple of 
$p$-nilpotent pair-wise commuting matrices. Let $M$ be a $k\GL_{(r)}$-module of 
dimension $N$, and  let  
$\rho: \GL_{m(r)} \to \GL_N$ be the associated structure map.  
The $(i,j)$-matrix entry of the action of 
the local $p$-nilpotent operator 
$\theta_{\ul \alpha} \in k(\ul \alpha)\GL_{n(r)}$ of (\ref{local-f}) on $M$ equals 
the coefficient of $T^{p^{r-1}}$ of 
$$ (\exp_{\ul \alpha})^*(\rho^*X_{i,j})  \ \in \ k(\ul \alpha)[\Gar]\simeq k(\ul \alpha)[T]/T^{p^r},$$ 
where $\{ X_{i,j}, 1 \leq i,j \leq N \}$ are the matrix coordinate functions of $\GL_N$.
\end{prop}

\begin{proof}  Let $\langle m_i \rangle_{1 \leq i \leq N} $  be the basis of $M$ corresponding to the structure map $\rho$. 
The structure of $M$ as a comodule for $k[\GL_{n(r)}]$ is given by
$$
M \to M \otimes k[\GL_{n(r)}], \quad m_j \mapsto \sum_i m_i \otimes \rho^*X_{i,j},
$$
and thus the comodule structure of  $M_{k(\ul \alpha)}$ for 
$k(\ul \alpha)[\bG_{a(r)}]$ is given by 
$$
M \to M \otimes k(\ul \alpha)[\bG_{a(r)}], \quad m_j \mapsto 
\sum_i m_i \otimes \exp_{\ul \alpha}^*(\rho^*X_{i,j}).
$$
The proposition follows from the fact that $u_{r-1}: k(\ul \alpha)[\bG_{a(r)}] 
\to k(\ul \alpha)$ is given by reading off the coefficient of $T^{p^{r-1}} \in 
k(\ul \alpha)[\bG_{a(r)}]$.
\end{proof}

\begin{ex}
\label{jordan}
\vskip .1in
We investigate the local Jordan type  of the various representations 
considered in Example \ref{ex-rep}.
\vskip .1in
(1)  Consider the adjoint representation $M = \fg^{\rm  ad}$ of a $p$-restricted 
Lie algebra $\fg$ and a 1-parameter subgroup 
$$\mu_X: \bG_{a(1),K} \ \to \ \ul\fg_K, \quad \text{~ inducing ~} \ K[u]/u^p \ \to \ \cu(\fg_K)$$
sending $u$ to some $p$-nilpotent $X \in \fg_K$.   The local Jordan type of $\fg^{\rm  ad}$
at $\mu_X$ is simply the Jordan 
type of the endomorphism ${\rm ad}_X: \fg^{ad}_K \to \fg^{\rm  ad}_K$,
$$\JType(\fg^{\rm ad},\theta_X) \ = \ \JType(X).$$

\vskip .1in
(2)  Let $M =  k[\Gar]/(u_0) \simeq k[u_1, \ldots,u_{r-1}]/(u_1^p, \ldots, u_{p-1}^p)$ be a cyclic 
$k\bG_{a(r)} = k[u_0,\ldots,u_{r-1}]/(u_0^p, \ldots, u_{r-1}^p)$-module, 
and let  $\mu_{\ul a}: \Gar \to \Gar$ be a 1-parameter subgroup
for some $\ul a$ of $V(\Gar) = \bA^r$.
Then 
$$
\JType(M,\theta_{\ul a})   = \
\begin{cases}
\ p^{r-2} [p], \quad  \exists i > 0, \ a_i \not=0 \\
\ p^{r-1}[1], \quad {\rm otherwise}. \\
\end{cases}
$$

\vskip .1in
(3)  Let $G = \GL_{n(r)}$, and  let $V_n$ be the canonical $n$-dimensional rational representation of   
$\GL_{n(r)}$.  We apply Proposition \ref{repr}, observing that $\rho$ 
for $V_n$ is simply the natural inclusion $\GL_{n(r)} \subset \GL_n$.
Since \begin{equation}
\label{beta1}
 \exp_{\ul \alpha}^*(X_{i,j}) = \sum\limits_{\ell =0}
^{p^r -1} [\beta_\ell]_{i,j}t^\ell,
\end{equation}
where $\beta_\ell$ are matrices determined by $\alpha_i$ as in Proposition~\ref{coact},
we conclude 
$$\JType(V_n,\theta_{\ul \alpha}) \ = \ \JType(\alpha_{r-1}).
$$
Specializing to $r=2$, 
$$
\JType(V_n,\theta_{(\alpha_0,\alpha_1)}) \ = \ \alpha_1.
$$

\vskip .1in
(4) ``Specializing"  to $G = \SL_{2(2)}$, consider 
$\ul \alpha =(\left [\begin{array}{cc} 
c_0 & a_0 \\
b_0 & -c_0
\end{array} \right ], \left [\begin{array}{cc} 
c_1 & a_1 \\
b_1 & -c_1
\end{array} \right ] ).$  Then 
$\JType(V_2,\theta_{\ul \alpha})$ equals the Jordan type of the matrix 
$\left [\begin{array}{cc} 
c_1 & a_1 \\
b_1 & -c_1
\end{array} \right ]$.  
\end{ex}
\vskip .1 in

We extend Example \ref{jordan}(3) by considering tensor powers $V_n^{\otimes d}$ 
of the canonical rational representation of $\GL_n$ restricted to $\GL_{n(2)}$.  In this 
example, the role of both 
entries of the pair $\ul \alpha = (\alpha_0,\alpha_1)$ is non-trivial.
 
\begin{ex}
\label{two}
Consider the $N = n^d$-dimensional rational $\GL_n$-module $M= V_n^{\otimes d}$
where $V_n$ is the canonical $n$-dimensional rational $\GL_n$-module.  Let
$\rho: \GL_{n(r)} \to \GL_N$ be the representation of $M$ restricted to $\GL_{n(r)}$.
A basis of $M$ is $\{ e_{i_1} \otimes \cdots \otimes e_{i_d}; \ 1 \leq i_j \leq n \}$, where
$\{ e_i; 1 \leq i \leq n \}$ is a basis for $V_n$.   Let 
$\{ X_{i_1,j_1;\ldots,i_d,j_d}, \ 1 \leq i_t,j_t \leq n \}$
denote the matrix coordinate functions on $\GL_N$, and let $\{ Y_{s,t}, \ 1 \leq s,t \leq n \}$
denote the matrix coordinate functions of $\GL_n$.

	Then $\rho^*: k[\GL_N] \to k[\GL_{n(r)}]$ is given by  
$$X_{i_1,j_1;\ldots;i_d,j_d} \mapsto 
 \ Y_{i_1,j_1}\cdots Y_{i_d,j_d}.$$ 
Thus, 
$$(\exp_{\ul \alpha})^*(\rho^*(X_{i_1,j_1;\ldots,i_d,j_d}))  \ = \  (\exp_{ \ul \alpha})^*(Y_{i_1,j_1})
\cdots (\exp_{\ul \alpha})^*(Y_{i_d,j_d}).$$
Now, specialize to $r=2$ so that we can make this more explicit.  Then the 
coefficient of $T^p$ of $(\exp_{(\alpha_0,\alpha_1)})^*(\rho^*(X_{i_1,j_1;\ldots,i_d,j_d})) $
is 
\begin{equation}
\label{r=2}
 \sum_{k=1}^d (\alpha_1)_{i_k,j_k} + \sum\limits_{
\stackrel{
0 \leq f_k <p}
{f_1+\cdots+f_d=p}
}\frac{1}{f_1!} \cdots \frac{1}{f_d!}
((\alpha_0)^{f_1})_{i_1,j_1}\cdots ((\alpha_0)^{f_d})_{i_d,j_d}.
\end{equation}
This gives the action of $\theta_{(\alpha_0,\alpha_1)}$ on $M$.

To simplify matters even further, consider
the special case $(\alpha_0)^2 = 0$. For $1 \leq d < p$,  
 $\theta_{(\alpha_0,\alpha_1)}$ on $M$ is given by the $N \times N$-matrix  
$$\big({i_1,j_1;\ldots;i_d,j_d}\big)  \mapsto (\sum_{k=1}^d (\alpha_1)_{i_k,j_k}\big).$$
For $d = p$, the action of $\theta_{(\alpha_0,\alpha_1)}$ on $M$ is
given by the $N\times N$-matrix 
$$\big({i_1,j_1;\ldots;i_p,j_p}\big)  \mapsto (\sum_{k=1}^p (\alpha_1)_{i_k,j_k}
+ (\alpha_0)_{i_1,j_1}\cdots (\alpha_0)_{i_p,j_p}\big).$$

An analogous calculation applies to the the $d$-fold symmetric
product $S^d(V_n)$ and $d$-fold exterior product $\Lambda^d(V_n)$ of
the canonical $n$-dimensional rational $\GL_n$-module $V_n$.
\end{ex}
\vskip .1in

The proof of Proposition \ref{repr} applies equally well to prove
 the following straight-forward generalization, which
one may view as an algorithmic method of computing the ``local Jordan type"
of a $kG$-module $M$ of dimension $N$.  The required input is an
explicit description of the map on coordinate algebras $\rho^*$ given
by $\rho: G \to \GL_N$ determining the $kG$-module $M$.

\begin{thm}
\label{matrix}
Let $G$ be an infinitesimal group scheme of height $\leq r$, 
and  let $\rho: G \to \GL_N$ be a 
representation  of $G$ on a vector space $M$ of dimension $N$.   Consider
some $v \in V(G)$, and let $\mu_v: \mathbb G_{a(r), k(v)} \to G_{k(v)}$ 
be the corresponding  1-parameter subgroup of height $r$.
Then the $(i,j)$-matrix entry of the action of 
$\theta_v \in k(v)G$ on $M$ equals the coefficient of
$T^{p^{r-1}}$ of 
$$ (\mu_v)^*(\rho^*X_{i,j})  \ \in \ k(v)[\Gar],$$ 
where $\{ X_{i,j}, 1 \leq i,j \leq N \}$ are the matrix coordinate functions of $\GL_N$.
\end{thm}

As a simple corollary of Theorem \ref{matrix}, we give a  criterion for 
the local Jordan type of the $kG$-module $M$ to be trivial (i.e., equal
to $(\dim M) [1]$) at a 1-parameter subgroup $\mu_v, \ v \in V(G)$.

\begin{cor}
\label{useful}
With the hypotheses and notation of Theorem \ref{matrix}, 
$$\JType(M,\theta_v) \ = \ \JType((\mu_{v,*})^*(M_{k(v)}),u_{r-1})
\ = \ (\dim M) [1]$$
if $\deg~ (\rho \circ \mu_v)^*(X_{i,j}) < p^{r-1}$ for all $1 \leq i,j \leq N$.
\end{cor}

One means of constructing $kG$-modules is by applying Frobenius twists to 
known $kG$-modules.  Our next objective is to establish (in Proposition \ref{twist1})
 a simple relationship between the  $p$-nilpotent operator 
 $\theta_{\ul \alpha}$ 
 on a $k\GL_{n(r)}$-module
$M$ and  $\theta_{\ul \alpha}$ on the $s$-th Frobenius twist $M^{(s)}$ of $M$ 
for any
$0 \not= v \in V(\GL_{n(r)})$.

 Before formulating this relationship, we make explicit the definition of
 the Frobenius map for an arbitrary affine group scheme over $k$.
Let $G$ be an affine  group scheme over $k$ and define for any $s > 0$ the
$s^{th}$ Frobenius map $F^s: G \to G^{(s)}$ given by the $k$-linear
algebra homomorphism 
\begin{equation}
\label{fs}
F^{s*}: k[G^{(s)}] = k \otimes_{p^s}  k[G]  \ \to \ k[G],
\quad a \otimes f \mapsto a \cdot f^{p^s},
\end{equation}
where $k \otimes_{p^s}  k[G] $ is the base change of $k[G]$ along the $p^s$-power
map $k \to k$ (an isomorphism only for $k$ perfect).   If $G$ is defined over
$\bF_{p^s}$ (for example, if $G = \GL_n$), then we have a natural isomorphism 
$$\UseComputerModernTips
 \xymatrix{k[G] = k \otimes_{\bF_{p^s}} \bF_{p^s}[G] \ar[r]^-{\sim} & 
 k \otimes_{p^s}  k \otimes_{\bF_{p^s}} \bF_{p^s}[G] =  k[G^{(s)}]} 
 $$
so that $F^s$ can be viewed as a self-map of $G$.

\begin{defn}
\label{ft}
If $M$ is a $k[G]$-comodule, then the $s^{th}$ Frobenius twist $M^{(s)}$
of $M$ is the $k$-vector space $k \otimes_{p^s} M$ equipped with the comodule
structure 
$$F^{s*} \circ (k\otimes_{p^s} \nabla_M): M^{(s)} \to M^{(s)} \otimes k[G^{(s)}] \to M^{(s)} \otimes k[G].$$
If $G$ is a finite group scheme, then we shall  view $M^{(s)}$
as a $kG$-module via the map $F_*^s: kG \to kG^{(s)}$ dual to (\ref{fs}).
\end{defn}

To be more explicit, suppose the $N$-dimensional $kG$-module $M$ is given
by $\rho: G \to \GL_N$ (so that $M = \rho^*(V_N)$, where $V_N$ is the canonical
$N$-dimensional $\GL_N$-module) and assume that $G$ is defined over $\bF_{p^s}$.  
Let $\mu_v: \bG_{a(r),K} \to G_K$ be a 1-parameter subgroup, corresponding to some $v \in V(G)$.
Then the identification of $M^{(s)}$ with $(\rho \circ F^s)^*(V_N)$ implies that
\begin{equation}
\label{jt}
\JType(M^{(s)},\theta_v) \ = \ \JType(M,\theta_{F^s(v)})
\end{equation}
where $\theta_{F^s(v)} = (F^s \circ \mu_v)_* (u_{r-1})$. 

\vspace{0.2in}
Let $G = \GL_{n(r)}$, and  let $R$ be a  finitely generated commutative $k$-algebra. 
The Frobenius self-map  is  given explicitly  on the $R$-values  of $\GL_{n(r)}$ by the formula 
$$F: \alpha  \mapsto \phi(\alpha),$$
where $\phi$ applied to $\alpha \in M_n(R)$ raises each entry of $\alpha$ to the $p$-th power.
For $t$ in $\Gar(R)$, we compute  
$$
(F \circ \exp_{(\alpha_0,\ldots,\alpha_{r-1})})(t) = 
F (\exp(t\alpha_0) \exp(t^p\alpha_1) \ldots \exp(t^{p^{r-1}}\alpha_{r-1})) = 
$$
$$
\exp(t^p \phi(\alpha_0)) \exp(t^{p^2}\phi(\alpha_1)) \ldots \exp(t^{p^{r-1}} \phi(\alpha_{r-2})) =  
\exp_{(0, \phi(\alpha_0), \ldots, \phi(\alpha_{r-2}))}(t).
$$
Iterating $s$ times, we obtain the following formula for $G = \GL_{n(r)}$:
\begin{equation}
\label{frob-shift}
F^s \circ \exp_{(\alpha_0,\ldots,\alpha_{r-1})} \ = \
\exp_{(0, 0, \ldots, 0, \phi^s(\alpha_0), \ldots, \phi^s(\alpha_{r-1-s}))}
\end{equation}
where the first non-zero entry on the right happens at the $(s+1)$-st place.
 
For $G = \Gar$,  the Frobenius map $F: \Gar \to \Gar $ is given by 
raising an element 
$a \in \Gar(R)$ to the $p$-th power.   Let $\ul a = (a_0,\ldots,a_{r-1})$ be a  point in  
$V(\Gar) \simeq \mathbb A^r$, and 
let $\mu_{\ul a} :\Gar \to \Gar$ be the corresponding  1-parameter subgroup. 
For $t  \in \Gar(R)$, we have  
$\mu(t) = a_0 +a_1 t + \cdots + a_{r-1}t^{p^{r-1}}$ (see \cite[\S 1]{SFB1}). 
The following formula  is now immediate:
\begin{equation}
\label{frob-shift2}
F^s \circ \mu_{(a_0,\ldots,a_{r-1})} \ = \
\mu_{(0, \ldots, 0, a^{p^s}_0,\ldots, a^{p^s}_{r-1-s})}.
\end{equation}

Combining (\ref{jt}) and (\ref{frob-shift}), we derive the following proposition.

\begin{prop}
\label{twist1}
Let $M$ be a finite dimensional representation of $\GL_{n(r)}$ and  let 
$\ul \alpha = (\alpha_0,\ldots,\alpha_{r-1})$ be a  point   in  $V(\GL_{n(r)})$.  
Then
$$\JType(M^{(s)},\theta_{\ul \alpha}) \ = \  \JType(M,\theta_{F^s \circ \ul \alpha}),$$
where $F^s \circ \ul \alpha = (0,\ldots,0,\phi^s(\alpha_1),\ldots,\phi^s(\alpha_{r-1-s}))$.

If $M$ is a finite dimensional $k\Gar$-module, and 
$\ul a = (a_0, \ldots, a_{r-1})$ is a  point   
in $V(\Gar) \simeq \mathbb A^r$, then 
$$\JType(M^{(s)},\theta_{\ul a}) \ = \  \JType(M,\theta_{F^s \circ \ul a}),$$
where $F^s \circ \ul a = (0, \ldots, 0, a^{p^s}_0,\ldots, a^{p^s}_{r-1-s})$.
\end{prop}

Proposition \ref{twist1} has the following immediate corollary (see Remark \ref{rk}).

\begin{cor}
\label{twist2}
Let $M$ be a finite dimensional representation of $\GL_{n(r)}$.  Then
$\ul \alpha = (\alpha_0, \ldots, \alpha_{r-1}) \in V(\GL_{n(r)})_M$ if and only if
$F^s \circ \alpha = (0,\ldots,0,\phi^s(\alpha_0),\ldots,\phi^s(\alpha_{r-1-s}))  \in V(\GL_{n(r)})_{M^{(s)}}$.
\sloppy{

}
\end{cor}

The following definition introduces interesting classes of $kG$-modules which 
have special local behavior.

\begin{defn}
\label{constant}
Let $G$ be an infinitesimal group scheme and $j$ a positive 
integer less than $p$. 
A finite dimensional $kG$-module $M$ is said to be of constant $j$-rank if and 
only if 
$$ \rk(M,\theta_v^j) \ \equiv \ \rk\{\theta_v^j: M_{k(v)} \to M_{k(v)} \}$$
is independent of $v \in V(G) - \{ 0 \}$, where $\theta_v$ is the local $p$-nilpotent
operator at $v$ as introduced in Definition \ref{local}.

$M$ is said to be of constant Jordan type if and only if it is of
constant $j$-rank for all $j, \ 1 \leq j < p$.  $M$ is said to be
of constant rank if it is of constant 1-rank.
\end{defn}

\vspace{0.1in}

As we see in the following example, one can have rational 
$\GL_n$-modules of constant Jordan type when restricted to 
$\GL_{n(r)}$ of arbitrarily high degree $d$.  This should be 
contrasted with Corollary \ref{sl-deg}.

\begin{ex}
\label{det}
Consider the rational $\GL_n$-module $M = \det^{\otimes d}$, the $d^{th}$ power of the
determinant representation for some $d > 0$.  This is a polynomial representation of degree $n^d$.
The restriction of $M$ to any Frobenius kernel
$\GL_{n(r)}$ has (trivial) constant Jordan type, for the further restriction of $M$ to
any abelian unipotent subgroup of $\GL_n$ has trivial action.
\end{ex}

We shall see below that $kG$-modules of constant $j$-rank lead to interesting
constructions of vector bundles (see Theorem \ref{bundle}).  We conclude this 
section by establishing two constraints, Propositions \ref{deg-bound}
and \ref{dim-bound}, on $kG$-modules to be modules of constant rank.

We first need the following elementary lemma.

\begin{lemma}
\label{trivial}
Let $M$ be a $\Gar$-module such that  the local Jordan type at every 
$v \in V(\Gar)$ is trivial.   Then $M$  is trivial as a $k\Gar$-module. 
\end{lemma}

\begin{proof}
The action of $\Gar$ on $M$ is given by the action of $r$ commuting
$p$-nilpotent operators $\tilde u_i, \ 0 \leq i < p$ on $M$.  Moreover, for 
$\ul a = (0,\ldots,1,0,\ldots,0)$, with $1$ at the $i^{\rm th}$ spot, 
$$\JType(M,\theta_{\ul a}) \ = \ \JType(M, \tilde u_i),$$
as follows from the explicit description of $\Theta_{\Gar}$ in Example \ref{p-univ}(2). Thus, if the local Jordan type of $M$ 
is trivial at each $\ul a = (0,\ldots,1,0,\ldots,0)$,
then each $\tilde u_i$ must act trivially on $M$ and $M$ is therefore a trivial
$\Gar$-module.
\end{proof}

\begin{prop}
\label{deg-bound}
Let $\fG$ be an algebraic group generated by 1-parameter subgroups
$i: \bG_a \subset \fG$.  Let $\rho: \fG \to \GL_N$ determine a non-trivial rational representation
$M$ of $\fG$.  Let $D_i$ be the minimum of  the degrees of 
$(\rho \circ i)^*(X_{s,t}) \in k[\bG_a] = k[T]$ as $X_{s,t}$ ranges over 
the matrix coordinate functions of $\GL_N$.  Let $D = \max\{ D_i, i: \bG_a \subset \fG \}$.

If $r   >  \log_p D + 1$, then $M$ is not of constant Jordan type as a $\fG_{(r)}$-module.
\end{prop}

\begin{proof}
Because $M$ is non-trivial and $\fG$ is generated by its 1-parameter subgroups, we
conclude that $i^*M$ is a non-trivial rational $\bG_a$ representation for some 
1-parameter subgroup $i: \bG_a \subset \fG$.  The condition $r > \log_pD \geq \log_p D_i$
 implies that $i^*M$ is not $r$-twisted (i.e., of the form $N^{(r)})$.
Lemma~\ref{trivial} implies that the local Jordan type of $i^*M$ at some 1-parameter 
subgroup $\mu_v: \bG_{a(r),k(v)} \to G_{k(v)}$ is non-trivial.  On the other hand, 
Corollary~\ref{useful}  implies that the Jordan type of $i^*M$ at the identity
1-parameter subgroup $\id: \Gar \to \Gar$ is trivial provided that $r-1 > \log_p D$. 
\end{proof}

As an immediate corollary, we conclude the following.

\begin{cor}
Let $\fG$ be an algebraic group generated by 1-parameter subgroups and $M$ a 
non-trivial rational representation of $\fG$.  Then for $r >> 0$, $M$ is not of 
constant Jordan type as a $k\fG_{(r)}$-module.
\end{cor}

As an explicit example of Proposition \ref{deg-bound}, we 
obtain the following corollary (which should be contrasted 
with Example \ref{det}).

\begin{cor}
\label{sl-deg}
Let $M$ be a non-trivial polynomial representation of $\SL_n$ of degree $D$.  If $r > \log_pD +1$,
then $M$ is not a $k\SL_{n(r)}$-module of constant rank.
\end{cor}

The following lemma, which is a straightforward application of the Generalized 
Principal Ideal Theorem (see \cite[10.9]{Eis}), 
shows that the dimension of a non-trivial module of constant rank 
of $\Gar$ cannot be ``too small" compared to $r$. 
\begin{lemma}
\label{pit}
Let $M$ be a finite dimensional $\Gar$-module. If $M$ is a non-trivial 
$\Gar$-module of constant rank, 
then the following inequality holds:
\begin{equation}
\label{ineqq}
\dim_k M \geq \sqrt{r}
\end{equation}

\end{lemma}
\begin{proof} By extending scalars if necessary we may assume that $k$ is 
algebraically closed.  Let $m = \dim_k M$.
Let $k\Gar = k[u_0, \dots. u_{r-1}]/(u_0^p, \dots, u_{r-1}^p)$,  
let $K = k(s_0, \dots, s_{r-1})$ where $s_i$ are independent variables, 
and let $\alpha_K: K[t]/t^p \to K\Gar$ be 
 a map of $K$--algebras defined by $\alpha_K(t) = s_0u_0 + \dots + s_{r-1}u_{r-1}$. 
 Choose a $k$-linear basis of $M$, and let 
 $A(s_0, \ldots, s_{r-1})$ be a nilpotent matrix in $M_m(k[s_0, \ldots, s_{r-1}])$ 
 representing the action of $\alpha_K(t)$ on $M_K$.  
 Let $I_n(A(s_0, \ldots, s_{r-1}))$ denote the ideal generated by all $n\times n$ 
 minors of $A(s_0, \ldots, s_{r-1})$.   By \cite [10.9]{Eis}, 
 the codimension of any minimal prime over $I_n(A(s_0, \ldots, s_{r-1}))$ is at 
 most $(m-n+1)^2$. 
 
 Assume that (\ref{ineqq}) does not hold, that is, $m < \sqrt r$. 
 Hence, $(m-n+1)^2 < r $ for  any $1\leq n \leq m$. 
 The variety of $I_n(A(s_0, \ldots, s_{r-1}))$ is a subvariety 
 inside $\Spec k[s_0, \ldots, s_{r-1}] \simeq \mathbb A^r$ which has dimension $r$. 
 Since the codimension  of the variety of $I_n(A(s_0, \ldots, s_{r-1}))$ 
 is at most $(m-n+1)^2$, we conlude that the dimension is at 
 least $r - (m-n+1)^2 \geq 1$.
 Hence, the minors of dimension 
 $n \times n$ have a common non-trivial zero.    Taking $n = 1$, we 
 conclude that $A(b_0,\ldots,b_{r-1})$  is a zero  matrix for some non-zero specialization 
 $b_0,\ldots,b_{r-1}$ of $s_0,\ldots,s_{r-1}$.  Consequently,
 $M$ is trivial at the $\pi$-point of $\Gar$ corresponding to  $b_0,\ldots,b_{r-1}$.
 Since $M$ is non-trivial, 
 Lemma \ref{trivial} implies that  $M$ is not a module of constant rank. 
 \end{proof}

\noindent
As an immediate corollary, we provide an additional necessary condition
for a $k\fG_{(r)}$-module to have constant rank. 

\begin{prop}
\label{dim-bound}
Let $\fG$ be a (reduced) affine algebraic group  and  $M$  be a rational
representation of $\fG$.   
Assume that $\fG$ admits a 1-parameter subgroup 
$\mu: \bG_{a(r)} \to  \fG$    such that $\mu^*(M)$ 
is a non-trivial $k\Gar$-module. 
If $r \ \geq \ (\dim M)^2 +1$, then
$M$ is not a $k\fG_{(r)}$-module of constant rank.
\end{prop}

\section{$\pi$-points and $\PG$}
\label{proj} 
In a series of earlier papers, we have considered $\pi$-points for a finite 
group scheme $G$ (as recalled in Definition \ref{pi}) and investigated 
finite dimensional $kG$-modules $M$ using the ``Jordan type of $M$"
at various $\pi$-points.  
In particular, in \cite{FPS}, we verified that
this Jordan type is independent of the equivalence class of the $\pi$-point
provided that either the $\pi$-point is generic or the Jordan type of $M$
at some representative of the equivalence class is maximal.

As we recall below, whenever $G$ is an
infinitesimal group scheme, then  the $\pi$-point 
space $\Pi(G)$ of equivalence classes of $\pi$-points is essentially the projectivization
of $V(G)$.   The purpose  of the first half of this section is to relate the discussion
of the previous section concerning the local Jordan type of a finite $kG$-module
to our earlier work formulated in terms of $\pi$-points for general finite
group schemes.

One special aspect of an infinitesimal group scheme $G$ is that equivalence
classes of $\pi$-points of $G$ have canonical (up to scalar multiple) 
representatives.

Unless otherwise specified (as in Definition \ref{pi} immediately below), $G$
will denote an infinitesimal group scheme over $k$, and $V(G)$ will denote $V_r(G)$
for some $r \geq \rm ht(G)$.  Throughout this section we  assume  that $\dim V(G)\geq 1$, 
and  work with $\PG = \Proj k[V(G)]$. We note that if $\dim V(G) = 0$, then Theorem~\ref{iso} implies that the projective resolution of the trivial module $k$ is finite. Since $kG$ is self-injective, this further implies that $k$ is 
projective. Hence, $kG$ is semi-simple, and does not have any $\pi$-points (see, for example, \cite[\S 2]{FP1}).   

\begin{defn}
\label{pi}
(see \cite{FP2})
Let $G$ be a finite group scheme.  

\begin{enumerate}
\item
A $\pi$-point of $G$ is a (left)
flat map of $K$-algebras $\alpha_K: K[t]/t^p \to KG$ for some field extension 
$K/k$ with the property that there exists a unipotent abelian closed subgroup scheme
$i: C_K \subset G_K$ defined over $K$ such that $\alpha_K$ factors through 
$i_*: KC_K \to KG_K = KG$.
\item
If $\beta_L: L[t]/t^p \to LG$ is another $\pi$-point of $G$, then $\alpha_K$ 
is said to be a {\it specialization} of $\beta_L$ , written 
$\beta_L \downarrow \alpha_K$, provided that 
for any finite dimensional $kG$-module $M$, $\alpha_K^*(M_K)$ being free  as $K[t]/t^p$-module 
implies that $\beta^*_L(M_L)$ is free as $L[t]/t^p$-module.
\item
Two $\pi$-points $\alpha_K: K[t]/t^p \to KG, ~ \beta_L: L[t]/t^p \to LG$ 
are said to be {\it equivalent},  written $\alpha_K \sim \beta_L$, if 
$\alpha_K \downarrow \beta_L$ and $\beta_L \downarrow \alpha_K$.
\item
A $\pi$-point of $G$, $\alpha_K: K[t]/t^p \to KG$, is said to be {\it generic} if there
does not exist another $\pi$-point $\beta_L: L[t]/t^p \to LG$ which specializes to $\alpha_K$  but is not
equivalent to $\alpha_K$.
\item
If $M$ is a finite dimensional $kG$-module and $\alpha_K: K[t]/t^p \to KG$
a $\pi$-point of $G$, then the Jordan type of $M$ at $\alpha_K$ is by
definition the Jordan type of $\alpha_K^*(M_K)$ as $K[t]/t^p$-module.
\end{enumerate}
\end{defn}

Because the group algebra of a finite group scheme is always faithfully
flat over the group algebra of a subgroup scheme (see \cite[14.1]{W}),
the condition on a flat map $\alpha_K:K[t]/t^p \to KG$ is equivalent to the existence of a factorization  
$i_* \circ \alpha_K^\prime$ with $\alpha_K^\prime: K[t]/t^p \to KC_K$ flat.

\begin{defn}
Let $G$ be an infinitesimal group scheme, and let $v \in V(G)$ be the point associated to
the 1-parameter subgroup $\mu_v: \bG_{a(r),k(v)} \to G_{k(v)}$.  Then
the $\pi$-point of $G$ associated to $v$ is  
$$
\mu_{v,*} \circ \epsilon: k(v)[u]/u^p \ \to \ k(v)G,
$$
where $\epsilon: k(v)[u]/u^p \to k(v)\bG_{a(r),k(v)}$ is as defined in (\ref{epsilon}).
\end{defn}

The following theorem is a complement to Theorem \ref{iso}, revealing that
spaces of (equivalence) classes of $\pi$-points are
very closely related to (cohomological) support varieties. 

\begin{thm}(\cite[7.5]{FP2})
\label{PI} 
Let $G$ be an finite group scheme.  Then the set of equivalence classes
of $\pi$-points, denoted $\Pi(G)$, can be given a scheme structure, which is 
defined in terms of the representation theory
of $G$.   Moreover, there is an
isomorphism of schemes
$$
\Proj \,\bH(G,k) ~ \simeq ~ \Pi(G).
$$

If $G$ is an infinitesimal group scheme so that $\bH(G,k)$
is related to $k[V(G)]$ as in Theorem \ref{iso}, then the resulting homeomorphism
\begin{equation}
\label{relate}
\UseComputerModernTips
 \xymatrix{\bP(G) \equiv \Proj k[V(G)] \ \ar[r] & \ \Proj \HHH^{\bu}(G,k) \ \simeq \ \ \Pi(G)}
\end{equation}
is given on points by sending $ x \in \bP(G)$ to 
the equivalence class of the $\pi$-point $ \mu_{v,*} \circ \epsilon$
for any
$v \in V(G) \backslash \{ 0 \}$ projecting to $x$.  In particular, 
equivalence classes of {\it generic} $\pi$-points of $G$ are represented by  
$(\mu_{v,*} \circ \epsilon)$ as $v \in V(G)$ runs through the (scheme-theoretic) 
generic points of $V(G)$.

Furthermore, for any finite dimensional $kG$-module $M$, (\ref{relate})
restricts to a homeomorphism of subvarieties
$$
\bP(G)_M \ \simeq \ \Pi(G)_M,
$$

where $\bP(G)_M = \Proj (k[V(G)_M])$, and $\Pi(G)_M$ consists of those equivalence classes of $\pi$-points 
$\alpha_K$ of $G$ such that $\alpha_K^*(M_K)$ is not free (as a $K[u]/u^p$-module).
\end{thm}

Generic $\pi$-points are particularly important when developing invariants of 
representations.  
The following corollary of Theorem \ref{PI} gives an explicit
set of representatives of equivalence classes of generic $\pi$-points of $G$.

\begin{prop} 
\label{gen-list}
Let $G$ be an infinitesimal group scheme with
universal 1-parameter subgroup $\UseComputerModernTips
 \xymatrix{\cU_G : \bG_{a(r),k[V(G)]} \ \ar[r] & \  G_{k[V(G)]}}$. 
For each minimal prime ideal $\cp_i$ of $k[V(G)]$,  let $K_i$ denote the field of 
fractions of $k[V(G)]/\cp_i$.  Then the compositions
$$(\cU_{G,*}  \otimes_{k[V(G)]} K_i )\circ \epsilon: 
K_i[u]/u^p \to K_iG$$
(sending $u$ to $\theta_{K_i}$)
are non-equivalent representatives of the equivalence classes of generic
$\pi$-points of $G$.  
\end{prop}

To obtain vector bundles, we require the following wll known, elementary observation about
commutative graded algebras.

\begin{lemma}
\label{twist}
Let $A$ be a finitely generated commutative, graded $k$-algebra with homogeneous 
generators whose degrees divide $d$ and let $X = \Proj A$. Then $\cO_X(d)$ is a
locally free sheaf of  rank 1 on $X$.
\end{lemma}

\begin{proof} Let $\{ f_i \}$ be a finite set of homogeneous generators of $A$, set
$d_i$ equal to the degree of $f_i$, and choose $d$ such that each $d_i$ divides $d$.  
Set $U_i = X - Z(f_i)$,
an affine open subset of $X$ with coordinate algebra $(A_{f_i})_0$, the degree zero subalgebra
of the localization of $A$ obtained by inverting $f_i$.  Then multiplication by $f_i^{d/d_i}$
induces an isomorphism $(\cO_X)_{|U_i} \stackrel{\to}{\to} (\cO_X(d))_{|U_i}$.  Thus 
the restriction of $\cO_X(d)$ to each $U_i$ of the open covering $\{ U_i \}$ is free
of rank 1.
\end{proof}

Let $G$ be an infinitesimal group scheme of height $\leq r$ and recall from Proposition~\ref{homog} that
$\Theta_{G,r} \in kG \otimes k[V_r(G)]$ is homogeneous of degree $p^{r-1}$.  

\begin{defn}
Let $G$ be an infinitesimal group scheme of height $\leq r$ and let 
$M$ be a finite dimensional
$kG$-module.  Then we denote by 
\begin{equation}
\label{theta-tilde}
\UseComputerModernTips
\xymatrix{\wt\Theta_G: \cM \equiv  M  \otimes  \cO_{\PG} \ \ar[r] & \  
\cM(p^{r-1}) \equiv  M \otimes \cO_{\PG}(p^{r-1} )
}\end{equation}
the associated homomorphism of (locally free) coherent 
$\cO_{\PG}$-modules determined by  the action of $\Theta_{G,r}
\in kG \otimes k[V_r(G)] $. 

We denote by 
\begin{equation}
\label{theta-twist}
\UseComputerModernTips
\xymatrix{\wt \Theta_G(n): \cM(n) \ar[r]& \cM(p^{r-1}+n)
}\end{equation}
the map obtained by tensoring (\ref{theta-tilde}) with $\cO_{\bP(G)}(n)$.

For any point $x\in \PG$, we use the notation
$$M_{k(x)} =  \cM  \otimes_{\cO_{\bP(G)}} k(x)$$ 
for the fiber of the coherent sheaf $\cM$ at $x$.  Here,
we have identified $k(x)$ with the
residue field of the stalk $\cO_{\PG,x}$.

\end{defn}

 \begin{prop}
 \label{indep-s}
Let $G$ be an infinitesimal group scheme of height $\leq r$, and let
$M$ be a finite dimensional $kG$-module.
 For any $v, \ v^\prime \in V(G)$ projecting to the same $x \in 
 \PG$, we have $$ \Im \{\theta_v: M_{k(v)} \to M_{k(v)} \} \ \simeq \  
 \Im \{\theta_{v^\prime}: M_{k(v^\prime)} \to M_{k(v^\prime)}\}\ $$
 and similarly for kernels.
  \end{prop}
 
 \begin{proof}
 This is essentially proved in \cite[6.1]{SFB2}.
 \end{proof}

In the next section, we shall be particularly interested in 
kernels and images of $\wt \Theta_G$.
The following proposition relates the fibers of the kernel and
image of the global $p$-nilpotent operator $\wt \Theta_G$ at a point
$x \in \bP(G)$ with the kernel and image of the 
local $p$-nilpotent operator $\theta_v$ on $M\otimes k(v)$ for 
$v$ representing $x$.

\begin{prop}
\label{sec}
Let $G$ be an infinitesimal group scheme of height $\leq r$, let
$M$ be a finite dimensional $kG$-module, let $\cM = M \otimes \cO_{\PG}$, and let 
$s \in \Gamma(\PG,\cO_{\PG}(p^{r-1}))$ be a non-zero 
global section with zero locus $Z(s) \subset \PG$.
Set $U = \PG \backslash Z(s)$.
Then there is a well defined endomorphism (depending upon $s$)
$$\wt\Theta_G/s : \cM_{|U} \ \to \ \cM_{|U}.$$

Moreover, 
the image and kernel of the induced map $\theta_x/s: M_{k(x)} \to M_{k(x)}$ 
on fibers at $x \in U\subset \PG$  is independent of $s$ and satisfies 
\begin{equation}
\label{imm}
   \Im \{\theta_x/s: M_{k(x)} \to M_{k(x)} \} \simeq \Im \{\theta_v: M_{k(v)} \to M_{k(v)}\}
  \end{equation}
and
\begin{equation}
\label{kerr}
\Ker \{\theta_x/s: M_{k(x)} \to M_{k(x)}\} \simeq \Ker \{\theta_v: M_{k(v)} \to M_{k(v)}\}
\end{equation}
for any $v \in V(G)\backslash \{0 \}$ that projects  onto $x$.
\end{prop}

\begin{proof}  Let $X =\PG$. Let 
$\frac{1}{s} \in \cO_X(-p^{r-1})(U)$ satisfy 
$$s \otimes \frac{1}{s} = 1 \ \in \cO_X(p^{r-1})(U)  \otimes_{\cO_X(U)} \cO_X(-p^{r-1})(U) \ 
\simeq \ \cO_X(U).$$ 
Then we define 
$$\wt\Theta_G/s \equiv (\wt\Theta_{G})_{|U} \otimes \frac{1}{s}: \cM_{|U} \to 
 \cM(p^{r-1})_{|U}  \otimes \cO_X(-p^{r-1}))_{|U} \ \simeq \ \cM_{|U}.$$  

Let  $A = k[V(G)]$.   Then $U$ is an affine subscheme of $\PG$ determined by the $0$-degree  elements of  
$A\left[\frac{1}{s}\right]$, $U \simeq \Spec A\left[\frac{1}{s}\right]_{0}$. Via this identification, $\cM_{|U}$ 
corresponds to the free $A\left[\frac{1}{s}\right]_0$-module $M \otimes A\left[\frac{1}{s}\right]_0$. 
The fiber $M_{k(x)}$ of 
$\cM_{|U}$ at the point $x$ is naturally identified with the  fiber of  $M \otimes A\left[\frac{1}{s}\right]_0$ at $x$. Since $\Theta_G: M \otimes k[V(G)] \to M \otimes k[V(G)]$ is homogeneous of degree $p^{r-1}$ and $s \in k[V(G)]$ is homogeneous of the same degree, the operator  $\Theta_G \otimes \frac{1}{s}$ is well-defined on  
$M \otimes A\left[\frac{1}{s}\right]_0$ and corresponds to the operator $\wt\Theta_G/s$ on $\cM_{|U}$.  
Hence, $\theta_x/s: M_{k(x)} \to M_{k(x)}$ is identified with 
$$(\wt\Theta_G \otimes \frac{1}{s}) \otimes k(x) = \frac{\Theta_G}{s}(x): M_{k(x)}\to M_{k(x)}.$$
Let $v \in V(G)$ be any point projecting onto $x$. We have $k(x) = k(v)$. By Def.~\ref{local}, the map $\theta_v: M_{k(v)} \to M_{k(v)}$ is given by 
$$\Theta_G(v): M_{k(v)} \to  M_{k(v)}. $$ 
We observe that $\frac{\Theta_G}{s}(x) = \frac{\Theta_G(v)}{s(v)}$ for any $v$ projecting onto $x$ (and, in particular, is independent of the choice of $v$). 
Therefore,
\begin{equation}
\label{identt}
\theta_x/s \ = \  \frac{\theta_v}{s(v)}.
\end{equation}
The equalities(\ref{kerr}) and (\ref{imm}) now follow.
\end{proof}

\begin{remark}
For a finite group $G$, there is no natural choice of $\pi$-point representing a 
typical equivalence class $x \in \Pi(G) \ \simeq \ \Proj \,\bH(G,k)$ of $\pi$-points.   
As seen in elementary examples \cite[2.3]{FPS}, the Jordan type of a $kG$-module $M$ typically 
can be different for two equivalent $\pi$-points representing the same point
$x \in \Pi(G)$.
\end{remark}

\begin{remark} Proposition \ref{sec} immediately generalizes to  $\wt \Theta^j_G$ 
for any $1 \leq j \leq p-1$. Thus, we have  
the following isomorphisms  for  any $x \in X=\PG$, $v  \in V(G)$ projecting  onto $x$, and a global section 
$s$ of $\cO_X(jp^{r-1})$ such that $s(x) \not = 0$:
$$\Im\{ (\theta_x/s)^j: M_{k(x)} \to M_{k(x)} \} \simeq \Im\{ \theta_v^j: M_{k(v)} \to M_{k(v)} \} \simeq $$
$$\Im\{  \wt\Theta_G^j  \otimes_{\cO_X} k(x):  \cM \otimes_{\cO_X} k(x) \to  \cM(jp^{r-1})\otimes_{\cO_X} k(x) \},$$
and similarly  for kernels.
\end{remark}
\noindent
In what follows, we shall use the following abbreviations:
$$\Im \{ \wt\Theta_G^j,\cM \} \ \equiv \ \Im \{ \wt\Theta_G^j(-jp^{r-1}): 
\cM(-jp^{r-1})   \to \cM \},
$$
\begin{equation}
\label{image}
\Im\{ \theta_x^j,M_{k(x)} \} \ \equiv \ \Im\{ (\theta_x/s)^j: M_{k(x)} \to M_{k(x)} \},
\end{equation}
$$
\Ker \{ \wt\Theta_G^j,\cM \} \ \equiv \ \Ker \{ \wt\Theta_G^j: 
\cM \to \cM(jp^{r-1})  \},
$$
$$
\Ker \{ \theta_x^j ,M_{k(x)}\} \ \equiv \ \Ker \{ (\theta_x/s)^j: M_{k(x)} \to M_{k(x)}\}.
$$
$$
\Coker \{ \wt\Theta_G^j,\cM \} \ \equiv \ \cM / \Im \{ \wt\Theta_G^j,\cM \},
$$
$$
\Coker \{ \theta_x^j ,M_{k(x)}\} \ \equiv \ \Coker \{ (\theta_x/s)^j: M_{k(x)} \to M_{k(x)}\}.
$$
Note that both $\Ker$ and $\Im$  are subsheaves of the free sheaf $\cM$, and $\Coker$ is a quotient sheaf of $\cM$.

\vspace{0.1in}
\noindent
We shall verify in Theorem \ref{equiv} that a necessary and sufficient 
condition on a finite dimensional $kG$-module $M$  for 
$\Im \{ \wt\Theta_G^j,\cM \} $ (and thus $\Ker \{ \wt\Theta_G^j,\cM \} $) to be 
an algebraic vector bundle on $X$ is that $M$ be a module of constant $j$-type.

The following proposition is given in \cite[5 ex.5.8]{Har} without proof.

\begin{prop}
\label{test}
Let $X$ be a reduced scheme 
and $\wt M$ a coherent $\cO_X$-module. Then 
$\wt M$  is locally free if and only if $\dim_{k(x)}( \wt M \otimes_{\cO_X} k(x))$ 
depends only upon the connected component of $x$ in $\pi_0(X)$.
\end{prop}

\begin{proof}   Assume that the function $x \mapsto \dim_{k(x)}(\wt M \otimes_{\cO_X} k(x))$ is constant  
on each connected component of $X$.   To prove that  $\wt M$ is locally  free it suffices to 
assume that $X$ is local so that  $X = \Spec R$  for some reduced local commutative ring,  
and that $M$ is a finite $R$-module (corresponding to the coherent sheaf $\wt M$) with the property that 
$\dim_{k(p)}( M \otimes_R k(p))$ is independent of the
prime $p \subset R$.  To prove that $M$ is free, we choose some surjective
$R$-module homomorphism  $g: Q \to M$  from a free $R$-module $Q \simeq R^n$
with the property that
$\ol g:  Q   \otimes_R R/\m\ \to \  M \otimes_R R/\m$ is an isomorphism where
$\m \subset R$ is the maximal ideal.  Then $g$ is surjective by Nakayama's
lemma.  By assumption, $g$  induces an isomorphism   after specialization to any prime $\p \subset R$: 
$Q \otimes_{R_\p} k(\p) \simeq M \otimes_{R_\p} k(\p)$.  Hence, 
$ Q_{\p}/\p Q_{\p} \simeq M_{\p}/\p M_{\p}$.  We conclude that if $ a \in \ker g$, then $a  \in \p Q_{\p} \cap Q$. 
Since this happens for any prime ideal, we further conclude that $\ker g \subset (\bigcap\limits_{\p \in \Spec R} \p Q_{\p}) \cap Q$. 
Recall that  $Q$  is a  free module so that $Q \simeq R^n$.  We  get   $(\bigcap \p Q_{\p}) \cap Q = (\bigcap \p R^n_{\p}) \cap R^n = 
((\bigcap \p R_{\p}) \cap R)^n = (\bigcap \p R)^n = 0$  since $R$  is reduced. 
\end{proof}

We shall find it convenient to ``localize" the notion of a $kG$-module of
constant $j$-rank given in Definition \ref{constant} as follows.

\begin{defn}
Let $G$ be an infinitesimal group scheme,
and let $M$ be a finite dimensional $kG$-module.
For any open subset $U \subset \PG$,
$M$ is said to be of constant $j$-rank when restricted to $U$
if $\rk _{k(x)}((\theta_x/s)^j: M_{k(x)} \to M_{k(x)})$ is independent of $x \in U$.
\end{defn}

Our next theorem emphasizes  the local nature of the concept of constant
$j$-rank.

\begin{thm}
\label{equiv}
Let $G$ be an infinitesimal group scheme, let $M$ be a finite 
dimensional $kG$-module, and  let $X = \bP(G)$.  
Let $U \subset X$ be a connected open subset, and 
$\wt\Theta_U^j : \cM_{|U} \to \cM(jp^{r-1})_{|U}$ be the restriction
to $U$
of the $j^{th}$  iterate of  $\wt\Theta_G$  on $\cM = M \otimes \cO_X $ as given in (\ref{theta-tilde}).
Then
the following are equivalent for some fixed $j, \ 1 \leq j < p$:
\begin{enumerate}
\item
$\Im \{\wt\Theta_U^j,\cM_{|U} \}$ is a locally free, coherent $\cO_U$-module.
\item
$  \Im\{ \wt\Theta_G^j,\cM\} \otimes_{\cO_X} k(x)$ has dimension independent of $x \in U$.
\item 
$\Im\{ \theta_x^j,M_{k(x)}\} \ \simeq \  \Im \{ \wt\Theta_G^j,\cM\} \otimes_{\cO_X} k(x) , 
\ \forall \ x \in U$
\item
$M$ has constant $j$-rank when restricted to $U$.
\end{enumerate}
Moreover, each of these conditions implies that 

\begin{enumerate}

\item[(5)] $\Coker  \{\wt\Theta_U^j,\cM_U \}$ is a 
locally free, coherent $\cO_U$-module.

\item[(6)] $\Coker \{ \theta_x^j,M_{k(x)}\} \ \simeq \  
\Coker\{ \wt\Theta_G^j,\cM\} \otimes_{\cO_X} k(x), \ \forall \ x \in U.$

\item[(7)] $\Ker  \{\wt\Theta_U^j,\cM_U \}$ is a 
locally free, coherent $\cO_U$-module.

\item[(8)] $\Ker \{ \theta_x^j,M_{k(x)}\} \ \simeq \  
\Ker\{ \wt\Theta_G^j,\cM\} \otimes_{\cO_X} k(x), \ \forall \ x \in U.$

\end{enumerate}

\end{thm}

\begin{proof}
Clearly,  (1) implies (2), whereas Proposition \ref{test} implies that
(2) implies (1).  

If we assume (1), we obtain a locally
split short exact sequence of coherent $\cO_U$-modules
\begin{equation}
\label{split}
0 \ \to \ \Ker\{ \wt\Theta_U^j,\cM_{|U} \} \ \to \  \cM_{|U} \ \to \ 
\Im\{ \wt\Theta_U^j,\cM_{|U} \} \to 0.
\end{equation}
In particular, $ \Ker\{ \wt\Theta_U^j,\cM_{|U} \}$ is a locally free, 
coherent $\cO_U$-module.
Locally on $U$,  $\wt\Theta_U^j$ on $\cM_{|U}$ is isomorphic to the projection
$$pr_2: \Ker\{ \wt\Theta_U^j,\cM_{|U} \}  \oplus  
\Im\{ \wt\Theta_U^j,,\cM_{|U} \}  \ 
\to \  \Im\{ \wt\Theta_U^j,,\cM_{|U} \} .$$
Since $\theta_x^j$ is the base change via $\cO_U \to k(x)$ of $\wt\Theta_U^j$,
$\theta_x^j$ can be identified with the base change of
this projection and thus we may conclude (3).

Let us now assume (3).  A simple argument using Nakayama's
Lemma as in the proof of Proposition \ref{test} implies that the function
$x \mapsto \Im \{\theta_x^j, M_{k(x)}\}$
is lower semi-continuous on $U$ whereas the function
$x \mapsto   \Im\{ \wt\Theta_G^j,\cM\} \otimes_{\cO_X} k(x)$
is upper semi-continuous
on $U$.  Thus, we conclude that each of these functions
is constant (since $U$ is connected), thereby concluding (2).  

Since $\rk\{(\theta_x/s)^j \} \ = \ \dim_{k(x)}(\Im\{ \theta_x^j,M_{k(x)} \})$, 
(2) and (3) imply (4).

Observe that 
if $f: V \to V$ is an endomorphism of a finite dimensional vector space
then $\dim\{\Coker f\} = \dim\{\Ker f \}$. The assumption that 
the $kG$-module $M$ has constant rank (i.e., (4)) implies that 
$$\dim_{k(x)}(\Coker\{ \theta_x^j,M_{k(x)} \})\  = \ 
\dim_{k(x)}(\Ker\{ \theta_x^j,M_{k(x)} \})$$ 
is independent of $x \in U$.  Hence, Proposition \ref{test} implies (5).
The right exactness of  $(-) \otimes_{\cO_X} k(x)$ applied to 
$$
\xymatrix{
\cM(-jp^{r-1}) \ar[rr]^-{\wt\Theta_G^j(-jp^{r-1})}  && \cM \ar[r]& 
\Coker\{ \wt\Theta_G^j,\cM  \} \ar[r] &0 }
$$
implies (6).

Under the assumption of (4),
we obtain a locally split short exact sequence of coherent $\cO_U$-modules
$$ 0 \to \Im\{  \wt\Theta_U^j,\cM_{|U}  \} \to 
\cM_{|U} \to \Coker\{ \wt\Theta_U^j,\cM_{|U}  \} \to 0, $$
so that $ \Im\{  \wt\Theta_U^j,\cM_{|U}  \}$ is a locally free, coherent $\cO_U$-module.
Now, using the short exact sequence  of coherent $\cO_U$-modules
$$0 \to \Ker \{ \wt\Theta_U^j,\cM_{|U}  \} \ \to  \ \cM_{|U} \ \to \
\Im\{  \wt\Theta_U^j,\cM_{|U}  \} \to 0,$$
we conclude that (4) implies (7) (i.e., that $\Ker \{ \wt\Theta_U^j,\cM_{|U}  \} $ is locally
free).  Since the short exact sequence
(\ref{split}) is locally split, applying $(-) \otimes_{\cO_X} k(x)$ to (\ref{split}) for any
$x \in U$ yields a short exact sequence, thereby implying (8).
\end{proof}


\section{Vector bundles for modules of constant $j$-rank}

In this section, we initiate the study of algebraic
 vector bundles associated to $kG$-modules 
of constant $j$-rank as defined in \ref{constant}.  
Our constructions have two immediate consequences.
The first is that certain $kG$-modules with the same ``local Jordan type" 
have non-isomorphic associated vector bundles, so that the isomorphism
classes of these vector bundles serve as a new invariant.  The second is
that our construction yields vector bundles on the highly non-trivial projective 
schemes $\PG$.

The reader will find formulas for the ranks of bundles considered, criteria
for non-triviality of bundles, a criterion for producing line bundles, 
a relationship to duality, and another test for the projectivity of
$kG$-modules.  We also investigate the dimension of global sections
of various bundles.

As in \S \ref{proj}, we assume that  $\dim V(G)\geq 1$ 
throughout this section. 

\vspace{0.2in}

The special case in which $U = \PG$  of Theorem \ref{equiv} is the following
assertion
that $kG$-modules of constant $j$-rank determine algebraic vector bundles over
$\PG$. 
\begin{thm}
\label{bundle}
Let $G$ be an infinitesimal group scheme, let $M$ be a finite 
dimensional $kG$-module, and  let $\cM = M \otimes \cO_{\PG}$ 
be a free coherent sheaf on $\PG$.
Then $M$ has constant $j$-rank if and only
if $\Im \{\wt\Theta_G^j,\cM \}$ is an 
algebraic vector bundle on $\PG$.

Consequently, if $M$ has constant $j$-rank, then 
$\Ker \{\wt\Theta^j_G, \cM \}$ also is an algebraic vector 
bundle on $\PG$.
\end{thm}

\begin{remark}
\label{nont}
Unless $M$ is trivial as a $kG$-module, 
$\Ker \{ \Theta_G:  M \otimes k[V(G)]  \ \to \ M \otimes k[V(G)]\}$ is not projective
as a $k[V(G)]$-module, since the local $p$-nilpotent operator $\theta_0$
at $0 \in V(G)$ is the 0-map.
\end{remark}

We observe the following elementary functoriality of this construction.

\begin{prop}
\label{funct}
Let $i: H \to G$ be an embedding of infinitesimal group schemes, let $M$
be a finite dimensional $kG$-module, and let $N$ be the restriction of $M$
to $kH$.  Let $\cM =  M\otimes\cO_{\bP(G)} $, and $\cN =  N \otimes\cO_{\bP(H)}$. 
Then for any $j, 1 \leq j <p$, there are natural isomorphisms of coherent sheaves on $\bP(H)$,
where $f: \bP(H) \to \bP(G)$ is induced by $i$:
$$f^* \Im\{ \wt \Theta_G^j,\cM \} \ \simeq \ \Im\{ \wt \Theta_H^j,\cN \}$$
$$f^* \Ker\{ \wt \Theta_G^j,\cM \} \ \simeq \ \Ker\{ \wt \Theta_H^j,\cN \}.$$
\end{prop}

\begin{proof}
The statement  follows immediately  from the commutativity of the diagram 
\begin{equation}\label{commutative}
\UseComputerModernTips
\xymatrix{ M  \otimes \cO_{\bP(G)}\ar[r]^-{\wt\Theta^j_G}\ar[d]_{f^*} &  
M \otimes \cO_{\bP(G)}(jp^{r-1}) \ar[d]_{f^*}\\
 N \otimes\cO_{\bP(H)}\ar[r]^-{\wt\Theta^j_H} & N \otimes \cO_{\bP(H)}(jp^{r-1}). 
}
\end{equation}
The diagram is commutative by Proposition \ref{pullback3}.

\end{proof}

The following Corollary will be used later in \S 6.

\begin{cor}
\label{ext-prod}
Let $G_1, \ G_2$ be  infinitesimal group schemes, let $G = G_1 \times G_2$, and let
$f: \bP(G_1) \to \bP(G)$ 
be the natural embedding of varieties induced by the embedding of group schemes $i: G_1 \hookrightarrow G$. Let $M_1, \ M_2$ be $kG_1, \ kG_2$ modules of dimensions 
$m_1, m_2$ respectively.   Then for any $j, \ 1 \leq j \leq p$, 

\begin{equation}
\label{i1}
f^*(\Ker\{\wt \Theta^j_{G}, \cM_1 \boxtimes \cM_2 \}) \ \simeq \ 
\Ker\{ \wt \Theta^j_{G_1}, \cM_1 \}^{\oplus m_2}
\end{equation}

\vspace{0.1in} 
\noindent
Here, $\cM_1 = M_1 \otimes \cO_{\bP(G_1)}, \ \cM_2 = M_2 \otimes \cO_{\bP(G_2)}$,
and $\cM_1 \boxtimes \cM_2 \simeq (M_1 \otimes M_2) \otimes \cO_{\bP(G)}$.
\end{cor}

\begin{proof}
By Proposition~\ref{funct}, it suffices to observe that $(M_1 \otimes M_2)\downarrow_{G_1} \simeq M_1^{\oplus m_2}$, and that $f^*$ and $\wt \Theta^j$ commute with direct sums.   
\end{proof}

We have a duality for kernel and cokernel bundles. For a vector bundle $\E$ on a 
projective variety $X$, we denote by $\E^\vee = \mathcal Hom_{\cO_X}(\cE, \cO_X)$ the dual bundle.

\begin{prop}
\label{duality}
Let $M$ be a finite dimensional $kG$-module of constant $j$-rank. Let $\cN = M^\# \otimes \cO_{\PG}$, 
and $\cM = M \otimes \cO_{\PG}$. Then  
$$\Ker\{\wt\Theta_G^j, \cM\}^\vee \simeq \Coker\{ \wt \Theta_G^j, \cN \}$$
as vector bundles on $\cO_{\bP(G)}$. 
\end{prop}

\begin{proof}
Choosing dual bases for $M$ and $M^\#$, we get an isomorphism of  trivial bundles  
$\cM$ and $\cN$. Hence, we may identify the dual bundle $\cM(jp^{r-1})^\vee$
with $\cN(-jp^{r-1})$.  Under this identification, the $\cO_{\bP(G)}$-dual of the map 
$$\wt\Theta_G^j: \cM \to \cM(jp^{r-1})$$ is identified with $$\wt\Theta_G^j(-jp^{r-1}) : \cN(-jp^{r-1}) \to \cN.$$ 
Since $\Ext^1_{\cO_{\PG}}(-,\cO_{\PG})$ vanishes on
locally free sheaves, taking the $\cO_{\bP(G)}$-dual of a short exact sequence of vector bundles 
$$\UseComputerModernTips
\xymatrix{0 \ar[r] & \Ker\{\wt\Theta_G^j, \cM \} \ar[r] & \cM \ar[r]^-{\wt\Theta_G^j} & \cM(jp^{r-1})
        \ar[r]& \Coker\{\wt \Theta^j, \cM\}(jp^{r-1}) \ar[r]& 0,
}
$$
we get an exact sequence 

\noindent
$\UseComputerModernTips
\xymatrix{0& \Ker\{\wt\Theta_G^j, \cM \}^\vee \ar[l]& \cN \ar[l] &&& \cN(-jp^{r-1})
        \ar[lll]_-{\Theta_G^j(-jp^{r-1})}  &  \ar[l] 
}
$
$$\UseComputerModernTips
\xymatrix{  \quad \quad &\Coker\{\wt \Theta^j, \cM\}^\vee(-jp^{r-1}) & 0\ar[l]
.}
$$

\noindent
Therefore, $\Ker\{\wt\Theta_G^j,\cM \}^\vee \simeq \Coker\{ \wt \Theta_G^j, \cN \}$.

\end{proof}

\begin{ex}
\label{cJ}
For each of our four examples of infinitesimal group schemes 
(initially investigated in Example \ref{four}), we give examples of
 $kG$-modules of constant Jordan type taken from \cite{CFP}.
\vskip .1in
(1)
Let $\fg$ be a finite dimensional $p$-restricted Lie algebra of dimension
at least 2.   For any Tate cohomology class of negative dimension,
$\zeta \in \widehat H^n(\cu(\fg),k) \simeq \Ext^1_{\cu(\fg)}(\Omega^{n-1}(k), k)$, 
we consider the extension of $\cu(\fg)$-modules
$$
\UseComputerModernTips
 \xymatrix{0 \ar[r]& k \ar[r] & M \ar[r] & \Omega^{n-1}(k) \ar[r] & 0}
 $$
determined by $\zeta$.  By \cite[6.3]{CFP}, 
$M$ is a $\cu(\fg)$-module of constant Jordan type.  We verify
by inspection that the Jordan type of $M$ is $(a, 0, \ldots, 0, 2)$ 
for some $a > 0$ if $n$ is odd, and 
$(b, 1, 0, \ldots, 0, 1)$ for some $b > 0$ if $n$ is even (see (\ref{jtype})  for notation).

\vskip .1in 
(2)
Let  $G = \bG_{a(r)}$, and set  $I$ equal to  the augmentation ideal
of $kG \simeq k[u_0, \ldots, u_{p-1}]/(u^p_0, \ldots, u^p_{p-1})$.  As observed in \cite{CFP},
$I^i/I^t$ is a module of constant Jordan type for any $t > i$.
As proven in \cite{CFS}, the only ideals of $k\bG_{a(2)}$ which are of constant Jordan type 
are of the form $I^i$.
\sloppy
{

}

\vskip .1in 
(3)
As observed in \cite{CFP},  the
$n^{th}$ syzygy module $\Omega^n(k), \ n \in \bZ$, is a module of constant
Jordan type for any infinitesimal group scheme $G$.  
For $n$ even, $\Omega^n(k)$ has constant Jordan type $(a, 0, \ldots, 0, 1)$
for some $a > 0$; whereas for $n$ odd, $\Omega^n(k)$ has constant Jordan 
type $(b, p-1, 0, \ldots, 0)$ for some $b > 0$.
\vskip .1in 
(4)
For $G = \SL_{2(2)}$, we recall that the cohomology algebra $\bH(G,k)$
is generated modulo nilpotents by classes $\zeta_1,\zeta_2,\zeta_3 \in
\HHH^2(G,k)$ and classes $\xi_1,\xi_2,\xi_3 \in \HHH^{2p}(G,k)$ (\cite{FS}).  As in
\cite[6.8]{CFP}, the $kG$-module
$$M \ \equiv \ \Ker\{ \sum \zeta_i + \sum \xi_j: (\Omega^2(k))^{\oplus 3}
\oplus (\Omega^{2p}(k))^{\oplus 3} \ \to \ k \}$$
is a $kG$-module of constant Jordan type $(a, 0, \ldots, 0, 1)$ for some 
$a > 0$.
\end{ex}
\vskip .2in

We elaborate  on the  Example~\ref{cJ}(2), constructing    $\bG_{a(r)}$-modules   
 of constant $j$-rank for  
but not  of constant Jordan type. 
   
\begin{ex}
We start with the  following simple observation.   
Let $M_1\subset M_2 \subset M$  be a chain of $k$-vector spaces, and  let $\phi$  
be an endomorphism  of $M$ such that $\phi(M_1) \subset M_1$ and $\phi(M_2) \subset M_2$.  
If $\dim (\Ker \phi_{|_{M_1}}) = \dim (\Ker \phi)$, then 
$\dim (\Ker \phi_{|_{M_1}}) =\dim (\Ker \phi_{|_{M_2}}) = \dim (\Ker \phi)$.

Let $G = \bG_{a(r)}$, and set  $I$ equal to  the augmentation ideal
of $kG \simeq k[u_0, \ldots, u_{p-1}]/(u^p_0, \ldots, u^p_{p-1})$. Consider any ideal $J$ of $kG$ with the property
that $I^i \subset J$ for some $i, \ i\leq p-1$.  Note that  for any $\ul a \in \bA^r$, and  any 
$j \leq p-i$, 
\sloppy
{

}

\begin{equation}
\label{j-ker}
\dim (\Ker \{\theta_{\ul a}^j: I^i  \to I^i \}) = pj = \dim (\Ker \{\theta_{\ul a}^j: kG  \to kG \}).
\end{equation}
Indeed,  since  $I^i$ is a  module of constant  Jordan type,  it suffices to check the statement for  
$\theta_{\ul a} = u_0$  for which it   is straightforward.     The observation  in the previous paragraph together with (\ref{j-ker}) 
and the   inclusions $I^i \subset J  \subset kG$   imply
$$ \dim (\Ker \{\theta_{\ul a}^j: J \to J \}) =  pj$$
for any $j \leq p-i$  and  any $\ul a \in \bA^r$.  Hence, $J$  has  constant $j$-rank for $1 \leq  j  \leq p-i$. 
\end{ex}

In the following example, we offer a method applicable to almost all infinitesimal
group schemes $G$ of constructing $kG$-modules which are of constant rank 
but not constant Jordan type.

\begin{ex}
Let $G$ be an infinitesimal group scheme with the property that $V(G)$ has
dimension at least 2.  Assume that $p$ is odd, and let $n > 0$ be an odd
positive integer.  Let $\zeta \in \HHH^n(G,k)$ be a non-zero cohomology class
and let $M$ denote the kernel of $\zeta: \Omega^n(k) \to k$.  Then $M$
has constant rank but not constant Jordan type.  Namely, the local Jordan
type of $M$ at $0 \not= v \in V(G)$ is $(a, 0, 1, 0, \ldots, 0)$ if $\zeta(v) \not= 0$, and is
$(a-1, 2, 0, \ldots, 0)$ if $\zeta(v) = 0$.   These Jordan types have the same rank.
\end{ex}

For $G = \SL_{2(1)}$, the restriction of any rational $\SL_2$--module is a module of constant  
Jordan type (see \cite{CFP}).  Irreducible $\SL_2$-modules $S_\lambda$ are parameterized by  their
highest weight, a non-negative integer $\lambda$.  Irreducible $\SL_{2(1)}$ modules are the restrictions of $S_\lambda$ to
$\SL_{2(1)}$ for $\ 0 \leq \lambda \leq p-1$. 

Another important family of $\SL_2$-modules are
the $V_\lambda$ (also denoted $H^0(\lambda))$
defined as the subspace of $k[s,t]$ (i.e., the symmetric algebra on
the natural 2-dimensional representation for $\SL_2$) consisting of homogeneous vectors
of degree $\lambda$. For $0 \leq \lambda \leq p-1$, we have an isomorphism of $\SL_{2(1)}$--modules: $S_\lambda \simeq V_\lambda$.  

Recall that $V(G)$ is the nullcone  in $sl_2$, and, hence, $A = k[V(G)]\simeq k[x,y,z]/(xy+z^2)$. 
Let  
\begin{equation}
\label{conic}
i: \bP^1 \ \to  \ \bP(G)
\end{equation}
be the isomorphism given on homogeneous coordinates by 
$$\frac{k[x,y,z]}{(xy+z^2)} \to k[s,t] \quad \quad  (x, y, z) \mapsto (s^2, -t^2, st).$$
In the next proposition, we compute  explicitly the kernel bundles  associated to  the irreducible $\SL_{2(1)}$-modules, and for the induced modules $V_\lambda$ for $p \leq \lambda \leq 2p-2$.  For convenience,  we give the answer in terms of pull-backs to $\bP^1$ via the isomorphism $i$. 

\begin{prop}
\label{sl2}
Let $G = \SL_{2(1)}$, and let $i: \bP^1 \ \stackrel{\sim}{\to} \ \bP(G)$  be the isomorphism defined in (\ref{conic}).
\begin{enumerate}
  \item For $0 \leq \lambda \leq p-1$,
$$i^*(\Ker\{\wt\Theta_G, S_\lambda \otimes \cO_{\PG} \}) \simeq 
\cO_{\bP^1}(-\lambda).$$
\item For $p \leq \lambda \leq 2p-2$, 
$$i^*(\Ker\{\wt\Theta_G, V_\lambda \otimes \cO_{\PG} \}) \simeq 
\cO_{\bP^1}(-\lambda) \oplus \cO_{\bP^1}(\lambda-2(p-1)).$$
\end{enumerate}\end{prop}

\begin{proof}
We adopt the conventions of \cite[\S 1]{BO}; in particular, we replace
$\lambda$ by $m$.
Let $v_0, v_1, \ldots, v_m$ be a basis for $V_m$ such that the 
generators  $e, f$ and $h$ of $sl_2$ act as  follows:  \\[1pt]
$hv_i = (2i-m)v_i$, \\ 
$fv_i = (m-i+1)v_{i-1}$ for $i>0$, $fv_0 = 0$, \\ 
$ev_i = (i+1)v_{i+1}$ for $i<m$, $ev_m=0$\\[1pt]
(see \cite[7.2]{Hum} or \cite[\S 1]{BO}).  Recall that $$\Theta_G = xe+yf+zh$$  (see Example \ref{ex-rep}(1)). Hence,  the operator 
$$
\UseComputerModernTips
 \xymatrix{\Theta_G:  V_m \otimes A \simeq A^{m+1} \ar[r]&  V_m \otimes A  \simeq A^{m+1}}$$ is 
represented by the matrix 
\begin{equation}
\label{matrix2}
\begin{pmatrix}    -m z   &    my & 0 & \ldots & \ldots   \\       
x  & -(m-2) z & (m-1) y & 0& \ldots \\
0& 2x  & -(m-4) z & (m-2) y & 0& \ldots \\
0& 0& 3x  & -(m-6) z & (m-3) y & \ldots \\
\ldots & \ldots & \ldots & \ldots& \ldots  \\
\ldots & \ldots & \ldots & 0& mx  & m z\end{pmatrix}_.
\end{equation}
Substituting $(s^2, -t^2, st)$ for $( x,y,z)$, we  get a degree two operator on $k[s,t]^{m+1}$ given by
\begin{equation}
\label{matrix3}
B_m(s,t) =\begin{pmatrix}    -m st   &    -mt^2 & 0 & \ldots & \ldots   \\       
s^2  & (2-m) st & (1-m) t^2 & 0& \ldots \\
0& 2s^2  & (4-m) st & (2-m) t^2 & 0& \ldots \\
0& 0& 3s^2  & (6-m) st & (3-m) t^2 & \ldots \\
\ldots & \ldots & \ldots & \ldots& \ldots  \\
\ldots & \ldots & \ldots & 0& ms^2  & m st\end{pmatrix}_.
\end{equation}

One verifies easily that the vector 
$$w_m \ = \ [t^m, -st^{m-1}, s^2t^{m-2}, \ldots, \pm s^m]$$
is annihilated by $B_m(s,t)$.
For $0 \leq m \leq p-1$, the module $V_m $ is irreducible and the kernel bundle has rank
$1$ because the (constant) Jordan type of $V_m$ has a single block.  The vector $w_m$
 generates the kernel as a graded $k[s,t]$-module
 (no element of smaller degree lies in the kernel).  Because $w_m$ is homogeneous
 of degree $m$ in $k[s,t]^{m+1}$, we conclude  for $0 \leq m \leq p-1$ that 
$$i^*(\Ker\{\wt\Theta_G, V_m \otimes \cO_{\PG} \}) \simeq \cO_{\bP^1}(-m), \quad 0 \leq m \leq  p-1.$$

\vspace{0.1in}
For $p\leq m \leq 2p-2$, $V_m$ has a decomposition series which can be represented as follows: 
$$\UseComputerModernTips
 \xymatrix{
S_{\mu}\ar[dr] && S_{\mu}\ar[dl]\\ 
& S_\lambda &
}
$$
where $\lambda = 2(p-1) -m$, $\mu = m-p = p-2-\lambda$, and $S_\lambda$, the irreducible module of highest weight $\lambda$, is the socle of $V_m$ (\cite[\S 1]{BO}).    By \cite{CFP}, $V_m$ has constant Jordan type. Plugging $x=1$, $y=z=0$ in (\ref{matrix2}), we get that the Jordan type is $[p] + [\mu+1]$. In particular, the rank of the kernel bundle is $2$. 

Using the relations $\mu \equiv m \,({\rm mod} \, p)$ and $ -\lambda \equiv \mu +2 \,({\rm mod} \, p)$, 
we obtain that the matrix $B_m(s,t)$  has the following form: 
\begin{equation}
\label{matrix4}
B_m(s,t) = \left[\begin{array}{rrrrlrrllrr}    &&&| &&&& | &&&\\
&&B_\mu(s,t)&| &&&& | &&&\\
&&&|&0 &&& | &&&\\
\hline
&&& _{(\mu+1)s^2} | &&&& | &&&\\
&&&| &&B_{\lambda}(t,s)^T&& | &&&\\
 & &&| &&&& | _{(\lambda+1)t^2}&&&\\
\hline
&&&| &&& 0&| &&\\
&&&| &&&& | &B_\mu(s,t)&&\\
&&&| &&&& | &&&\\
\end{array}\right]
\end{equation}
Here, the top left and bottom right corners are of  size $\mu+1 \times \mu+1$, whereas the matrix in the center is of  size $\lambda+1 \times \lambda+1$.  The only non-zero entries  outside of these three square diagonal blocks are at $(\mu+2, \mu+1)$ and $(\mu + \lambda + 2, \mu + \lambda+3)$, equaled to $(\mu +1)s^2$ and $(\lambda+1)t^2$ respectively.  

In particular, the $\mu +1 \times \lambda +1$ blocks above and below the matrix $B_\lambda(t,s)^T$ in the middle are zero. This implies that the kernel of this operator contains a copy of the kernel of $B_\lambda(s,t)$ (since it coincides with the kernel of $B_\lambda(t,s)^T$). We therefore obtain the vector
$$w_m^\prime \ = \ [ 0, \ldots, 0, t^\lambda, -st^{\lambda - 1}, \ldots, \mp s^\lambda, 0, \ldots, 0]$$
in the kernel, where the non-zero entries are at the positions $(\mu+2, \ldots, \mu + \lambda+2)$.

One verifies that $\{ w_m, w_m^\prime \}$ generate the kernel of $B_m(s,t)$ as a graded $k[s,t]$-module.
For example, one can check this by restricting to the affine pieces $U(s \not = 0)$ and $U(t \not  = 0)$.  Hence, 
$$i^*(\Ker\{\wt\Theta_G, V_m \otimes \cO_{\PG} \}) \simeq 
\cO_{\bP^1}(-m) \oplus \cO_{\bP^1}(m-2(p-1)).$$
\end{proof}

One may readily determine the rank of various bundles of $\bP(G)$ associated to 
modules of constant Jordan type using the next proposition.

\begin{prop}
\label{sub-rk}
Let $G$ be an infinitesimal group scheme, let $M$ a $kG$-module of constant
Jordan type $\sum_{i=1}^p a_i[i]$, and  let $\cM=M \otimes \cO_{\PG}$.  
Then for any $j, 1 \leq j < p$, 
\begin{equation}
\label{rkk}
\rk(\Im\{ \wt\Theta_G^j,\cM \}) \ = \ \sum_{i = j+1}^p a_i(i-j).
\end{equation}
In particular,
$$\Ker\{ \wt\Theta_G,\cM \} \subset 
\Ker\{ \wt\Theta_G^2,\cM \} \subset \cdots
\subset \Ker\{ \wt\Theta_G^{p-1},\cM \} \subset \cM$$
is a chain of $\cO_{\PG}$-submodules with 
$\rk(\Ker\{ \wt\Theta_G^{j-1},\cM \}) \ <  \rk(\Ker\{ \wt\Theta_G^j,\cM \})$ 
if and only if $a_i \not= 0$ for some $1 \leq j \leq i \leq p$. \end{prop}

\begin{proof}  The formula (\ref{rkk}) is the formula for the rank of $u^j$ on the
$k[u]/u^p$-module $\oplus_i (k[u]/u^i)^{\oplus a_i}$ of Jordan type 
$\sum_{i=1}^p a_i[i]$.    This is therefore the dimension of the image of 
$\theta_v, \ 0 \not= v \in V(G)$ on $M_{k(v)}$, and thus the rank of the 
vector bundle $\Im\{ \wt\Theta_G^j,\cM \}$ by Theorem~\ref{equiv}.  
\end{proof}

The following class of modules, of interest in its own right, is currently being 
studied by Jon Carlson and the authors.

\begin{defn}
Let $G$ be an infinitesimal group scheme, $M$ a finite dimensional $kG$-module, 
and $j < p$ a positive integer.  
We say that a $kG$-module $M$ has the {\it constant $j$-image property} 
if there exists a  subspace $I(j) \subset M$ 
such that for every $v \not  = 0 $ in $V(G)$, the image of $\theta_v^j: M_{k(v)} \to M_{k(v)}$ 
equals $I(j)_{k(v)}$. Similarly, we say that $M$ has {\it constant $j$-kernel property} 
if there exists some submodule $K(j) \subset M$ 
such that for every $v \not  = 0 $ in $V(G)$, the kernel of $\theta_v^j: M_{k(v)} \to M_{k(v)}$ 
equals $K(j)_{k(v)}$.
\end{defn}

We see that these modules are precisely those whose associated vector bundles are trivial
vector bundles.

\begin{prop}
\label{non-trivial}
Let $G$ be an infinitesimal group scheme, and let $M$ be a $kG$-module of constant $j$-rank.  
Then the algebraic vector bundle $\Im\{\wt \Theta_G^j,\cM\}$ is trivial (i.e., a free coherent sheaf) 
on $\PG$ if and only if $M$ has the constant $j$-image property. Similarly, $\ \Ker\{\wt \Theta_G^j,\cM\}$ is trivial if and only 
if $M$ has the constant $j$-kernel property.
\end{prop}

\begin{proof}
If $M$ has a constant  $j$-image property then $\Im\{ \wt\Theta_G^j,\cM \}$   is a free $\cO_X$-module generated by $I(j)$.  
Conversely,  assume that  $\Im\{ \wt\Theta_G^j,\cM \}$  is a free $\cO_X$-module.  
Then there exists a subspace $I(j) \subset
M = \Gamma(X,\cM)$ which maps to and spans each fiber 
$\Im\{ \theta_v^j,M_{k(v)}\}$, for $\ 0 \not= v \in V(G)$.  The argument  for kernels  is similar. 
\end{proof} 

\begin{remark}  We  point  out the properties of  constant  j-image and  constant  j-kernel are independent  of each other.   
Consider  the module $M^\#$  of  Example~\ref{duals}.  As shown  in that example, $\Ker \{ \wt\Theta_G,\cM^\# \}$  is locally free of rank 2   
but  not free, since  the global sections have dimension one.   On the other hand, $\Im \{ \wt\Theta_G,\cM^\# \}$  is a free $\cO_X$-module generated by  
the global section $n_3$.   In particular,  $M^\#$  has constant  $1$-image  property but  not constant  $1$-kernel property.  

For the  module $M$ of Example~\ref{duals},  the sheaf $\Ker \{ \wt\Theta_G,\cM\}$  is free of rank $2$   
whereas $\Im \{ \wt\Theta_G,\cM \}$   is locally free of rank $1$  but  not free since
it does not have any global sections. 
Hence, $M$   has a constant  $1$-kernel  property but  not constant  $1$-image property. 
\end{remark}

We consider an analogue of the sheaf construction of Duflo-Serganova
for Lie superalgebras \cite{DS}.  This construction enables 
us to produce additional algebraic 
vector bundles on $\PG$.  We implicitly use the observation $\wt\Theta^p_G = 0$.

\begin{defn}
\label{bracket}
Let $G$ be an infinitesimal group scheme, and  let $M$ be a finite dimensional $kG$--module. Let $\cM = M \otimes \cO_{\PG}$.
For any $i, \ 1 \leq i \leq p-1$, we define coherent 
$\cO_{\PG}$-modules, subquotients of $\cM$:
$$\cM^{[i]}\  \equiv \  \Ker\{ \wt\Theta_G^i,\cM \}/\Im\{\wt\Theta_G^{p-i},\cM \}.$$
\end{defn}

The following simple lemma helps to motivate these subquotients.

\begin{lemma}
\label{simple}
Let $V$ be a finite dimensional $k[t]/t^p$-module, and  let $JType(V, t) = (a_p, \ldots, a_1)$ (using  the notation   introduced in (\ref{jtype})). 
Let $$
V^{[j]}= \Ker\{ t^j:V \to V \}/\Im\{ t^{p-j}: V \to V \}
$$ for 
$j \leq p-1$.  Then 
$$\dim(V^{[j]}) \ =  \ \sum_{1 \leq i \leq j} ia_i  +
\ \sum_{i > j} ja_i -  
\sum_{i+j >p} (i+j-p)a_i.$$
In particular, $V$ is projective as a $k[t]/t^p$-module
if and only if $V^{[1]} \ = \ 0$.

\vspace{0.1in}
\noindent
Furthermore, for $j\leq p-1$, $V^{[j]} \simeq V^{[p-j]}$ as $k[t]/t^p$-modules.
\end{lemma}

As seen in the next proposition, these subquotients can provide additional
examples of algebraic vector bundles over $\PG$.

\begin{prop}
\label{bundle2}
Let $G$ be an infinitesimal group scheme and let 
$M$ be a finite dimensional $kG$-module which is of constant $j$-rank
and constant $(p-j)$-rank for some $j, \ 1 \leq j < p$.
Then $\cM^{[j]}$ is a locally free $\cO_{X}$-module and 
$  \cM^{[j]} \otimes_{\cO_X} k(x)  \ \to \ M_{k(x)}^{[j]} $ is an
isomorphism for all $x \in X \equiv \PG$.  
\end{prop}

\begin{proof}
For any $x \in X$, consider the map of exact sequences
\[
\UseComputerModernTips
 \xymatrix{
 \Im\{\wt\Theta_G^{p-i},\cM \} \otimes_{\cO_X} k(x) \ar[d]\ar[r] 
&   \Ker\{ \wt\Theta_G^i,\cM \} \otimes_{\cO_X} k(x)
 \ar[d]\ar[r] &  \cM^{[j]} \otimes_{\cO_X} k(x)  \ar[d]\ar[r] & 0 \\
\Im\{ \theta^{p-i}_x,M_{k(x)} \} \ar[r] & \Ker \{\theta^{i}_x,M_{k(x)} \} \ar[r] &
M_{k(x)}^{[j]}  \ar[r] & 0.
}
 \]
The left and middle vertical maps are  isomorphisms by Theorem \ref{equiv}.
Thus, the 5-Lemma implies that the right vertical arrow is also an
isomorphism.
\end{proof}

We  give an application of  this $(-)^{[1]}$ construction to  endotrivial modules. 
An interested reader   can compare   our construction to  \cite{BBC}.  
Recall that a  module $M$   of  a finite group scheme  $G$ is endotrivial if 
$\End_k(M) \simeq k + \proj$.    It was shown in \cite[\S 5]{CFP} that  an endotrivial 
module is a  module of constant  Jordan type with possible types $[1] + \proj$ and $[p-1] + \proj$. 

\begin{prop}
\label{endo} Let $G$  be an infinitesimal group scheme, and assume that $G$ has a subgroup scheme isomorphic to $\mathbb G_{a(1)}\times\mathbb G_{a(1)}$ or $\mathbb G_{a(2)}$.
Let $M$  be a module of constant Jordan type, and  set 
$\cM = M \otimes \cO_{\bP(G)} $.  
Then $\cM^{[1]}$  is a   line bundle (i.e., an algebraic vector bundle of rank one) 
if and only  if $M$  is endotrivial.    
\end{prop}

\begin{proof}  The sheaf $\cM^{[1]}$ is locally free by Lemma~\ref{simple}.  
Let $ \sum\limits_{i=1}^{p} a_i[i]$  be the  Jordan type of $M$.  Proposition 
\ref{bundle2}  implies that  the rank of the vector bundle 
$\cM^{[1]}$ equals $\sum\limits_{i=0}^{p-1} a_i$.  Hence, $\cM^{[1]}$ is a   line bundle
 if and only  if the Jordan type of $M$  has   only   
one  non-projective block.  A theorem of D. Benson \cite{Ben2}   states that modules of 
constant  Jordan type  with  unique non-projective 
block  must  be of type $[1] + \proj$ or $[p-1] + \proj$.  By \cite[\S 5]{CFP},  this happens  
if and only if $M$ is endotrivial.  
\end{proof}

We next give a global version of the observation in Lemma \ref{simple} that
$V^{[j]} \simeq V^{[p-j]}$ for $j \leq p-1$.   Recall that 
for a variety $X$, and a coherent sheaf $\E$, we denote by $\E^\vee = Hom_{\cO_X}( \E, \cO_X)$  the dual sheaf.

\begin{prop}
\label{dual}
Let $G$ be an infinitesimal group scheme, and let 
$M$ be a  $kG$-module which is of constant $j$-rank
and of constant $(p-j)$-rank 
for some $j, \ 1 \leq j < p$.  Let $\cM = M\otimes \cO_{\PG}$, $\cN = M^\#\otimes \cO_{\PG}$.
Then 
$$\cN^{[p-j]} \simeq (\cM^{[j]})^\vee
$$
as $\cO_{\PG}$--modules.
\end{prop}

\begin{proof} Let $X = \PG$.
As discussed in the proof of Proposition \ref{duality}, 
the $\cO_X$-linear dual of the complex of $\cO_X$-modules
$$\UseComputerModernTips
\xymatrix{\cM(-(p-j)p^{r-1}) \  \ar[rrr]^-{\wt\Theta_G^{p-j}(-(p-j)p^{r-1})} &&& \cM \  
\ar[r]^-{\wt\Theta_G^j}& \cM(jp^{r-1}) } $$
is the complex 
$$\UseComputerModernTips
\xymatrix{\cN((p-j)p^{r-1})  \  &   \cN \  
\ar[l]_-{\,\,\wt\Theta_G^{p-j}}&&  \cN(-jp^{r-1})\ar[ll]_-{\wt\Theta_G^j(-jp^{r-1})}}.$$
A similar statement
applies with $\theta_v$ in place of $\wt\Theta_G$.

For any scheme $Y$ and any complex of $\cO_Y$-modules
$$\UseComputerModernTips
\xymatrix{S_1 \ar[r]^-{f}& S_2 \ar[r]^-{g}& S_3}$$
with $\cO_Y$-linear dual
$$\UseComputerModernTips
\xymatrix{S_1^\vee  & S_2^\vee \ar[l]_-{f^\vee}&
S_3^\vee\ar[l]_-{g^\vee}},$$
there is a natural pairing
\begin{equation}
\label{pairing}
\UseComputerModernTips
\xymatrix{(\Ker\{ g \}/\Im \{ f \}) \otimes (\Ker\{ f^\vee \}/\Im\{ g^\vee \}) \ \ar[r] & \ \cO_Y,}
\end{equation}
induced by the evident pairing 
$$S_2 \otimes S_2^\vee \to \cO_Y.$$
In particular, we have a pairing 
$$\cM^{[j]} \otimes  \cN^{[p-j]} \to \cO_X,$$
and, hence, a map 
$$f: \cN^{[p-j]} \to (\cM^{[j]})^\vee.$$ 
By Proposition~\ref{bundle}, 
$$\cM^{[j]} \otimes_{\cO_X} k(x) \simeq  M_{k(x)}^{[j]}, \quad \cN^{[p-j]} \otimes_{\cO_X} k(x) \simeq  N_{k(x)}^{[p-j]}$$
for any $x \in X$. 
By naturality, the specialization  of $f$ at a point $x$ corresponds to the map $f_x: N_{k(x)}^{[p-j]} \to (M_{k(x)}^{[j]})^\#$ induced by the paring 
(\ref{pairing}) for $Y = \Spec k(x)$. 
One readily verifies that this is a perfect pairing if $\cO_Y$ is a field. Hence, $$f \otimes_{O_X} k(x): \cN^{[p-j]} \otimes_{\cO_X} k(x) \to (\cM^{[j]})^\vee  \otimes_{\cO_X} k(x)$$
is an isomorphism for any $x \in X$. Therefore,  $\cN^{[p-j]} \simeq (\cM^{[j]})^\vee$.
\end{proof}

Consideration of  $\cM^{[1]}$ leads to another characterization of projective 
$kG$-modules.

\begin{prop}
\label{char}
Let $G$ be an infinitesimal group scheme and let 
$M$ be a finite dimensional $kG$-module.  Then $M$ is projective if and
only if $M$ has constant rank, has constant $(p-1)$--rank, and satisfies 
$\cM^{[1]} \ = \  0$.
\end{prop}

\begin{proof}  Assume that $M$ is a projective $kG$-module.  Then $M$ has
 constant Jordan type (which is some multiple of $[p]$), and hence has constant
rank and constant $(p-1)$-rank.
For any $x \in \PG = X$,  $\theta_x^*(M_{k(x)})$ is a free $k(x)[t]/t^p$-module
of rank equal to $\frac{\dim(M)}{p}$.    If we lift a basis of this free
module to $\cM_{(x)} \equiv  \cM  \otimes_{\cO_X}\cO_{X,x} $, then an application of 
Nakayama's Lemma tells us that $\cM_{(x)}$ is free as an $\cO_{X,x}[t]/t^p$-module.
This readily implies that 
$(\cM_{(x)} )^{[p-1]} \equiv 
\Ker\{\wt\Theta_{G,(x)}^{p-1},\cM_{(x)}\}/\Im\{ \wt\Theta_{G,(x)},\cM_{(x)} \}$
vanishes.
Using the exactness of localization, we conclude that  
$(\cM^{[p-1]})_{(x)} = (\cM_{(x)} )^{[p-1]} $.   Consequently, $\cM^{[p-1]} = 0$.
By Proposition \ref{dual}, we conclude that $\cM^{[1]} = 0$.

Conversely, if $M$ has constant rank and constant $(p-1)$-rank and if
$\cM^{[1]} = 0$, then Proposition \ref{bundle2} tells us that $M_{k(x)}^{[1]} 
\equiv \Ker \theta_x/\Im \theta_x^{p-1}$ equals $0$ for all $x \in X$.  
 Lemma \ref{simple} thus implies that each $M_{k(x)}$ is projective, so
 that the local criterion for projectivity \cite{SFB1} implies that $M$
 is projective.
\end{proof}

One very simple invariant of the algebraic vector bundle $\Ker\{\wt \Theta_G^j,\cM \} $
is the dimension of its vector space of global sections.  The following proposition
gives some understanding of $\Gamma(\PG,\Ker\{\wt \Theta_G^j,\cM \}) \ \subset \ 
\Gamma(\PG,\cM) $.

\begin{prop}
\label{sub}
Let $G$ be an infinitesimal group scheme, and  assume that  $V(G)$ is reduced.  
Let $M$ be a $kG$-module and let $\cM =
M \otimes \cO_{\PG} $.  Then 
$$
\Gamma(\PG,\Ker\{\wt \Theta_G^j,\cM \}) \ \subset \ M
$$
consists of those $m \in M$ such that $\theta_x^j(m) = 0$ for all
$x \in \PG$.
\end{prop}

\begin{proof}
Recall  that $\PG$ is
connected by \cite[3.4]{CFP} and thus $\Gamma( \PG,\cM) = M$. 
Under this identification, the global sections of $\Ker\{ \wt\Theta_G^j,\cM\}$   coincide with 
the subset 
$$ \{m  \in M \, | \, \Theta_G^j(m \otimes 1)  = 0 \}.$$ 
Since $V(G)$ is reduced, we have   $\Theta_G^j(m \otimes 1) = 0$ if and only if 
$\theta_v^j(m \otimes 1) = \Theta_G^j(m \otimes 1) \otimes_{k[V(G)]} k(v) = 0$ for 
every $v \in V(G)$.  
Hence,  $ m  \in \Gamma(\PG,\Ker\{ \wt\Theta_G^j,\cM\})$ if and  only if 
$m \in \Ker\{\theta_v^j,M_{k(v)}\}$ for every $v \in V(G)$ if and only if 
$\theta_x^j(m) = 0$ for all $x \in \PG$.
 \end{proof}
 
We make Proposition \ref{sub} more explicit in the case of a classical Lie algebra.
 
\begin{prop}
\label{generate}
Let $\fG$ be a (reduced, irreducible) algebraic group over $k$, let $G = \fG_{(1)}$,
and let $\fg = Lie(\fG)$.  If $\fg$ is generated by $p$-nilpotent elements and
if $M$ is a rational $\fG$-module, then
$\Gamma(\PG,\Ker\{\wt \Theta_G^j,\cM \})$ is a rational $\fG$-submodule of $M$.

Furthermore, 
$$
\Gamma(\PG,\Ker\{ \wt \Theta_G,\cM \}) \ = \ \HHH^0(G,M).
$$
 \end{prop}
 
 \begin{proof}

To prove that $\Gamma(\PG,\Ker\{\wt \Theta_G^j,\cM \})$ is a $\fG$-submodule 
of $M$, we may base change to the algebraic closure of $k$,
and thus we may assume $k$ is algebraically closed.  Let $g \in \fG$
be a $k$-rational point.
Then
\begin{equation}
\label{action1}
\theta_v^j(gm \otimes 1) = g \theta^j_{v^{g^{-1}}}(m\otimes 1),.
\end{equation} 
where the action of $\fG$ on $V(G) = \cN_p(\fg)$ (the $p$-nilpotent cone of $\fg$)
is via the adjoint action of $\fG$.

Hence, we   have the following   equalities:
$$\{ m \in M \, | \,  \Theta^j_G(m \otimes  1) = 0  \} \ = \
\bigcap \limits_{0 \not  = v \in V(G)} \{ m \in M \, | \,  
\theta^j_v(m \otimes  1) = 0  \} \ = 
$$
$$
\bigcap \limits_{0 \not  = \Ad(g^{-1})v \in V(G)} \{ m \in M \, | \,  
g \theta^j_{\Ad(g^{-1})v}(m\otimes 1) = 0  \} \ = \
\bigcap \limits_{0 \not  = v \in V(G)} \{ m \in M \, | \,  
\theta^j_v(gm \otimes  1) = 0  \} 
$$
$$
= \ \{ m \in M \, | \,  \Theta^j_G(gm \otimes  1) = 0  \},
$$
where  the  first and the last equality follow  from Proposition \ref{sub}, the second  
equality follows from the fact that $\Ad(g^{-1}): V(G) \to V(G)$ is a bijection,  and the third equality from (\ref{action1}). 
We conclude that $\{m \in M \, | \,  \Theta^j_G(m \otimes  1) = 0  \}$ 
is  a $\fG$-stable  subspace of $M$.  

The second assertion follows immediately from Proposition \ref{sub} and the 
facts that $v \in V(G)$ corresponds to a $p$-nilpotent element $X_v$ of $\fg$ and
that the action of $\theta_v$ is the action of $X_v$.

\end{proof}

Combining Proposition \ref{sub-rk} and Corollary \ref{generate} in the special case $j=1$ yields
the following criterion for the non-triviality of $\Ker\{ \wt\Theta_G,\cM\}$.

\begin{cor}
\label{subb}
Let $G$ be an infinitesimal group scheme such that $V(G)$ is 
reduced and positive dimensional. Assume  that for any field extension $K/k$, 
$KG$ is generated by $\theta_v \in k(v)G$,  for all $v \in V(G)$ such that $k(v) \subset K$.  
Let $M$ a finite dimensional  $kG$-module of constant Jordan type $\sum_i a_i[i]$.  If 
$$\dim \HHH^0(G,M) \ < \sum_{i = 1}^p a_i,$$  
then $\Ker\{ \wt \Theta_G,\cM\}$ is a non-trivial algebraic vector bundle 
over $\PG$
\end{cor}

\begin{proof}
By Proposition \ref{sub-rk}, the  dimension of the fibers of $\Ker\{ \wt \Theta_G,\cM\}$ is 
$\dim M - \sum_{i = 2}^p a_i(i-1) = \sum_{i = 1}^p a_i$. 
By Proposition \ref{sub},  the global sections  of $\Ker\{ \wt \Theta_G,\cM\}$ equal $\HHH^0(G,M)$. Hence, the inequality $\dim \HHH^0(G,M) \ < \sum_{i = 1}^p 
a_i$ 
implies that the dimension of the global sections  is less than the dimension of the fibers.  
Therefore, the sheaf is not free.  

\end{proof}

The following two lemmas will be applied to prove Proposition \ref{bracket2}.

\begin{lemma}
\label{nak}
Let $R$ be a local commutative ring with residue field $k$ and let 
$M$ be a finite $R[t]/t^p$-module which is free as an $R$-module.  
If $ M  \otimes_R k$ is a free $k[t]/t^p$-module, then $M$ is free as
an $R[t]/t^p$-module.
\end{lemma}    

\begin{proof}
Let $m_1,\ldots,m_s \in M$ be such that $\ol m_1,\ldots,\ol m_s$ 
form a basis for $M  \otimes_R k$ as a $k[t]/t^p$-module.  Let
$Q$ be a free $R]t]/t^p$-module of rank $s$  with basis $q_1,
\ldots,q_s$ and consider the $R[t]/t^p$-module homomorphism
$f: Q \to M$ sending $q_i$ to $m_i$.

By Nakayama's Lemma, $f: Q\to M$ is surjective.  Because $M$ is
free as an $R$-module, applying $- \otimes_R k$ to the short exact
sequence $0 \to \Ker\{ f \} \to Q \to M \to 0$ determines the short exact sequence
$$
0 \to  \Ker\{ f \}  \otimes_R k\ \to \ Q  \otimes_R k\ \to\ M  \otimes_R k\to 0.
$$
Consequently, $ \Ker\{ f \}    \otimes_R k= 0$, so that another application
of Nakayama's Lemma implies that $\Ker \{ f \} = 0$.     Hence, $f$
is an isomorphism, and thus $M$ is free as an $R[t]/t^p$-module.  
\end{proof}  

\begin{lemma}
\label{emptyint}
Let $G$ be an infinitesimal group scheme and $M$ be a finite dimensional 
$kG$-module.  Set $A = k[V(G)]$; for any $f \in A$, set $A_f = A[1/f]$.  Assume
that $\Spec A_f \subset V(G)$ has empty intersection with $V(G)_M$.
Then $(\cU_G \circ \epsilon)^*(M \otimes A_f)$ is a projective 
$A_f[t]/t^p$ -module.
\end{lemma}

\begin{proof}
By definition, $V(G)_M$ consists of those points 
$v \in V(G)$ such that $\theta_v^*(M_{(k(v)})$ is not free as a 
$k(v)[t]/t^p$-module.  By the universal property of \, $\cU_G \circ \epsilon$, 
the assumption that $\Spec A_f \cap V(G)_M = \emptyset$ implies
for every point $v \in \Spec A_f $ that 
$\theta_v^*(M_{k(v)}) =  (\cU_G \circ \epsilon)^*(M \otimes A_f)  \otimes_{A_f} k(v)$ is free 
as a $k(v)[t]/t^p$-module.  Let $A_{(v)}$ denote the localization of $A$ 
at $v$.  Then Lemma \ref{nak} implies for every point $v \in \Spec A_f$ 
that the localization
$ (\cU_G \circ \epsilon)^*(M \otimes A_f)  \otimes_{A_f} A_{(v)}$
is free as a $A_{(v)}[t]/t^p$-module.  This implies that
$M \otimes A_f$ is projective (since projectivity of a finitely generated module over a 
$A$ is determined locally).  
\end{proof}

We conclude with a property of the (projectivized) rank variety 
$\PG_M$ of a $kG$-module $M$.

\begin{prop} 
\label{bracket2}
Let $G$ be an infinitesimal group scheme, $M$ be a finite dimensional 
$kG$-module, and set $\cM = M \otimes \cO_{\PG}$.   Then 
$$\Supp_{\cO_{\PG}}(\cM^{[1]}) \ \subset   \ \PG_M,$$ 
where $\Supp_{\cO_{\PG}}(\cM^{[1]})$ is the support of the coherent
sheaf   $\cM^{[1]}$ (the closed subset of points $x \in \PG$ at which
$\cM^{[1]}_{(x)} \not= 0$).                                                              
\end{prop}                                                                       
                                                                                 
\begin{proof}
Let $A$ denote $k[V(G)]$ and let $X$ denote $\PG$.  
Consider some $x \notin X_M$ and choose some homogeneous
polynomial $F \in A$ vanishing on $X_M$
such that $F(x) \not= 0$.  Thus, $x \in \Spec (A_F)_0
\subset X$ and $\Spec (A_F)_0 \cap X_M = 
\emptyset$, where $(A_F)_0$ 
denote the elements of degree 0 in the localization 
$A_F = A[1/F]$.  It suffices to prove that $x \notin
\Supp_{\cO_X}(\cM^{[1]}) \cap \Spec (A_F)_0$.  Since localization is exact, this is equivalent to proving that
$v \notin \Supp_{A_F}((M \otimes A_F)^{[1]}) $
for some $v \in \Spec A_F$ mapping  to $x$.

The condition $\Spec (A_F)_0 \cap X_M = 
\emptyset$ implies that $\Spec (A_F) \cap V(G)_M = 
\emptyset$.  Hence, by Lemma \ref{emptyint}, 
$(\cU_G \circ \epsilon)^*(M \otimes A_F)$ is a projective 
$A_F[t]/t^p$ -module.  This implies that 
$(\cU_G \circ \epsilon)^*((M \otimes A_F)^{[1]}) = 0$, and thus that 
$v \notin \Supp_{A_F}((M \otimes A_F)^{[1]}) $.
\end{proof}

\begin{remark}
The reverse inclusion $\PG_M \subset \Supp_{\cO_{\PG}}(\cM^{[1]})$ 
seems closely
related to the condition that $ 
\Ker\{ \wt\Theta_G,\cM\} \otimes_{\cO_{\PG}}  k(x)\to \Ker\{ \theta_x,M_{k(x)} \}$ be surjective.
\end{remark}

\section{Examples and calculations with bundles}

In this final section, we investigate numerous specific examples.  The case in 
which $G$ equals either $\bG_{a(1)}^{\times 2}$ or $\ul sl_2$ (the infinitesimal group scheme associated to the restricted Lie algebra $sl_2$) is particularly amenable
to computation for $\bP(G)$ is isomorphic to $\bP^1$.  Specifically, we do calculations for 
projective $kG$-modules, examples of modules of constant Jordan type which
are not distinguished by support varieties.  We also compute bundles obtained from``zig-zag modules" and
syzygies.

As we see in the following simple example, the isomorphism type of the vector bundles 
discussed in Theorem \ref{bundle} can be used to distinguish certain $kG$-modules
which have the same local Jordan type.   We remind the
reader that the local Jordan type of a finite dimensional $kG$-module $M$ of constant
Jordan type is the same as that of its linear dual $M^\#$.

\begin{ex}
\label{duals}
Let $G = \bG_{a(2)}$, so that  $k\bG_{a(2)} = k[u_0,u_1]/(u^p_0,u^p_1)$, 
$V(\bG_{a(2)}) = \bA^2$, $A = k[V(\bG_{a(2)})] = k[x_0,x_1]$
graded so that $x_0$ is given degree 1 and $x_1$ is given degree $p$.
Then 
$$\Theta_G = x_1u_0 + x_0^pu_1 \in A[u_0,u_1]/(u^p_0,u^p_1).
$$
(see (\ref{p-univ}(2)).
We consider the 3-dimensional $kG$-module $M$ of constant Jordan type
$[2] + [1]$ and its linear dual $M^\#$, which we represent
 diagrammatically as follows:

$$
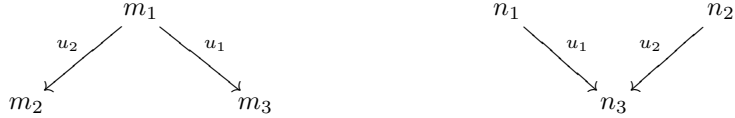
\UseComputerModernTips
 \xymatrix{
& m_1 \ar[dl]_{u_2}\ar[dr]^{u_1}  &&  & n_1 \ar[dr]^{u_1} && n_2 \ar[dl]_{u_2} \\
 m_2  &&  m_3 \  &\qquad  &&  n_3 &
}
$$
The $k\Ga2$-invariant subspace  of $M$ is two dimensional, and, hence, the global sections of $\Ker\{ \wt\Theta_G, M \otimes \cO_{\bP(G)} \}$  have dimension two by Prop.~\ref{sub} (in fact, an explicit calculation shows that this is a trivial bundle of rank 2). On the other hand, $M^\#$ has only one-dimensional invariant subspace, and, hence, the global sections of $\Ker\{ \wt\Theta_G, M^\# \otimes \cO_{\bP(G)} \})) $ have dimension one.
Thus, $$\Ker\{ \wt\Theta_G, M \otimes \cO_{\bP(G)} \} \ \ncong \ 
\Ker\{ \wt\Theta_G, M^\# \otimes \cO_{\bP(G)}\}. $$
\end{ex}
\vskip .1in

We next give a somewhat more interesting example of pairs of
modules of the same constant Jordan type with different associated
bundles.

\begin{ex}
As in Proposition \ref{sl2}, let $S_\lambda$ be the 
irreducible $\SL_2$-module of highest weight $\lambda, 0 \leq \lambda \leq p-1$, 
and consider $S^p(S_\lambda)$, the $p^{\rm th}$ symmetric power of $S_\lambda$.
 Since $S_\lambda$ is self-dual,
the dual of $S^p(S_\lambda)$ is $\Gamma^p(S_\lambda)$, the $p^{\rm th}$ 
divided power of $S_\lambda$.  We have a 
short exact sequence of rational $\SL_2$-modules
\sloppy{

}

$$
\UseComputerModernTips
\xymatrix{0 \ar[r]& S_\lambda^{(1)} \ar[r]& S^p(S_\lambda) \ar[r]& \Gamma^p(S_\lambda) \ar[r]& S_\lambda^{(1)} \ar[r]& 0.}$$
Here, $S_\lambda^{(1)}$ is the first Frobenius twist of $S_\lambda$,
thus trivial as a $\cu(sl_2)$-module.

Let $X=\Proj k[N(sl_2)]$. By Prop.~\ref{sub}, the space of global sections of $\Ker\{ \wt\Theta_{\ul sl_2}, S^p(S_\lambda) \otimes \cO_X \}$ equals the $sl_2$ invariants of $S^p(S_\lambda)$.  Hence, 
$$\Gamma(X, \Ker\{ \wt\Theta_{\ul sl_2}, S^p(S_\lambda) \otimes \cO_X \}) = S_\lambda^{(1)}.$$
On the other hand, $\Gamma^p(S_\lambda)$ does not have any $sl_2$-invariants, and, hence, $\Ker\{ \wt\Theta_{\ul sl_2}, \Gamma^p(S_\lambda) \otimes \cO_X \}$ does not have any global sections. We conclude that the kernel bundles associated to the dual modules $S^p(S_\lambda)$ and $\Gamma^p(S_\lambda)$ are non-isomorphic.
\sloppy{

}
\end{ex}
\vskip .1in

We continue our consideration of $\SL_{2(1)} \equiv \ul sl_2$ in the following proposition.

\begin{prop}
\label{pim} 
Let $G =  \ul sl_2$, let $S_\lambda$ be the irreducible $kG$-module of highest weight $\lambda$, $ 0 \leq \lambda \leq p-1$, and let 
$P_\lambda \to S_\lambda$ be the
projective cover of $S_\lambda$.    Then 

$$i^*(\Ker\{\widetilde \Theta_G, P_\lambda \otimes \cO_{\PG}\}) \simeq 
\begin{cases} \cO_{\bP^1}(1-p) \quad \text{ if } \lambda = p-1 \\
\cO_{\bP^1}(\lambda - 2(p-1)) \oplus \cO_{\bP^1}(-\lambda) \quad   \text{ if }  0 \leq \lambda \leq p-2\\
\end{cases}
$$

where $i: \bP^1 \to \PG$ is the isomorphism (\ref{conic}).
\end{prop} 

\begin{proof}
For $\lambda=p-1$, $P_\lambda = \St$ is the Steinberg module for $sl_2$, and is irreducible.  Hence, in this case the statement follows from Proposition~\ref{sl2}. 

The decomposition series of $P_\lambda$ for $0 \leq \lambda < p-1$
 is represented by the following diagram (see \cite[2.4]{FPa2}): 
\begin{equation}
\label{diag}
\UseComputerModernTips
 \xymatrix{
& S_\lambda \ar[dl]\ar[dr]& \\
S_{p-2-\lambda}\ar[dr] && S_{p-2-\lambda}\ar[dl]\\ 
& S_\lambda &
}
\end{equation}
Thus, we have a short exact sequence of $\SL_2$-modules
\begin{equation}
0 \to V_{2p-2-\lambda} \ \to \ P_\lambda \ \to \ S_\lambda \to 0.
\end{equation}
By Proposition \ref{sl2}, it suffices to prove that $V_{2p-2-\lambda} \subset P_\lambda $ 
induces an isomorphism
$$
\Ker\{\wt \Theta_G,V_{2p-2-\lambda} \otimes \cO_{\PG}\}\ \simeq \ 
\Ker\{\wt \Theta_G,P_\lambda \otimes \cO_{\PG}\}.
$$
By Theorem \ref{equiv}(8) and Nakayama's Lemma, it suffices to prove that 
$$\Ker\{\theta_x,V_{2p-2-\lambda} \otimes k(x) \} \to \Ker\{\theta_x,P_\lambda \otimes k(x) \}$$
is an isomorphism for all $x \in \PG$. 
This last statement follows from the observation that the Jordan decomposition of $\theta_x$ 
on both  $V_\lambda$ and $P_\lambda$ consists of two blocks: on $P_\lambda$, because
$P_\lambda$ is projective of dimension $2p$; and on $V_{2p-2-\lambda}$ as discussed in Proposition~\ref{sl2}.
\end{proof}   

If $X$ is an algebraic variety over $k$ (for example, $X = \bP(G)$), then $K_0(X)$ 
denotes the Grothendieck group of algebraic vector bundles over $X$ (i.e., finitely generated, 
locally free $\cO_X$-modules) defined as the free abelian group on the set of isomorphism
classes of such vector bundles modulo relations given by short exact sequences.  We shall also consider
$K_0^{\oplus}(X)$ defined as the free abelian group on the same set of generators modulo
relations given by split short exact sequences.  Thus, there is a canonical surjective homomorphism
$$K_0^{\oplus}(X) \ \to \ K_0(X).$$

\begin{defn}
\label{kappa}
Let $G$ be an infinitesimal group scheme and let $K_0(kG)$ denote the Grothendieck group
of finitely generated projective $kG$-modules.
For any $j, \ 0 \leq j \leq p-1$, the homomorphism
$$\kappa_{G,j}^{\oplus}: K_0(kG) \ \to \ K_0^{\oplus}(\bP(G))$$
is defined by sending a projective $kG$-module $Q$ to 
$\Ker \{ \wt \Theta_G^j,Q\otimes \cO_{\bP(G)} \}$.    We define $ \ul { \kappa}_G^{\oplus}$ 
by 
$$
 \ul { \kappa}_G^\oplus  = ( \kappa_{G,1}^{\oplus},\ldots  \kappa_{G,p}^{\oplus}): 
K_0(kG) \ \to \ K_0^{\oplus}(\bP(G))^{\oplus p}.$$

Moreover, the homomorphism
$$ \kappa_{G,j}: K_0(kG) \ \to \ K_0(\bP(G))$$
is defined to be the composition of $\tilde \kappa_{G,j}^{\oplus}$ with the canonical projection
$K_0^{\oplus}(\bP(G)) \to K_0(\bP(G))$, and the homomorphism
$$ \ul \kappa_G: K_0(kG) \ \to \ K_0(\bP(G))^{\oplus p}$$
is defined to be the composition of $\ul \kappa_G$ with the canonical projection
$K_0^{\oplus}(\bP(G))^{\oplus p} \to K_0(\bP(G))^{\oplus p}.$
\end{defn}

We shall omit the subscript $G$ is $\kappa_G$ when the group scheme is clear from the context.  Note that since $\wt \Theta^p=0$, $\kappa_p$ returns the entire class $[Q \otimes \cO_{\PG}]$. 

The {\it dimension of the global sections} function: 
$$ \cE \mapsto \dim \Gamma(\PG, \cE)$$ 
extends to a homomorphism
$$\rho: K_0^{\oplus}(\bP(G)) \  \to \bZ$$ 
Observe that since $\Gamma(\bP(G),-)$ is not right exact for non-split exact sequences of 
algebraic vector bundles on $\bP(G)$, $\rho$ does {\it not factor} through $K_0(\bP(G))$.

\begin{prop}
\label{sl2-proj}  Let $G = \ul sl_2$.
The composition
$$ \rho \circ \ul {\kappa}^{\oplus}: K_0(kG) \ \to \ 
K_0^\oplus(\PG)^{\oplus p} \ \to \bZ^{\oplus p}$$
 is a rational isomorphism.
\end{prop}

\begin{proof}
Recall that $K_0(kG) = K_0(u(sl_2)) \simeq \bZ^{\oplus p}$, spanned by the projective indecomposable
$u(sl_2)$-modules $P_\lambda, \ 0 \leq \lambda \leq p-1$.  
By  Proposition \ref{generate}, the global sections of   
$\Ker \{\wt\Theta_G^j, P_\lambda \otimes \cO_{\PG}  \})$ is a rational $\SL_2$-submodule
of $P_\lambda$.  Hence, 
$\Gamma(\PG, \Ker \{\wt\Theta_G^j, P_\lambda \otimes \cO_{\PG}  \}) \not = 0$ if and only if $S_\lambda$, which is the socle of $P_\lambda$,  belongs to
the global sections.  Hence,  
$$\Gamma(\PG, \Ker \{\wt\Theta_G^j, P_\lambda \otimes \cO_{\PG}  \}) \not = 0 \quad \Leftrightarrow \quad \wt \Theta^j(S_\lambda \otimes \cO_{\PG}) =0 \quad \Leftrightarrow \quad \theta^j_v(S_\lambda) = 0 $$ 
for any $v \in V(G)\simeq \cN(sl_2)$.  Since $S_\lambda$ is a module of constant  Jordan type $[\lambda+1]$, $\theta_v^j$ annihilates $S_\lambda$ if and only if $j>\lambda$. Thus, 
$$\Gamma(\PG, \Ker \{\wt\Theta_G^j, P_\lambda \otimes \cO_{\PG}  \}) \not = 0 \text{ if and only if } j>\lambda.$$
Moreover, the decomposition series for $P_\lambda$ (\ref{diag}) implies that the Jordan type of any rational submodule of $P_\lambda$ that is bigger than the socle $S_\lambda$  has a Jordan block of size $p$. Hence, 
$$\Gamma(\PG, \Ker \{\wt\Theta_G^{\lambda+1}, P_\lambda \otimes \cO_{\PG}  \}) = S_\lambda.$$ 
Note that the last equality holds trivially for $\lambda =p-1$, since in this case the kernel bundle is the entire free sheaf $P_{p-1} \otimes \cO_{\PG}$, where $P_{p-1} = S_{p-1}$ is the Steinberg module for $sl_2$. 

We conclude that the homomorphism 
$$ \rho \circ \ul {\kappa}^{\oplus}: K_0(kG)\simeq  \bZ^{\oplus p}\ \to \bZ^{\oplus p}$$
is given by a non-singular upper-triangular matrix
 
$$\left[\begin{array}{ccccc} 1 & \cdots &  &  & \cdot\\
0& 2& \cdots & &\cdot \\
0& 0& 3&  \cdots  &\cdot \\ 
\vdots&\vdots& \vdots& \ddots&\vdots\\
0&0&0& 0 & p
\end{array}\right] 
$$
Hence, $\rho \circ \ul {\kappa}^{\oplus}$ is a rational isomorphism. 
\end{proof}

In contrast with Proposition \ref{sl2-proj}, we have the following computations for 
$\kappa_1: K_0(kG) \to K_0(\PG)$ and $\ul \kappa:K_0(kG) \to K_0(\PG)^{\oplus p}$ for $G=\ul sl_2$.

\begin{lemma}
\label{kappa_sl2} Let $G=\ul sl_2$, and denote by $\St$ the Steinberg module for $sl_2$.  The image of the homomorphism
$$\kappa_1: K_0(kG) \to K_0(\PG)$$
is generated by $\kappa_1(\St)$. Consequently, $\rk \kappa_1 =1$.
\end{lemma}
\begin{proof}
The isomorphism $i : \PG \simeq \bP^1 $ (\ref{conic})  induces an isomorphism $K_0(\PG) \simeq K_0(\bP^1)$. We denote by the letter $\kappa_1$ the composition 
$K_0(kG) \to K_0(\PG) \stackrel{\sim}{\to} K_0(\bP^1)$.  Clearly, it suffices to prove the statement of the Lemma for this composition.

Let $a_n = [\cO_{\bP^1}(-n)] \in K_0(\bP^1)$, $n\in \Z$. Then $a_0 = [\cO_{\bP^1}]$, $a_1 = [\cO_{\bP^1}(-1)]$ generate
$K_0(\bP^1) \simeq  \Z^{\oplus 2}$.  
Using the short exact sequence of $\cO_{\bP^1}$-modules
$$
0 \to \cO_{\bP^1}(-(n+1)) \to \cO_{\bP^1}(-n) \oplus \cO_{\bP^1}(-n)\to \cO_{\bP^1}(-(n-1)) \to 0,
$$ 
we obtain the recurrence relation $a_{n+1} \ = \ 2a_n - a_{n-1}. $
By induction, $a_n \ = \ n a_1 - (n-1)a_0.$

By Proposition \ref{pim}, for $0 \leq \lambda < p-1$, 
$$\kappa_1(P_\lambda) \ = [\cO_{\bP^1}(-\lambda)] + [\cO_{\bP^1}(\lambda - 2(p-1)].$$ 
Hence, 
$\kappa_1(P_\lambda)=a_\lambda+ a_{2(p-1)-\lambda} =  \lambda a_1 - (\lambda-1)a_0 + (2(p-1)-\lambda)a_1 - (2(p-1)-\lambda-1)a_0 = 
\ 2(p-1)a_1 - 2(p-2)a_0$.
Moreover, 
$$\kappa_1(\St) = \ [\cO_{\bP^1}(1-p)] \ = \ (p-1)a_1 - (p-2)a_0$$
by Proposition \ref{sl2}.  Hence, 
$$\kappa_1(P_\lambda) = 2\kappa_1(\St)
$$
for $0 \leq \lambda\leq p-2$ which proves the statement.
\end{proof}

The second conclusion of the following Proposition  is in sharp contrast with Proposition~\ref{sl2-proj}.
\begin{prop}
\label{unexpected}  
Let $G=\ul sl_2$, and denote by $\St$ the Steinberg module for $sl_2$.  
\begin{enumerate}
  \item  The image of 
$\kappa_j: K_0(kG) \to K_0(\PG)$  is generated by $\kappa_j(\St)$.
\item
The image of $$\ul \kappa: K_0(kG) \ \to \ K_0(\PG)^{\oplus p}$$  is generated by $\ul\kappa(\St)$ and, consequently, has rank 1.
\end{enumerate} 
\end{prop}

\begin{proof} As in the proof of Lemma~\ref{kappa_sl2}, we identify $K_0(\PG)$ with $K_0(\bP^1)$. 
We have a short exact sequence of bundles:

\begin{equation}
\label{exact_ker}
\xymatrix{
0\ar[r]&  \Ker\{ \wt\Theta, P_\lambda \otimes \cO_{\PG} \} \ar[r]& \Ker\{ \wt\Theta^j, P_\lambda \otimes \cO_{\PG} \} \ar[r]^-{\wt \Theta} & }
\end{equation}
$$\xymatrix{
\Ker\{ \wt\Theta^{j-1}, P_\lambda \otimes \cO_{\PG} \}(2) \ar[r] & 0 }
$$
Indeed, the composition is clearly zero and the first map is an embedding. Moreover, by Theorem~\ref{equiv}, the specialization of this sequence at any point $x \in \PG$ looks as follows:
$$\xymatrix{ 
0 \ar[r]& \Ker\{\theta_x, P_\lambda \otimes k(x) \} \ar[r]& \Ker\{\theta^j_x, P_\lambda \otimes k(x) \} \ar[r]^{\theta_x} & \Ker\{\theta^{j-1}_x, P_\lambda \otimes k(x) \} \ar[r]& 0
}
$$
The projectivity of $P_\lambda$ implies that this sequence is exact. Hence, (\ref{exact_ker}) is exact by Nakayama's lemma.  

We conclude that in $K_0(\bP^1)$, 
\begin{equation}
\label{recursion}\kappa_j(P_\lambda) = \kappa_1(P_\lambda) + \kappa_{j-1}(P_\lambda)(2).
\end{equation}
By Lemma~\ref{kappa_sl2}, $\kappa_1(P_\lambda) = 2\kappa_1(\St)$ for $ 0 \leq \lambda\leq p-2$.
Applying induction, we  get  $ \kappa_j(P_\lambda) = 2\kappa_1(\St) + 2\kappa_{j-1}(\St)(2)$. Now applying the formula (\ref{recursion}) to the Steinberg module, we get $2\kappa_1(\St) + 2\kappa_{j-1}(\St)(2) = 2\kappa_j(\St)$.
To summarize,
$$\kappa_j(P_\lambda) = 
\begin{cases}\kappa_j(\St) \quad \lambda =p-1 \\
2\kappa_j(\St)  \quad 0 \leq \lambda \leq p-2
\end{cases}
$$ 
which implies (1). Moreover, we conclude that $\ul \kappa: K_0(kG) \to K_0(\bP^1)^{\oplus p}$ is given by the formula
$$\ul \kappa(P_\lambda) = \begin{cases} 
(\kappa_1(\St), \kappa_2(\St), \ldots, \kappa_{p}(\St))   \quad \lambda=p-1\\ 
(2\kappa_1(\St), 2\kappa_1(\St), \ldots, 2\kappa_{p}(\St)) = 2 \ul \kappa(\St) \quad 0 \leq \lambda \leq p-2
\end{cases} 
$$
which proves (2).
\end{proof} 

The Proposition~\ref{prod_rank2} gives us some information about the behavior of $\kappa$ 
and $\kappa^{\oplus}$ with respect to products.  First, we need a trivial linear algebra lemma.

\begin{lemma}
\label{linear}
Let $V$, $W$  be vector spaces over a field $k$.  Let $\{ v_1, \ldots, v_r \}$ be a basis  of $V$, 
and $\{w_1, \ldots, w_s\}$ be a basis of $W$.  Let $v \in V$, $w \in W$, and  let    
$X \ \subset V\oplus W$ be the
 span of $\{ (v_1,w),\ldots,(v_r,w),(v,w_1),\ldots,(v,w_s) \}$. Then the dimension of $X$ 
 is at least  $r+s-1$.
\end{lemma}

\begin{proof} If $v=0$ and $w=0$ then we obviously have $\dim X = r+s$. 
We assume $v \not  = 0$. Observe that 
\begin{equation}
\label{projection}
\xymatrix{X \ar@{->>}[r] & (V\oplus W)/\langle (v,0)+(0,w)\rangle}
\end{equation} 
is surjective (thus has image of dimension $r+s-2$).  Let $v = a_1v_1 + \ldots + a_rv_r$. 
Then the vector $a_1(v_1,w) + \ldots + a_r(v_r,w) = (v,w)$ is non-trivial since $v \not = 0$ 
and belongs to the kernel of the projection (\ref{projection}). Hence, $\dim X \geq r+s-1$.
\end{proof}

\begin{prop}
\label{prod_rank2}
Let $G_1$ and $G_2$ be infinitesimal group schemes, and let $G = G_1\times G_2$.    Then 
$$
\rk \kappa^\oplus_G \geq \rk \kappa^\oplus_{G_1} + \rk \kappa^\oplus_{G_2} -1,
$$ 
where $\kappa^\oplus = \kappa^\oplus_j$ for any $j$, $1 \leq j\leq p-1$.
\end{prop}

\begin{proof} Let $i^*_\ell:K^\oplus_0(\PG) \to K^\oplus_0(\bP(G_\ell))$ be the map induced by the pull-back of vector bundles along the embedding $\bP(G_\ell) \hookrightarrow \PG$, for $\ell =1,2$.  We consider the composition 
$$\xymatrix{K_0(kG) \ar[r]^{\kappa^\oplus}& K_0^\oplus(\PG) \ar[rr]^-{i_1^* + i_2^*} && K^\oplus_0(\bP(G_1)) \oplus K^\oplus_0(\bP(G_2))}.$$

Let $P$ be a projective $kG_1$-module of dimension $m$, and $Q$ be a projective $kG_2$-module of dimension $n$. Then $P \otimes Q$ is a projective $kG$-module. By Corollary~\ref{ext-prod}, 
$$
(i_1^* + i_2^*)\circ \kappa^\oplus_G (P \otimes Q) = (i_1^* \circ \kappa^\oplus_G (P\otimes Q), i_2^* \circ \kappa^\oplus_G (P\otimes Q))  = 
$$
\begin{equation}
\label{decomp}
(n\, \kappa^\oplus_{G_1}(P), m\, \kappa^\oplus_{G_2}(Q)).
\end{equation} 

Let $\{P_1, \ldots, P_r \}$ be projective $kG_1$-modules of dimensions $\{p_1, \ldots, p_r\}$ such that 
$$
\{\kappa^\oplus_{G_1}(P_1), \ldots, \kappa^\oplus_{G_1}(P_r)\} 
$$  are linearly independent generators of $\Im \kappa^\oplus_{G_1} \subset K^\oplus_0(\bP(G_1))$, so that $\rk \kappa_{G_1}^\oplus = r$.  
Similarly, let   
$\{Q_1, \ldots, Q_r \} $ be projective $kG_2$-modules such that
$$\{\kappa^\oplus_{G_2}(Q_1), \ldots, \kappa^\oplus_{G_2}(Q_s)\} $$  are linearly independent generators of $\Im \kappa^\oplus_{G_2} \subset K^\oplus_0(\bP(G_2))$, so that $\rk \kappa_{G_2}^\oplus = s$.  
Finally,  let $P$ be any projective $kG_1$-module, $m =\dim_k P$, and $Q$ be any projective $kG_2$-module, $n = \dim_k Q$.  
Consider
$$S = \Span \{P_1 \otimes Q, P_2 \otimes Q, \ldots, P_r \otimes Q, 
P \otimes Q_1, P \otimes Q_2, \ldots,  P \otimes Q_s \}  \subset K_0(kG)
$$
Then the image of $(i_1^* + i_2^*) \circ \kappa_G$ contains 
$$\{ (n\, \kappa^\oplus_{G_1}(P_1), p_{_1}\kappa^\oplus_{G_2}(Q)), (n\, \kappa^\oplus_{G_1}(P_{_2}), p_{_2}\kappa^\oplus_{G_2}(Q)), \ldots, (n\,\kappa^\oplus_{G_1}(P_r), p_r\kappa^\oplus_{G_2}(Q)) \}$$
and 
$$\{ (q_{_1} \kappa^\oplus_{G_1}(P), m\, \kappa^\oplus_{G_2}(Q_1)), (q_{_2}\kappa^\oplus_{G_1}(P), m\, \kappa^\oplus_{G_2}(Q_2)), \ldots, (q_s\kappa^\oplus_{G_1}(P), m\, \kappa^\oplus_{G_2}(Q_s)) \}$$
Since all the coefficients $m, n, p_i, q_j$ are positive, Lemma~\ref{linear} implies that the dimension of the image of $(i_1^* + i_2^*) \circ \kappa^\oplus_G$ is at least $r+s-1$.

\end{proof}

\begin{remark}
As the reader can easily check,
 Proposition  \ref{prod_rank2} and its proof hold for $ \kappa: K_0(kG) \to K_0(\bP(G))$ in
 place of $\kappa^{\oplus} : K_0(kG)  \to K_0^\oplus(\bP(G))$. 
 \end{remark}

Observe that $\bP(\ul{ sl_2^{\oplus r}}) \simeq \bP^{2r-1}$, the join of $r$ copies of 
$\bP(\ul sl_2) \simeq \bP^1$.

\begin{cor}
\label{times_sl2}
Let $G=\ul{sl}_2^{\times r}$ be the infinitesimal group scheme corresponding to the restricted Lie algebra $sl_2^{\oplus r}$.  Then
$$\kappa^\oplus_1: K_0(kG) \to K_0^\oplus(\bP^{2r-1})$$
has rank at least $r(p-1)+1$. 
\end{cor}
\begin{proof}
For $r=1$, $\kappa^\oplus_{\ul{sl}_2}: K_0(u(sl_2)) \to K^\oplus_0(\bP^1)$ is injective by Proposition~\ref{pim} and, hence, has rank $p$. 
The statement now follows by induction and Proposition~\ref{prod_rank2}.
\end{proof}

Recall that if $H \hookrightarrow G$ is a subgroup scheme of a finite group scheme 
$G$, then $kG$ is free as a $kH$--module \cite[2.4]{OS}.  Hence, the restriction  
functor $\res: (kG-{\rm mod}) \to (kH-{\rm mod})$  
induces a well-defined map on $K$-groups: 
$\res^*: K_0(kG) \to K_0(kH)$. The commutativity of the diagram (\ref{commutative}) implies 
that restriction commutes with $\kappa_j$. 
That is, we have a commutative diagram: 
\begin{equation}
\label{kappa_diag}
\UseComputerModernTips
 \xymatrix{ K_0(kG) \ar[d]^{\res^*} \ar[r]^-{\kappa_{G,j}} \quad  & 
 K_0(\bP(G)) \ar[d]^{i^*} \\
 K_0(kH) \ar[r]^-{\kappa_{H,j}} \quad  & K_0(\bP(H)) 
, }
 \end{equation} 
where $i: \bP(H) \hookrightarrow \PG$ 
 is the closed embedding of projective varieties  induced by the embedding of group schemes and $i^*: K_0(\bP(G))
 \to K_0(\bP(H))$ is the pull-back via $i$ of locally free $\cO_{P(G)}$-modules. 
An analogous commutative diagram  holds for $\kappa^{\oplus}$.

\begin{prop}
\label{r-times}
Let $G$ be an infinitesimal group scheme, and let $H \simeq \ul sl_2^{\times r} \subset G$
be a closed subgroup scheme with the property that $\res^*: K_0(kG) \to K_0(kH)$ is
rationally surjective.  
Then the composition 
$$\xymatrix{K_0(kG) \ar[r]^{\kappa_1^\oplus} & K_0^{\oplus}(\PG) \ar[r]^-{i^*} & K^\oplus_0(\bP(H)) \simeq K_0^\oplus(\bP^{2r-1})}$$
has rank at least $r(p-1)+1$.
\end{prop}

\begin{proof}  We apply the diagram (\ref{kappa_diag}) for $\kappa^\oplus_1$: 
$$
\UseComputerModernTips
 \xymatrix{ K_0(kG) \ar[d]^{\res^*} \ar[r]^-{\kappa^\oplus_{G,1}} \quad  & 
 K_0^\oplus(\bP(G)) \ar[d]^{i^*} \\
 K_0(kH) \ar[r]^-{\kappa^\oplus_{H,1}} \quad  & K_0^\oplus(\bP(H)) 
, }
$$
Since $\res^*$ is assumed to be rationally surjective,  $\rk \left(i^* \circ \kappa^\oplus_{G,1} \right) = \rk \kappa^\oplus_{H,1}$. Since $H = \ul{sl}_2^{\times r}$, $ \rk \kappa^\oplus_{H,1} \geq r(p-1)+1$ by Corollary~\ref{times_sl2}.  This proves the statement. 
\end{proof}

\vspace{0.1in}
For the rest of this computational section, we calculate some examples of bundles for $E = \bG_{a(1)} \times  \bG_{a(1)}$.  We have 
$kE \simeq k[x,y]/(x^p, y^p)$. 
Let $X_n$  be  $(2n+1)$--dimensional 
``zig-zag" module.  Pictorially,  we  represent  $X_n$  
by the following  diagram: 

$$
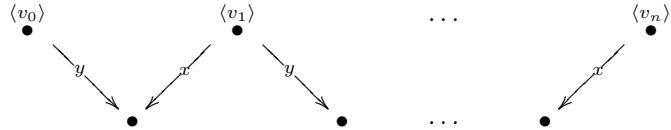
\begin{xy}*!C\xybox{%
\xymatrix{ \stackrel{\langle v_0 \rangle }{\bu} 
\ar[dr]|y &&
\stackrel{\langle v_1 \rangle }{\bu} \ar[dl]|x \ar[dr]|y
&&\dots &&\stackrel{\langle v_n \rangle }{\bu} \ar[dl]|x\\
& \bu && \bu &\dots &  \bu &}}
\end{xy}
$$
It is straightforward to check that $X_n$ has constant  Jordan type $n[2] +  [1]$  
(see \cite[\S 2]{CFP}).  We proceed to prove that for any integer $m$ we can obtain 
the line bundle $\cO_{\bP^1}(m)$ on $\bP(E) = \bP^1$ by applying our constructions 
to some $X_n$ or its linear dual $X_n^\#$.  

Note that for $X_n$, the  map
$$
\xymatrix{
\wt\Theta_E:    X_n \otimes \cO_{\bP^1}   \ar[r] & X_n \otimes \cO_{\bP^1}(1)  
}
$$ as defined in \ref{theta-tilde} has nilpotentcy degree $2$. Hence, 
there  is an inclusion $\Im \{\wt\Theta_E, X_n \otimes \cO_{\bP^1}   \}  
\subset \Ker \{ \wt\Theta_E,     X_n \otimes\cO_{\bP^1}(1)   \}$. 
We, therefore, may define a subquotient sheaf of the free sheaf $X_n\otimes   \cO_{\bP^1}  $ as
$$\cX_n : =   \Ker \{ \wt\Theta_E,    X_n \otimes \cO_{\bP^1} \} / 
\Im \{\wt\Theta_E(-1), X_n \otimes \cO_{\bP^1}(-1)   \}.$$ 
Arguing as in the  proof of Proposition \ref{bundle2},  one verifies that $\cX_n$ is  
locally free with the fiber at a point $t \in \bP^1$  isomorphic to the 1-dimensional
vector space
 $\frac{\Ker  \{\theta_t: X_{n,k(t)} \to X_{n,k(t)} \}}{\Im \{\theta_t: X_{n,k(t)} 
\to X_{n,k(t)} \}}$.    Hence, $\cX_n$  is a   line bundle. 
The linear dual $X^\#_n$ of $X_n$ is represented by the diagram:
$$
\begin{xy}*!C\xybox{%
\xymatrix{ &{\bu} 
\ar[dl]|x \ar[dr]|y &&
{\bu} \ar[dl]|x 
&\dots &{\bu}  \ar[dr]|y\\
\stackrel{\bu}{\langle  w_0 \rangle } && \stackrel{\bu}{\langle  w_1 \rangle }  && \dots && \stackrel{\bu}{\langle w_n \rangle }}}.
\end{xy}
$$
Define  a subquotient sheaf of $   X^\#_n \otimes \cO_{\bP^1}  $ as
$$\cY_n : = \Ker \{ \wt\Theta_E,       X^\#_n \otimes \cO_{\bP^1}   \} / \Im \{\wt\Theta_E(-1),   X^\#_n  \otimes \cO_{\bP^1}(-1)  \}.$$

\begin{prop}
\label{zigzag}
$\cX_n \simeq \cO_{\bP^1} (-n)$, $\cY_n \simeq \cO_{\bP^1} (n)$.  
\end{prop}

 \begin{proof} 
 Let $k[s,t] = k[\mathbb A^2] \simeq k[V(E)]$. The   universal $p$-nilpotent operator $\Theta_E  \in  k[x,y]/(x^p, y^p) \otimes k[s,t]$  is given by 
 $$\Theta_E = xs+yt,$$
(see, for example, Ex. \ref{exp-elem}). We identify the graded $k[s,t]$-module 
 $\Ker \{\Theta_E,       X_n \otimes k[s,t] \} / \Im \{\Theta_E,       X_n  \otimes  k[s,t] \}$, 
thereby  determining the vector bundle  $\cX_n$.
  It is easy to see that  $\Im  \{\Theta_E,    X_n \otimes  k[s,t] \}$  is generated by  the bottom row of the diagram representing $X_n$   
as a $k[s,t]$--module and  $\Ker \{\Theta_E,       X_n \otimes  k[s,t] \}$ is generated by the same bottom row and the vector  
$s^nv_0 + s^{n-1} t v_1 + \cdots t^{n} v_n$.
Hence, $\Ker \{\Theta_E,     X_n \otimes  k[s,t]\} / 
\Im \{\Theta_E,    X_n  \otimes  k[s,t]\}$  is generated by 
$s^nv_0 + s^{n-1} t v_1 + \cdots t^{n} v_n$ as 
a $k[s,t]$-module.   Since the generator is in degree $n$, we conclude that  the corresponding     
 locally free sheaf  of rank $1$  is $\cO_{\bP^1}(-n)$. 

We now  compute  $\cY_n$. The graded $k[s,t]$-module $\Ker \{\Theta_E,    X^\#_n \otimes  k[s,t]\}$   is 
generated by $\langle w_0, \ldots, w_n \rangle$  in degree 0, 
and $\Im \{\Theta_E,    X^\#_n \otimes  k[s,t](-1)\}$  is generated by  
$\langle sw_0 + tw_1, sw_1+tw_2, \ldots, sw_{n-1} + tw_n \rangle $, also  in degree $0$. 
Hence,  on $U_0 = \bP^1 - Z(s)$,  the restriction of $\cY_n$ is generated by $w_n$, with $w_0 = (-\frac{t}{s})^{n} w_n$.  We map  
$\cY_n(U_0)$ to $K = k(t/s)$, the residue field at the generic   point  of $\bP^1$,  by sending $w_n$ to $1$.  The image of $w_0$ is  $(-\frac{t}{s})^{n}$. 
On the  other affine piece, $U_1 = \bP^1 - Z(t)$, the restriction  of $\cY_n$ is generated by $w_0$, with the relation $w_n = (-\frac{s}{t})^{n}w_0$.  
We map this to $K= k(s/t)$ by sending $w_0$ to $(-\frac{t}{s})^{n}$. 
Hence,  the vector bundle  is given by the Cartier divisor 
$(U_0, 1), (U_1, (-\frac{t}{s})^{n})$. This divisor  is equivalent to  the Cartier divisor $(U_0, 1), (U_1, (\frac{t}{s})^{n})$  which correspond to 
the line bundle $\cO_{\bP^1}(n)$.  Hence, $\cY_n \simeq \cO_{\bP^1}(n)$.   
 \end{proof}

In the next Proposition we  calculate explicitly the line bundles corresponding to the syzygies of the trivial  modules, $\Omega^n k$.  
For convenience,  we use the notation $\cH^{[1]}(M)$ for the  bundle $\cM^{[1]}$ associated  to $M$  as defined  in (\ref{bracket}). 

\begin{prop}   Let $E = \bG_{a(1)}^{\times r}$. Then   
$$
\cH^{[1]}(\Omega^n k) \simeq 
\begin{cases}
\cO_{\bP^{r-1}}(-\frac{np}{2}) & \text{if $n$  is even,}\\
\cO_{\bP^{r-1}}(-\frac{n+1}{2}p + 1) & \text{if $n$  is odd} \\
\end{cases}
$$
for $p$  odd and 
$$
\cH^{[1]}(\Omega^n k) \simeq \cO_{\bP^{r-1}}(-n)
$$ 
for $p=2$. 
\end{prop}

\begin{proof}  
Let $r=2$, and assume $n \geq 0$.   As  in the proof of Prop. \ref{zigzag}, the universal operator $\Theta_E = sx+ty$ where $k[V(E)] = k[s,t]$. 
The structure of a
minimal $kE \simeq k[x,y]/(x^p, y^p)$-projective resolution $P_\bu \to k$ is well known \cite{CTVZ}, 
with $P_{n-1} = kG^{\times n}$. A set of generators $a_1, \ldots, 
a_n$ for $P_{n-1}$ can be chosen so that $\Omega^n(k)$ 
is the submodule generated by the elements 
$$
x^{p-1}a_1, \ \ ya_1 -xa_2, \ \ y^{p-1}a_2- x^{p-1}a_3, \ \ 
ya_3-xa_4, \ \ \ldots, \ \ ya_{n-1} -xa_n, \ \ y^{p-1}a_n$$
for  $n$  even, and 
$$xa_1, \ \ ya_1-x^{p-1}a_2, \ \ y^{p-1}a_2 - xa_3, \ \ ya_3  - x^{p-1}a_4, \ \ \ldots, \ \ y^{p-1}a_{n-1} - xa_n, \ \ ya_n $$
for $n$ odd. 

Let $n$  be even.  For  illustrational purposes,  we  include a  picture of  $\Omega^4 k$ for $p=3$,
\vspace{0.2in}

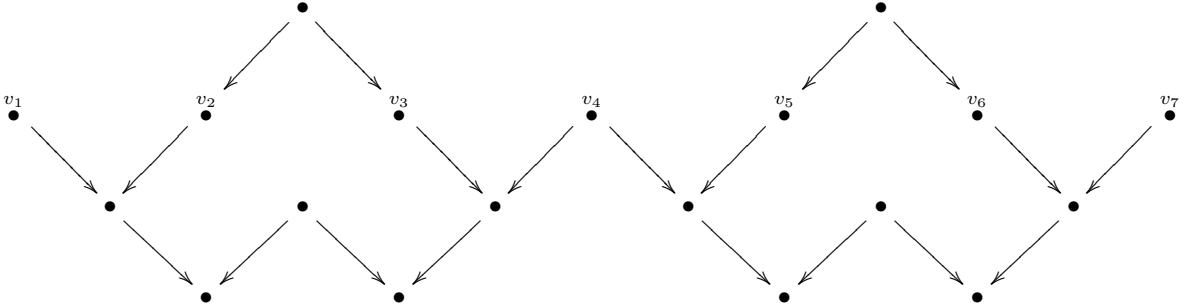
\begin{changemargin}{-1.5cm}{-1cm} 
\noindent
$
\begin{xy}*!C\xybox{%
\xymatrix{ &&& \bu \ar[dr]\ar[dl] &&&&&& \bu \ar[dr]\ar[dl] &&& \\
\stackrel{v_1}{\bu} \ar[dr] && \stackrel{v_2}{\bu} \ar[dl] && \stackrel{v_3}{\bu} \ar[dr] && \stackrel{v_4}{\bu} \ar[dl]\ar[dr] && \stackrel{v_5}{\bu} 
\ar[dl] && \stackrel{v_6}{\bu} \ar[dr] && \stackrel{v_7}{\bu} \ar[dl]\\
&\bu \ar[dr] && \bu\ar[dr]\ar[dl] && \bu\ar[dl] &&\bu \ar[dr] && \bu\ar[dr]\ar[dl] && \bu \ar[dl] & \\
&& \bu && \bu &&&& \bu && \bu && \\ 
}}
\end{xy}
$
\end{changemargin}
\vspace{0.1in}

The kernel  of $\Theta_E = sx+ty$  on $ \Omega^n k \otimes  k[s,t]$  is a submodule of a free  $k[s,t]$-module  generated by the ``middle layer"  of 
$\Omega^n k$, that  is, 
by $v_1 = x^{p-1}a_1, v_2 = x^{p-2}(ya_1 -xa_2), v_3 = x^{p-3}y(ya_1 -xa_2), \ldots, v_p = y^{p-2}(ya_1 -xa_2), v_{p+1} = y^{p-1}a_2- x^{p-1}a_3, 
\ldots, v_{\frac{np}{2}+1} = y^{p-1}a_n$.
Everything below the middle layer which  is in $\Ker \Theta_E$  also  lies in $\Im \Theta^{p-1}_E$.  One verifies that the quotient 
$$\Ker \{\Theta_E: \Omega^n k \otimes  k[s,t] \to   \Omega^n k \otimes  k[s,t] \} / \Im \{\Theta^{p-1}_E:  \Omega^n k \otimes  k[s,t] \to 
\Omega^n k \otimes  k[s,t]\}$$  is generated by 
$$
s^{\frac{np}{2}} v_1 - s^{\frac{np}{2}-1}t v_2 + \ldots \pm t^{\frac{np}{2}}v_{\frac{np}{2}+1}.
$$      Arguing as   in the proof of Prop. \ref{zigzag},  we conclude 
that the corresponding  locally  free sheaf of rank 1 is $\cO_{P^1}(-\frac{np}{2})$.  
The calculation for odd   positive $n$ is similar.     For negative values of $n$,   one  verifies the formula  by doing again a similar calculation with dual modules.  

Now let $r>2$,  let $i: F \subset E$    be a subgroup scheme isomorphic to $\bG_{a(1)}^{\times 2}$, and let $f: \bP(E) \to \bP(F)$  be the map  induced by the embedding $i$. 
Since  $(\Omega^n_E k) \downarrow_F  \simeq  \Omega^n_F k \oplus \rm proj$,  we conclude  that 
$\cH^{[1]}((\Omega_E^n k)\downarrow_F)  \simeq \cH^{[1]}(\Omega_F^n k) $.   
Proposition  \ref{funct}   implies   an isomorphism  $f^* (\cH^{[1]}(\Omega_E^n k)) \simeq  \cH^{[1]}((\Omega_E^n k)\downarrow_F)$.  Hence, 
$$f^* (\cH^{[1]}(\Omega_E^n k)) \simeq \cH^{[1]}(\Omega_F^n k).$$

The proposition now follows from the observation that  
$$\xymatrix{f:  \bP(F) \simeq  \bP^1 \ar[r] & \bP(E) \simeq \bP^{r - 1}}$$  induces an isomorphism  on Picard groups   via $f^*$.
\end{proof}


\end{document}